\documentclass[12pt]{article}

\AtBeginDocument{}

\usepackage{amsmath}
\usepackage{amssymb}
\usepackage{titling}
\usepackage[margin= 1in]{geometry}
\usepackage{graphicx}
\usepackage{mathtools}
\usepackage{graphicx}
\usepackage{xcolor}
\usepackage{amsthm}
\usepackage[utf8]{inputenc}
\usepackage[english]{babel}
\usepackage{dsfont}
\usepackage[all,cmtip]{xy}
\usepackage{subcaption}
\usepackage{float}
\usepackage{framed,enumitem}
\usepackage[nodisplayskipstretch]{setspace}
\usepackage{array}
\usepackage{tabularray}
\usepackage[export]{adjustbox}
\usepackage{multirow}
\usepackage{pdfpages}
\usepackage{comment}

\newtheorem{theorem}{Theorem}[section]

\newtheorem{lemma}[theorem]{Lemma}
\newtheorem{proposition}[theorem]{Proposition}
\theoremstyle{definition}
\newtheorem{remark}{Remark}[]

\begin{document}

\title{Skeletal Snub Polyhedra in Ordinary Space, I} 
 
\author{
Egon Schulte
\thanks{Supported by the Simons Foundation Award No. 420718}\\
Northeastern University\\
Boston, Massachusetts,  USA, 02115
\and and \\[.05in]
Tom\'a\v{s} Sk\'acel\\
Northeastern University\\
Boston, Massachusetts,  USA, 02115}

\date{ \today } 
\maketitle 

\begin{center}
\textit{In memory of Chris King, our friend, colleague, and teacher.}
\end{center}
\bigskip

\begin{abstract}
Skeletal polyhedra are discrete connected structures consisting of finite (planar or skew) or infinite (linear, planar, or spatial) polygons as faces, with two faces on each edge and a circular vertex figure at each vertex. The present paper describes the blueprint for the snub construction and shows that it can be applied to both regular and chiral skeletal polyhedra in ordinary space. The resulting skeletal snub polyhedra are vertex-transitive and highly locally symmetric. Their properties - from a combinatorial, topological, and geometric perspective - are described and illustrated on some particularly interesting examples. We examine when the construction yields uniform skeletal polyhedra and discuss the completeness of our list of generated structures.   
\bigskip\medskip

\medskip

\end{abstract}




\section{Introduction}
\label{intro}

The exploration of polyhedra in ordinary Euclidean 3-space has a long and varied history, with roots dating back to ancient Greece. Throughout the centuries, the concept of what constitutes a polyhedron has evolved, giving rise to new classes of highly symmetrical structures, such as the five Platonic solids, the Kepler-Poinsot polyhedra, the Petrie-Coxeter polyhedra, and the Grünbaum-Dress polyhedra (see \cite{5,7,18,19,22}). The various classes of polyhedra mentioned can be seen as representing the development of the concept of a polyhedron, starting with the examination of convex solids, progressing to a topological and algebraic approach that analyzed polyhedra as maps on surfaces, and culminating in more recent graph-theoretical approaches that emphasize the combinatorial incidence structure of skeletal figures in three-dimensional space.\smallskip

The theory of skeletal polyhedra received an important impetus from Grünbaum when he relaxed the constraint that membranes must span the faces of a polyhedron and allowed skew and even infinite polygons to occur as faces (see \cite{22}). This new approach gave rise to the Grünbaum-Dress polyhedra, pushing the number of regular polyhedra to 48. The enumeration of regular skeletal polyhedra was presented by Grünbaum \cite{22} and Dress \cite{18,19}. This classification received more attention later when McMullen and Schulte presented a simpler approach to the classification \cite{40,41}. Naturally, the study of highly symmetric non-regular skeletal polyhedra followed, and the chiral skeletal polyhedra were classified in \cite{Chiral 1, Chiral 2}. Further work related to geometric chiral polyhedra can be found in Pellicer and Weiss \cite{47}. More recently, certain classes of skeletal polyhedra with few flag orbits have been classified: the finite polyhedra with 3 flag orbits were enumerated in Cunningham and Pellicer \cite{CuPe,CuPe2}, and certain types of polyhedra with 2 flag orbits (other than chiral polyhedra) were classified by Pellicer and Williams \cite{PeWi} (see also the related work in Hubard \cite{Hu}, and Cutler and Schulte \cite{14}).
\smallskip

The present paper and its companion \cite{Sk2} (as well as \cite{Sk1}) by the second author were inspired by the desire to achieve a more profound understanding of the uniform skeletal polyhedra in ordinary Euclidean 3-space, through building on the line of investigation begun in Schulte and Williams~\cite{Wythoffians} and in~\cite{Abby Thesis}. These skeletal polyhedra, by definition, have a vertex-transitive symmetry group and regular faces, and serve as skeletal equivalents to the Archimedean solids (and prisms and antiprisms). Coxeter, Longuet-Higgins, and Miller published a comprehensive list of finite, convex or non-convex, uniform polyhedra with planar (convex or star-polygon) faces in \cite{11}, but the completeness of this list was only proved later, independently, by Skilling \cite{55} and Har'El \cite{30}. Uniform skeletal polyhedra have not been classified, in fact, not even the finite uniform skeletal polyhedra have been classified.  \smallskip

Our approach exploits a ``skeletal" variant of Wythoff's construction (see \cite{7,41}) and generates highly locally symmetric ``snub" polyhedra, skeletal variants of the classical snub polyhedra, from orbits of initial points (initial vertices) under the subgroup of the symmetry groups of regular or chiral polyhedra consisting of all  combinatorial rotations (viewed as geometric symmetries). The method generates a large collection of novel snub polyhedra from the 48 regular polyhedra and from the 6 infinite families of chiral polyhedra.  These newly generated objects are often uniform skeletal polyhedra, that is, in addition to a vertex-transitive symmetry group they also have regular faces. Famous examples are the snub cube and snub dodecahedron. Our work contributes considerably to the growing incomplete list of uniform skeletal polyhedra, adding further to the list of \cite{Wythoffians,Abby Thesis}. \smallskip

The paper is organized as follows. In Sections~\ref{skelpo} and \ref{abpo}, we begin with a review of basic concepts about skeletal polyhedra and abstract polyhedra. Section~\ref{regpo} revisits the classification of skeletal regular polyhedra in 3-space. Then, Section~\ref{snub} describes in detail the main construction and explores the geometric, combinatorial, and topological properties of the generated structures. In Section~\ref{completenessarguments}, we investigate a key question, a converse, relevant for a completeness proof of uniform polyhedra of snub type: which polyhedra of snub type, $p.3.3.q.3$, are snub polyhedra of regular or chiral polyhedra? We provide strong partial answers to this question. Finally, in Sections~\ref{secgplus}, \ref{finite uniform polyhedra}, and \ref{new finite uniform polyhedra}, we illustrate examples of uniform and non-uniform snub polyhedra derived from the 18 finite regular polyhedra. A full analysis of the snub polyhedra associated with all 48 regular polyhedra is carried out in the companion paper \cite{Sk2} and produces a wealth of new uniform examples (see also \cite{Sk1}). The snub polyhedra derived from chiral polyhedra have not yet been fully investigated. 
\smallskip

\section{Skeletal Polyhedra}
\label{skelpo}

In this section, we give a brief overview of geometric polyhedra in $\mathbb{E}^3$ as described in \cite{22} and \cite[Ch.~7]{41}. We will use the term geometric and skeletal interchangeably. \smallskip

Informally speaking, a geometric polyhedron consists of a set of vertices, edges, and faces, all fitting together in a way that is consistent with the combinatorial structure of convex polyhedra. A vertex of a geometric polyhedron is simply a point in space. An edge of a geometric polyhedron, denoted by $\{u,v\}$, is the closed line segment between two distinct vertices $u$ and $v$. A face of a geometric polyhedron is either a finite or infinite polygon, which we define below. \smallskip

A \textit{finite polygon}, or simply an \textit{$n$-gon}, $(v_1,v_2,...,v_n)$ in ordinary Euclidean space $\mathbb{E}^3$ is a figure formed by distinct points $v_1,...,v_n$, together with the line segments $\{v_i,v_{i+1}\}$ for $i=1,\dots,n-1$ and $\{v_n,v_1\}$. Similarly, an \textit{infinite polygon}, or \textit{apeirogon}, consists of an infinite sequence of distinct points $(\dots,v_{-2},v_{-1},v_0,v_1,v_2,\dots)$ and of the line segments $\{v_i,v_{i+1}\}$ for each $i$, such that each compact subset of $\mathbb{E}^3$ meets only finitely many line segments. In either case the points are the \textit{vertices} and the line segments are the \textit{edges} of the polygon. \smallskip

We say that a polygon is \textit{geometrically regular} if its symmetry group acts transitively on the flags of this polygon. By a \textit{flag} of the polygon we mean a 2-element set consisting of a vertex and an incident edge. The planar regular polygons comprise the familiar regular convex and star-polygons, as well as the regular zigzags and linear apeirogons. The spatial, non-planar  regular polygons comprise the finite skew (prismatic or anti-prismatic) regular polygons and the helical regular polygons. For more on regular polygons, see \cite{7,22}.  \smallskip

General polygons in $\mathbb{E}^3$ can of course be considerably more complicated than regular polygons. However, apart from triangles and some quadrangles and pentagons, all polygons in this paper are regular polygons.\smallskip 

A \textit{geometric} (or \textit{skeletal}) \textit{polyhedron}, or simply \textit{polyhedron} (if the context is clear), $P$ in $\mathbb{E}^3$ consists of a set of distinct points, called \textit{vertices}, a set of line segments, called \textit{edges}, and a set of polygons, called \textit{faces}, such that the following properties are satisfied: 
\noindent
\begin{itemize}
\item[(a)] The graph defined by the vertices and edges of $P$, called the \textit{edge graph} of $P$, is connected.
    
\item[(b)] The vertex figure of $P$ at each vertex of $P$ is connected. By the \textit{vertex figure} of $P$ at a vertex $v$ we mean the graph whose vertices are the neighbors of $v$ in the edge graph of $P$ and whose edges are the line segments $\{u,w\}$, where $\{u,v\}$ and $\{v,w\}$ are adjacent edges of a common face of $P$. 
    
\item[(c)] Each edge of $P$ is contained in exactly two faces of $P$. 
    
\item[(d)] $P$ is \textit{discrete}, meaning that each compact subset of $\mathbb{E}^3$ meets only finitely many faces of $P$. 
\end{itemize}

If $P$ is a geometric polyhedron, then the set of all vertices, edges, and faces of $P$ is a partially ordered set, where the partial order is induced by incidence. This is an abstract polyhedron (if suitable ``improper" elements of ranks $-1$ and $3$ are added) and often is a lattice.
\smallskip

A \textit{flag} of a geometric polyhedron $P$ is a 3-element set containing a vertex, an edge, and a face of $P$, all mutually incident. We say that two flags of $P$ are \textit{adjacent} if they differ in precisely one element. \smallskip

We call a geometric polyhedron $P$ in $\mathbb{E}^3$ (geometrically) \textit{regular} if its symmetry group $G(P)$ acts transitively on the flags of $P$. Thus, if $P$ is regular, then $G(P)$ is transitive, separately, on the vertices, edges, and faces of $P$. Furthermore, the faces and vertex figures of $P$ must be regular polygons. The faces of $P$ can take finite (convex, star, or skew) or infinite (zig-zag or helical) forms. Linear apeirogons do not occur as faces of geometrically regular polyhedra. We say that $P$ is of type $\{p,q\}$ if the faces of $P$ are $p$-gons and the vertex figures are $q$-gons, allowing $p=\infty$; note that $q$ is finite, by the discreteness condition~(d) above.\smallskip

Suppose $P$ is a (geometric) regular polyhedron, and let $\Phi := \{F_{0},F_{1},F_{2}\}$ be a {\em base\/} flag of $P$. Then $G(P)$ is generated by {\em distinguished generators\/} $r_{0},r_{1},r_{2}$ ({\em with respect to\/}~$\Phi$), where $r_{j}$ is the unique symmetry of $P$ which fixes all elements of $\Phi$ but the vertex if $j=0$, the edge if $j=1$, or the face if $j=2$. These generators satisfy the standard Coxeter-type relations for a string Coxeter diagram with three nodes,
\begin{equation} 
\label{symgreg}
r_{0}^2 = r_{1}^2 = r_{2}^2 = 
(r_{0}r_{1})^{p} = (r_{1}r_{2})^{q} =
(r_{0}r_{2})^{2} = 1,
\end{equation}
where $\{p,q\}$ is the type of $P$ (when $p=\infty$ the corresponding relation is omitted); in general there are other independent relations too. As involutory isometries, the generators $r_j$ must be reflections in points, lines, or planes. Note that, in a natural way, the symmetry group of the face $F_2$ in $\Phi$ is $\langle r_{0},r_{1}\rangle$, while that of the vertex figure  at $F_0$ is $\langle r_1,r_2\rangle$. 
\smallskip

We define three special symmetries of $P$, 
\[s_1:=r_{0}r_{1},\, s_2:= r_{1}r_2,\, s_0:= r_{0}r_2 =s_{1}s_{2},\] 
and set
\[G^{+}(P):=\langle s_1,s_2 \rangle.\]
Then observe that the relations in (\ref{symgreg}) imply 
$$s_1^p = s_2^q = (s_1 s_2)^2 = 1.$$ 
We refer to the subgroup $G^{+}(P)$ of $G(P)$ as the \textit{combinatorial rotation subgroup} of $G(P)$, and note that $G^{+}(P)$ has index at most 2 in $G(P)$ and can be generated by any two of the \textit{distinguished combinatorial rotations} $s_1,s_2,s_0$. The generator $s_1$ cycles through (or shifts) the vertices of $F_2$, and $s_2$ cycles through the vertices of the vertex figure at $F_0$. However, it is important to keep in mind that the elements of $G^+(P)$, and thus $s_1,s_2$, need not be direct isometries of $\mathbb{E}^3$, that is, geometric rotations, translations, or screw rotations. We say that $P$ is {\it directly regular\/} if $G^+(P)$ has index 2 in $G(P)$. 

A geometric polyhedron $P$ in $\mathbb{E}^3$ is called (geometrically) \textit{chiral} if its symmetry group $G(P)$ has two orbits on the flags such that adjacent flags are in distinct orbits. A geometrically chiral polyhedron $P$ has regular polygons as faces and vertex figures, and its symmetry group $G(P)$ acts transitively, separately, on the vertices, edges, and faces. Both polyhedra with finite faces and polyhedra with infinite faces occur. The symmetry group $G(P)$ has a pair of generators $S_1,S_2$ (with respect to a base flag $\Phi$) with properties similar to the generators $s_1,s_2$ for the combinatorial rotation subgroup of a regular polyhedron, but in this case there are no symmetries corresponding to the reflections $r_j$. 
For a chiral polyhedron we set
\[G^{+}(P):=G(P)\]
and again define three special symmetries, 
\[s_1:=S_1,\, s_2:= S_2,\, s_0:=s_{1}s_{2}=S_1S_2.\] 
Thus the symmetry group of a chiral polyhedron coincides with its combinatorial rotation subgroup, $G^+(P)$. Changing the base flag to an adjacent flag of $\Phi$ results in new distinguished generators of $G(P)$ that are not conjugate in $G(P)$ to those determined by $\Phi$ itself, that is, $S_1,S_2$; the two sets of generators represent the two ``enantiomorphic" (mirror image) forms of a chiral polyhedron, in a sense, a left form and a right form. The chiral polyhedra in $\mathbb{E}^3$ were classified in \cite{Chiral 1,Chiral 2}.\smallskip

If a geometric polyhedron $P$ is \textit{vertex transitive}, meaning that its symmetry group $G(P)$ acts transitively on the vertex set of $P$, we can define the \textit{vertex symbol} of $P$ as the sequence of numbers of vertices of the polygonal faces surrounding, in cyclic order, each vertex of $P$, allowing $\infty$. This symbol is unique up to cyclic permutation and reversal of orientation. So, the vertex symbol $p.q.r.s$ means that each vertex of $P$ is surrounded by a $p$-gon, $q$-gon, $r$-gon, and $s$-gon, in that order. A vertex transitive  polyhedron is said to be of \textit{snub type}, or has a \textit{vertex symbol of snub type}, $p.3.3.q.3$, if each vertex is surrounded, in cyclic order and up to reversal of orientation, by a $p$-gon, two triangles, a $q$-gon, and another triangle; we allow the possibility that $p=3$, $q=3$, or $p=q$.
\smallskip

It is often useful to encode additional information about the geometric structure of the polygons present in the vertex symbol of a vertex transitive polyhedron $P$. If $p > 3$ is finite, then in the vertex symbol of $P$ we record a regular convex $p$-gon, a regular star $p$-gon (with density $d$), or a regular skew $p$-gon as $p_c$, $\frac{p}{d}$, or $p_s$, respectively. Triangles are convex, so no subscript on $p$ is needed when $p=3$. Similarly, if $p=\infty$, then in the vertex symbol of $P$ we denote a regular linear apeirogon, a regular zig-zag, or a regular helical polygon over a base $b$-gon by $\infty_1$, $\infty_2$, or $\infty_b$, respectively. In this paper, we frequently encounter triangles in the vertex symbol, both regular and non-regular. \smallskip 

Finally, we discuss geometric polyhedra which are skeletal analogs of the Archimedean solids and are neither regular nor chiral, but are still highly symmetric. We say that a geometric polyhedron $P$ in $\mathbb{E}^3$ is \textit{uniform} if the faces of $P$ are regular polygons and the symmetry group of $P$ acts transitively on the vertex set of $P$. All regular or chiral polyhedra in $\mathbb{E}^3$ are uniform. 
\smallskip

\section{Abstract Polyhedra}
\label{abpo}

In this section, we briefly review basic notions and results about abstract polyhedra following~\cite{41}. Topologically, abstract polyhedra are maps on surfaces (if they are locally finite) \cite{Conder}. Geometric polyhedra are often viewed as realizations of abstract polyhedra. 
\smallskip
 
An {\em abstract polyhedron\/} is a partially ordered set $\mathcal{P}$ with a strictly monotone {\em rank\/} function with range $\{-1,0,1,2,3\}$. Its elements are called faces. (As is common for abstract polytopes, in this section, the term ``face" is used in a more general sense than in the geometric context above where it meant ``face of rank 2".)  The elements of rank $j$ are the {\em $j$-faces\/} of~$\mathcal{P}$. For $j = 0$, $1$ or $2$, the $j$-faces are also called {\em vertices}, {\em edges\/} and {\em facets\/}, respectively. There are two \textit{improper} faces: a minimum face $F_{-1}$ (of rank $-1$) and a maximum face~$F_3$ (of rank $3$). The \textit{flags} (maximal totally ordered subsets) of $\mathcal{P}$ all contain one vertex, one edge and one facet, in addition to $F_{-1}$ and $F_3$ (which are often suppressed in listing the elements in a flag). Further, $\mathcal{P}$ is \textit{ strongly flag-connected}, meaning that any two flags $\Phi$ and $\Psi$ of $\mathcal{P}$ can be joined by a sequence of flags $\Phi = \Phi_{0},\Phi_{1},\ldots,\Phi_{k} =\Psi$, where $\Phi_{i-1}$ and $\Phi_{i}$ are {\em adjacent\/} (differ by one element), and $\Phi \cap\Psi \subseteq \Phi_{i}$ for each $i$. Finally, if $F$ and $G$ are a $(j-1)$-face and a $(j+1)$-face with $F < G$ and $0 \leq j \leq 2$, then there are exactly {\em two\/} $j$-faces $H$ such that $F < H < G$. As a consequence, for $0 \leq j \leq 2$, every flag $\Phi$ of $\mathcal{P}$ is adjacent to just one other flag, denoted $\Phi^j$ and called the \textit{$j$-adjacent} flag of $\Phi$, differing in the $j$-face.
\smallskip

When $F$ and $G$ are two elements of an abstract polyhedron with $F \leq G$, we call
$G/F := \{H \mid F \leq H \leq G\}$ 
a \textit{section} of $\mathcal{P}$. We usually identify a $j$-face $F$ with the section $F/F_{-1}$. If $F$ is a facet, then $F/F_{-1}$ (or simply, $F$) is isomorphic to the face lattice of a (finite) convex polygon or an (infinite) apeirogon. The same also holds for the \textit{vertex figure } $F_{3}/F$ at a vertex $F$ of $\mathcal{P}$.
\smallskip

If all facets of an abstract polyhedron $\mathcal{P}$ are $p$-gons for some $p$, and all vertex figures are $q$-gons for some $q$, then $\mathcal{P}$ is said to be of {\em type\/} $\{p,q\}$; here $p=\infty$ or $q=\infty$ are permitted. 
\smallskip

An abstract polyhedron $\mathcal{P}$ is said to be {\em regular\/} if its (combinatorial) {\em automorphism group\/} $\Gamma(\mathcal{P})$ is transitive on the flags of $\mathcal{P}$. Suppose $\mathcal{P}$ is an abstract regular polyhedron, and let $\Phi := \{F_{0},F_{1},F_{2}\}$ be a {\em base\/} flag of $\mathcal{P}$. Then $\Gamma(\mathcal{P})$ is generated by {\em distinguished generators\/} $\rho_{0},\rho_{1},\rho_{2}$ ({\em with respect to\/}~$\Phi$), where $\rho_{j}$ is the unique automorphism which fixes all faces of $\Phi$ but the $j$-face (and maps $\Phi$ to its $j$-\textit{adjacent} flag $\Phi^j$). As with geometric regular polyhedra we define three automorphisms, 
\[\sigma_1:=\rho_0\rho_1,\; \sigma_2:= \rho_1\rho_2,\; \sigma_0:= \rho_0\rho_2,\]
called the \textit{distinguished rotations} of $\mathcal{P}$, and observe that any two of these generate the \textit{rotation subgroup} $\Gamma^+(\mathcal{P})$ of $\Gamma(\mathcal{P})$, which has index at most $2$ in $\Gamma(\mathcal{P})$. We say that an abstract polyhedron $\mathcal{P}$ is \textit{directly regular} (or \textit{orientably regular}) if the index is 2. Note that a locally finite abstract regular polyhedron (with finite $p$, $q$) is directly regular if and only if it corresponds to a regular map on an orientable surface~\cite{Conder}. A directly regular geometric polyhedron is also a directly regular abstract polyhedron.
\smallskip

The generators $\rho_0,\rho_1,\rho_2$ of $\Gamma(\mathcal{P})$ and $\sigma_1,\sigma_2$ of $\Gamma^+(\mathcal{P})$, respectively, have properties very similar to those of the generators $r_0,r_1,r_2$ of $G(P)$ and $s_1,s_2$ of $G^+(P)$ of a geometric regular polyhedron $P$, and we will not repeat them all. In particular, $\sigma_1$ cycles through (or shifts) the vertices of $F_2$, and $\sigma_2$ cycles (or shifts) through the vertices of the vertex figure at $F_0$. 
\smallskip

An abstract polyhedron $\mathcal{P}$ is {\em chiral\/} if $\Gamma(\mathcal{P})$ has two orbits on the flags, such that adjacent flags are in distinct orbits. Chiral polyhedra have automorphism groups generated by a pair $\sigma_1,\sigma_2$ of ``rotations'' with properties similar to those of the generators for the combinatorial rotation subgroup of a regular polyhedron. Note that the underlying abstract polyhedron of a geometrically chiral polyhedron must combinatorially be chiral or regular.
\smallskip

We say that an abstract polyhedron $\mathcal{P}$ is \textit{vertex-transitive} if $\Gamma(\mathcal{P})$ is transitive on the vertices. This is equivalent to saying that $\mathcal{P}$ is (combinatorially) \text{uniform}, meaning that $\mathcal{P}$ is vertex-transitive and has (combinatorially) regular facets. In fact, the latter holds trivially for any abstract polyhedron.
\smallskip

The abstract theory of polyhedra is closely linked to the geometric theory via the concept of a realizations (see \cite[Ch.5]{41} and McMullen~\cite{McMullen New Book}). The regular, chiral, or uniform polyhedra investigated here are faithful realizations of underlying abstract polyhedra in the sense of \cite{41,McMullen New Book}. In this situation, the geometric symmetry group becomes a subgroup of the combinatorial automorphism group, but in general only a proper subgroup. For geometric regular polyhedra the two groups coincide, but for geometric chiral or uniform polyhedra the subgroup relationship can be proper. For example, all geometric chiral polyhedra with helical faces are combinatorially regular and their symmetry group has index 2 in the automorphism group~\cite{47}. 
\smallskip

From now on, when there is little chance of confusion, we use the term ``face'' to mean ``$2$-face'' (facet) and will do so primarily in the context of geometric polyhedra. 

\section{Geometric Regular Polyhedra}
\label{regpo}

In this section, we briefly describe the regular (skeletal) polyhedra following the notation and classification scheme of \cite[Ch.~7E]{41} (which differs from that in \cite{18,19,22}). There are 48 regular polyhedra, up to similarity and scaling of components (when applicable). They can be arranged into the following families: finite regular polyhedra, planar regular polyhedra, blended regular polyhedra, and pure (non-blended) regular polyhedra. There are 18 finite, 6 planar, 12 blended, and 12 pure regular polyhedra. (In \cite{18,19,22}, the grouping into families is slightly different from ours and leads to a low count of polyhedra.)\smallskip 

Within each of these families, we can choose a set of ``base polyhedra" and obtain all other polyhedra in the family via certain operations $\alpha$. These operations are the duality $\delta$, Petrie duality $\pi$, faceting $\varphi_2$, halving $\eta$, and skewing $\sigma$. When the operation $\alpha$ applies to a regular polyhedron $P$ to give a new regular polyhedron $P^\alpha$, the distinguished generators $r_0,r_1,r_2$ of $G(P)$ are transformed into the distinguished generators of $G(P^\alpha)$ as follows: 
\begin{equation}
\label{ops}
\begin{array}{rccl}
\delta:& (r_0,r_1,r_2)& \mapsto & (r_2,r_1,r_0), \\
\pi:& (r_0,r_1,r_2) &\mapsto &(r_0r_2,r_1,r_2), \\
\varphi_2:& (r_0,r_1,r_2)& \mapsto &(r_0,r_1r_2r_1,r_2), \\
\eta:& (r_0,r_1,r_2)& \mapsto &(r_0r_1r_0,r_2,r_1), \\
\sigma:& (r_0,r_1,r_2)& \mapsto &(r_1,r_0r_2,(r_1r_2)^2).
\end{array}
\end{equation}
\smallskip

For example, the family of finite regular polyhedra is given in the following display taken from \cite{41}:
\begin{equation*}
\includegraphics[height=8cm]{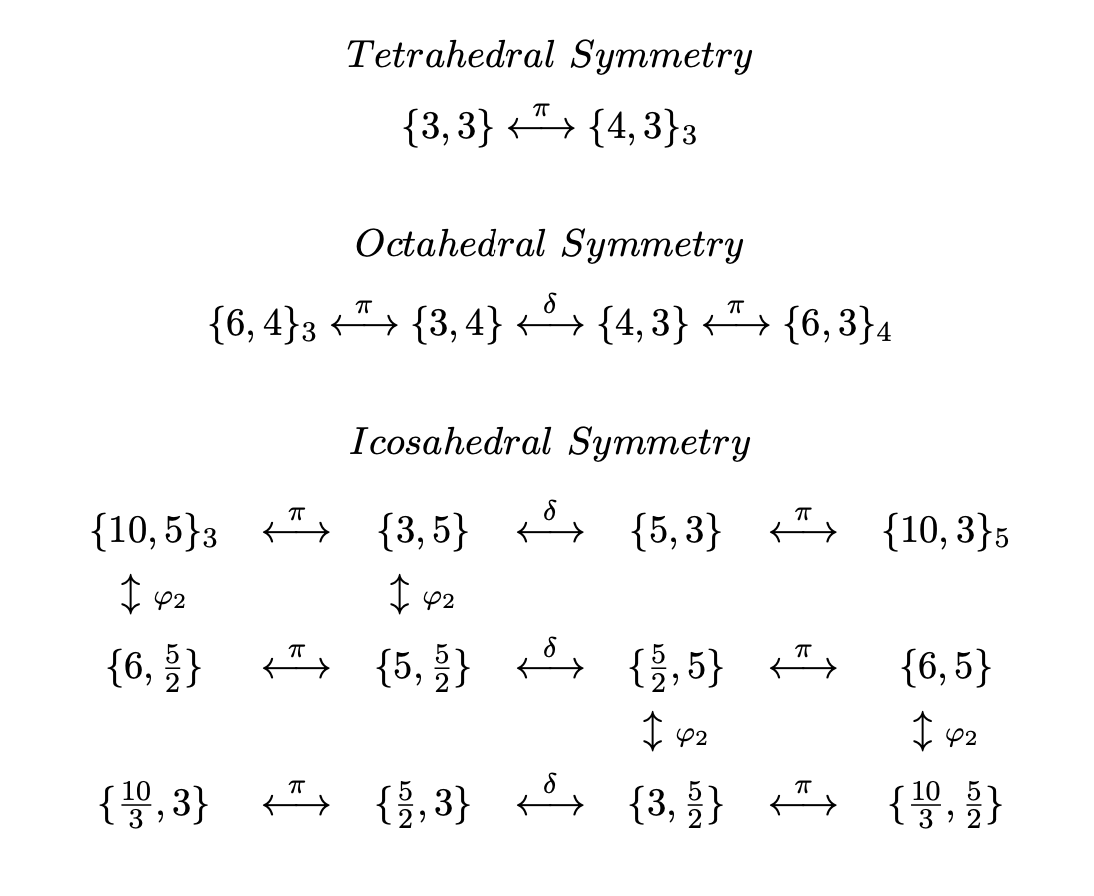}
\label{fig: finite regular polyhedra table 1}
\end{equation*}
Possible base polyhedra are the tetrahedron $\{3,3\}$, the cube $\{4,3\}$, and the icosahedron $\{3,5\}$. All the other finite regular polyhedra (and their symmetry groups) can be obtained from these by applying the operations $\delta$, $\pi$, and $\varphi_2$.
\smallskip

A \textit{Petrie polygon} of a regular polyhedron $P$ is a path along edges such that any two, but no three, successive edges lie in a face of $P$. A regular polyhedron of type $\{p,q\}$ is denoted $\{p,q\}_r$ if the length $r$ of its Petrie polygons determines its combinatorial type completely. This explains some of the subscripts in the above display. Similarly, a \textit{2-hole} of $P$ is a path along edges that leaves a vertex by the second edge from which it entered, always in the same sense (in some local orientation). A regular polyhedron of type $\{p,q\}$ is denoted $\{p,q\mid h\}$ if the length $h$ of its 2-holes determines its combinatorial type completely. 
\smallskip

For regular polyhedra $P$ we often require the index of the combinatorial rotation subgroup $G^+(P)$ in the full symmetry group $G(P)$. We know the index to be 1 or 2. In determining the index we exploit a well-known fact about presentations for groups generated by involutions, which in the context of regular polyhedra can be phrased as follows. Suppose $P$ is a regular polyhedron and its symmetry group has a presentation of the form 
\[G(P) = \langle r_0,r_1,r_2\mid \mathcal{R} \rangle,\] where $\mathcal{R}$ is a set of relators (words in $r_0,r_1,r_2$ that equal the identity element in $G(P)$). Then, $G^+(P)$ has index 2 in $G(P)$ if and only if $\mathcal{R}$ contains only words of even length in $r_0,r_1,r_2$.
\smallskip

Inspection of the presentations for the symmetry groups of regular polyhedra given in \cite[Ch. 7E]{41}, then allows us to sort the regular polyhedra into two groups; regular polyhedra with $G^+(P)$ having index 2 in $G(P)$ and regular polyhedra with $G^+(P) = G(P)$. The former are precisely the directly regular polyhedra, which topologically are regular maps on an orientable surface if they have finite faces. The latter are not directly regular, and topologically are regular maps on non-orientable surfaces if they have finite faces. We summarize the results in Theorem~\ref{index2summ} using the notation of \cite[Ch. 7E]{41}.
\smallskip

\begin{theorem}
\label{index2summ}
\noindent\begin{itemize}
\item[(a)] The finite regular polyhedra with an index 2 combinatorial rotation subgroup are: 
$$\{3,3\}, \ \{3,4\}, \ \{4,3\}, \ \{6,3\}_4, \ \{3,5\}, \ \{5,3\}, \ \{5,\tfrac{5}{2}\}, \ \{\tfrac{5}{2},5\}, \ \{3,\tfrac{5}{2}\}, \ \{\tfrac{5}{2},3\}.$$
For all other finite regular polyhedra the index is~1. 
\item[(b)] The planar regular apeirohedra with an index 2 combinatorial rotation subgroup are 
$$\{4,4\}, \ \{3,6\}, \ \{6,3\}, \ \{\infty,4\}_4, \ \{\infty,3\}_6.$$ 
For the remaining planar regular apeirohedron, $\{\infty,6\}_3$, the index is~1.
\item[(c)] Every blended regular apeirohedron has an index 2 combinatorial rotation subgroup: 
$$\{4,4\}\#\{\}, \ \{3,6\}\#\{\}, \ \{6,3\}\#\{\}, \ \{\infty,4\}_4\#\{\}, \ \{\infty,3\}_6\#\{\}, \ \{\infty,6\}_3\#\{\},$$ 
$$\{4,4\}\#\{\infty\}, \, \{3,6\}\#\{\infty\}, \, \{6,3\}\#\{\infty\}, \, \{\infty,4\}_4\#\{\infty\}, \, \{\infty,3\}_6\#\{\infty\}, \, \{\infty,6\}_3\#\{\infty\}.$$
\item[(d)] The pure regular apeirohedra with an index 2 combinatorial rotation subgroup are: 
$$\{4,6|4\}, \ \{6,4|4\}, \ \{6,6\}_4, \ \{4,6\}_6, \ \{6,4\}_6,$$ $$\{6,6|3\}, \ \{\infty,4\}_{6,4}, \ \{\infty,6\}_{4,4}, \ \{\infty,3\}^{(a)}, \ \{\infty,3\}^{(b)}.$$ For the remaining pure regular apeirohedra, $\{\infty,6\}_{6,3}$ and $\{\infty, 4\}_{\cdot,*3}$, the index is~1.
\end{itemize}
\end{theorem}
\smallskip

From these results, one may make the following useful observation. 
\smallskip

\begin{lemma}
The Petrie dual $P^\pi$ of a regular geometric polyhedron $P$ of type $\{p,q\}$ cannot be directly regular if $p$ is odd. Thus $G^{+}(P^\pi)=G(P^\pi)$ in this case.
\end{lemma}

\begin{proof}
This follows directly from Theorem \ref{index2summ}, but may also be proved directly as follows. Suppose $G(P)=\langle r_0,r_1,r_2\rangle$ so that $G(P^{\pi})=\langle r_0r_2,r_1,r_2\rangle=G(P)$. Let $u_0, u_1, u_2$ be the three generators of $G^+(P^\pi)$ defined in the same way as $s_0, s_1, s_2$ for $G^+(P)$. Then $u_1 = r_0r_2r_1$, $u_2 =r_1r_2 = s_2$, $u_0 = u_1u_2 = r_0$. Now, let $p$ be odd, say $p = 2k+1$ for some positive integer~$k$. As $G(P) = G(P^\pi) \geq G^+(P^\pi)$, it is sufficient to show that the generators of $G(P)$ can be obtained from the generators of $G^+(P^\pi)$. This is clear for $r_0$, since $r_0 = u_0$. For $r_1$, we can argue as follows. Since 
$$(r_0r_1)^2 = r_0r_1r_2r_2r_0r_1 = r_0r_1r_2r_0r_2r_1 = u_0u_2u_0u_2^{-1} = u_0u_2u_1$$ 
we observe that 
$$ 1 = (r_1r_0)^p = r_1(r_0r_1)^{2k}r_0 = r_1(u_0u_2u_1)^k u_0$$ 
and thus $r_1 \in \langle  u_0,u_1,u_2 \rangle = G^+(P^\pi)$. 
For $r_2$, we have 
$$r_2 = r_0(r_0r_2r_1)r_1 = u_0u_1r_1$$ 
and thus $r_2 \in G^+(P)$ also. Thus, if $p$ is odd, the generators of $G(P)$ all lie in $G^+(P^\pi)$ and so $G(P^\pi) = G^+(P^\pi)$.
\end{proof}
\smallskip

We are often concerned with finding fundamental regions for the combinatorial rotation subgroups of the symmetry groups of regular or chiral polyhedra $P$ in $\mathbb{E}^3$. Recall that a \textit{fundamental region} for a discrete group of isometries $G$ of $\mathbb{E}^n$ is an open subset $D$ of $\mathbb{E}^n$ such that
\[g(D) \cap D = \emptyset\;\;\; (g\in G,\, g\neq 1)\]
and 
\[\mathbb{E}^n = \bigcup_{g\in G} g(\text{cl}(D)).\]
That is, the images of the closure cl$(D)$ of $D$ under $G$ tile 3-space. Discrete groups of isometries permit many possible ways to select a fundamental region. For our purposes we begin with the orbit of a point $w$ not fixed by any non-trivial element of $G$ and choose as a fundamental region $D$ the open Dirichlet-Voronoi region of $w$ in the Dirichlet-Voronoi tessellation defined by the point orbit of $w$ \cite{Wythoffians}. The Appendix of \cite{Sk1} includes code written in Mathematica \cite{Mathematica} which generates the fundamental regions for the symmetry groups and combinatorial rotation subgroups of regular and chiral polyhedra.
\smallskip

\section{Skeletal snub polyhedra} 
\label{snub}

In this section, we define the snub construction for both regular and chiral skeletal polyhedra in $\mathbb{E}^3$ and discuss the geometric, combinatorial and topological structures of the resulting objects.  Our construction is based on Wythoff's construction, and the structures generated from the classical regular convex and star polyhedra generally yield well-known convex and star snubs; these classical snub polyhedra are treated in \cite{11}. Consequently, we will call these resulting structures {\it snubs}.
\smallskip

Let $P$ be a regular or chiral polyhedron in $\mathbb{E}^3$, and let $\mathcal{P}$ be the underlying abstract polyhedron of $P$. Recall from Section~\ref{skelpo} the definitions of $G^+(P)$ and the three symmetries $s_1,s_2,s_0$. More explicitly, if $P$ is regular and $r_0,r_1,r_2$ are the generators of $G(P)$ (associated with a base flag $\Phi$), then
$$s_{1}:=r_{0}r_{1}, \ s_{2}:=r_{1}r_{2}, \ s_{0}:=s_{1}s_{2} = r_{0}r_{2}.$$ 
On the other hand, if $P$ is chiral and $S_1, S_2$ are the generators of $G(P)$ (associated with a base flag $\Phi$ of $P$), then 
$$s_1 := S_1, \ s_2:= S_2, \ s_0:= s_1s_2 = S_1S_2.$$ 
In either case, the combinatorial rotation subgroup $G^+(P)$ of $G(P)$ is given by 
$G^+(P):= \langle s_1, s_2 \rangle$. 
Now if $P$ is regular, then $G^+(P)$ is isomorphic to $\Gamma^+(\mathcal{P})$. However, if $P$ is chiral, then $\mathcal{P}$ is either (combinatorially) regular or chiral and thus $G^+(P)=G(P)$ is either isomorphic to $\Gamma^+(\mathcal{P})$ (if $\mathcal{P}$ is regular) or $\Gamma(\mathcal{P})$ (if $\mathcal{P}$ is chiral).  \smallskip

Our construction proceeds from a suitably chosen initial point, $v$, and constructs the snub of a regular or chiral polyhedron $P$ as an orbit structure under the combinatorial rotation subgroup $G^+(P)$ of $P$. When $v$ is chosen in the closure of a fundamental region of $G^+(P)$ in $\mathbb{E}^3$ ``close" to the base flag of $P$, the resulting polyhedron is ``close" (in some sense) to the original polyhedron. In our context, we choose a fundamental region of $G^+(P)$ in such a way that for each $i=0,1,2$ for which the fixed point set of $s_i$ is non-empty, the fundamental region has a boundary point (often many such points) fixed by the symmetry $s_i$. However, the snub construction can be applied more generally with points $v$ chosen outside this fundamental region, but in this case the geometry of the resulting polyhedra is often harder to visualize. 
\smallskip

Suppose a suitable fundamental region $D$ of $G^+(P)$ has been chosen as described, and recall that $D$ is an open set. Roughly speaking, the closure of $D$ will ``split" (generally) into four sets which we also refer to as \textit{regions} (although three of them are not open sets); three regions will consist of points in space that are invariant under one of the symmetries $s_0, \ s_1$ or $s_2$ respectively, and the fourth will consist of points in space which are moved under all $s_0, \ s_1$ and $s_2$. It can happen that the second region is empty; for example, when $s_1$ is a screw translation. Note that if a point $v$ is invariant under two of the symmetries $s_0, s_1$ and $s_2$, then it is also invariant under the third and therefore under $G^+(P)$. Such a point $v$ exists only if $G^+(P)$ is finite, and then this point is unique and given by the center of $P$. Therefore, no two of the first three regions share a common point, unless $P$ is finite. Generally, by letting $v$ be in these four regions (but $v$ not a common fixed point of $s_0$, $s_{1}$, $s_{2}$) we obtain four combinatorially different polyhedra or polyhedra-like structures, namely the medial of $P$, the dual of $P$, the polyhedron $P$ itself, and a genuine snub, respectively. 
\smallskip

Note that the points in $D$ itself are not fixed by any nontrivial isometry in $G^+(P)$, so $D$ lies in the fourth region. The case when $v\in D$ is the scenario we are most interested in, and it results in a genuine snub. 

As we will see, genuine snubs arise more generally from points $v\in\mathbb{E}^3$ that satisfy what is called the \textit{initial placement condition}, 
\begin{equation}
\label{initialplacem}
{\rm IPC:}\;\;\; G_v^+(P)=\{1\},
\end{equation}
where $G_v^+(P)$ denotes the stabilizer of $v$ in $G^+(P)$. This will become clear in the subsequent sections. Note that the IPC  just means that $G^+(P)$ acts freely on the orbit of $v$ under $G^+(P)$, that is, 
$g(v)\neq v$ whenever $g\in G^+(P)$, $g\neq 1$. 

\subsection{The snub ${S}_P(v)$}
\label{snubcon}

We now proceed with the construction of the \textit{snub}  of $P$, which we denote by $S_P(v)$.  One may use Figure~\ref{fig:base faces} as a reference.  \smallskip

\begin{figure}[H]
    \centering
    \includegraphics[height=9cm]{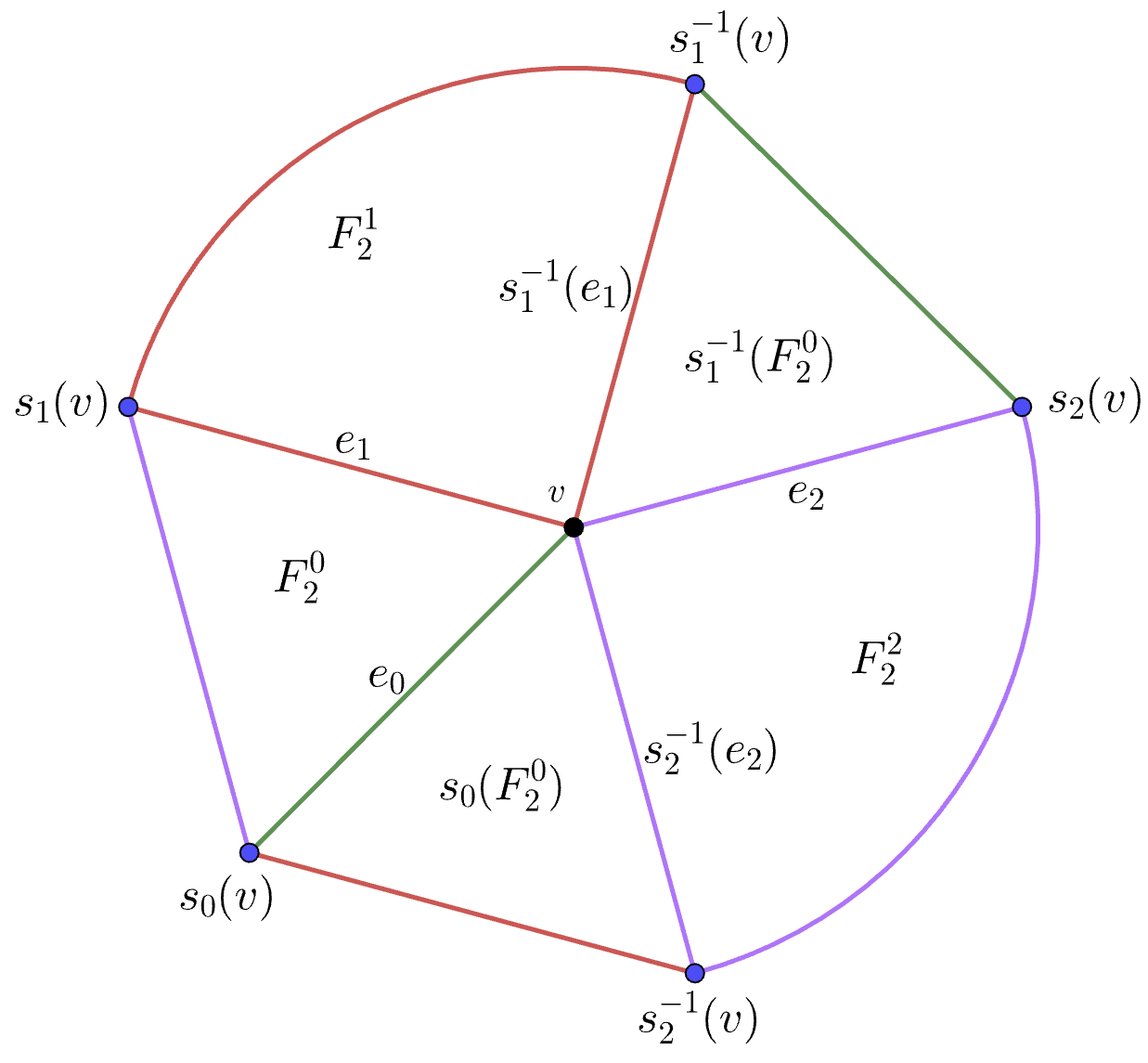}
    \caption{Vertices, edges, and faces of $S_P(v)$ adjacent to the initial vertex $v$}
    \label{fig:base faces}
\end{figure}


Let $v\in\mathbb{E}^3$; in practice, we often choose $v\in\text{cl}(D)$. We set $F_0:=v$ and $\mathcal{F}_0:=\{v\}$, and call $v$ the \textit{initial/base vertex} of the snub. At the outset we are not requiring that $v$ satisfies the initial placement condition of (\ref{initialplacem}). Further, define the \textit{type set} $I_v\subseteq\{0,1,2\}$ of $v$ by 
\begin{equation}
\label{Iv}
I_v:=
\begin{cases} 
      \{1,2\} & \text{if } s_0(v) = v \\
      \{2\} & \text{if } s_1(v) = v \\
      \{1\} & \text{if } s_2(v) = v \\
      \{0,1,2\} & \text{if } s_i(v) \neq v \text{ for all } i \\
\end{cases}
\end{equation}
We already mention here that the four possible type sets $I_v$ generally align with the four possible outcomes of $S_P(v)$:\ the medial of $P$, the dual of $P$, the polyhedron $P$ itself, and a genuine snub, respectively. Note that $I_v=\{0,1,2\}$ if $v$ satisfies the IPC. \smallskip 
   
Next, define the \textit{base edges} of ${S}_P(v)$ as 
\begin{equation}
\label{F1i}
F_1^{i} := \{v,s_i(v)\} \ \ \ \ \ (i \in I_v).
\end{equation}
We will often refer to these base edges as $e_i := F_1^i$. Denote the set of base edges by 
$$\mathcal{F}_1 := \{F_1^{i}:i \in I_v\}.$$ 
\smallskip

Defining the \textit{base faces} of $S_P(v)$ is less straightforward. We certainly would like to include 
\begin{equation}
\label{F2i}
F_2^{i} := \{s(F_1^{i}):s\in \langle s_i \rangle\}\ \ \ \ (i \in I_v \cap \{1,2\})
\end{equation}
as faces. In general, $F_2^1$ will be a face that is combinatorially isomorphic to the base face of~$P$, and $F_2^2$ will be a face that is combinatorially isomorphic to the vertex figure  about the base vertex in the base flag of $P$. There will be no other base faces if $I_v\subseteq\{1,2\}$. However, if $0\in I_v$ (that is, $v$ is moved under all of $s_0,s_1,s_2$), then the images of the faces above, under $G^+(P)$, are insufficient to close up the polyhedron. In this case we close the polyhedron up in a manner consistent with well-known snub polyhedra, like the snub cube or the snub dodecahedron, but variations of this are possible. If $0\in I_v$, then define 
\begin{equation}
\label{F20}
F_2^{0}:=\{F_1^{0}, F_1^{1}, s_1(F_1^{2})\} = \{e_0,e_1,s_1(e_2)\}.
\end{equation}
We show in Theorem \ref{thm1} that this choice is sufficient to close up this new structure. Next, denote the set of base faces as 
\[\mathcal{F}_2 := \{F_2^{i}:i \in I_v\}.\]
\smallskip

Finally, define the greatest face of ${S}_P(v)$ as $$F_3 := \{ s(F):s\in G^+(P), \ F\in\mathcal{F}_2 \}$$ and set $$\mathcal{F}_3 := \{F_3\}.$$ One may wish to be completely formal and consistent with the language of abstract polyhedra, in which case we must define the full \textit{snub} as 
\begin{equation}
\label{SPv}
{S}_P(v) := \{s(F): s\in G^+(P), \ F\in \bigcup_{j=-1}^3 \mathcal{F}_j \},
\end{equation}
with $\mathcal{F}_{-1}:=\{\emptyset\}$. In other words, to find all the vertices (0-faces), edges (1-faces), and faces (2-faces) of ${S}_P(v)$, look at the orbits of the base vertex, base edges, and base faces under $G^+(P)$. We let $V({S}_P(v))$, $E({S}_P(v))$, and $F({S}_P(v))$ denote the sets of vertices, edges, and faces of ${S}_P(v)$, respectively. Note that ${S}_P(v)$ is a partially ordered set. More precisely, if $i\leq j$ and $g,h \in G^+(P)$, then the image $g(B_i)$ of a base $i$-face $B_i$ is incident with the image $h(B_j)$ of a base $j$-face $B_j$ if and only $h^{-1}g(B_i)$ is an $i$-face of $B_j$ (in the usual meaning). \smallskip

We are mostly interested in the case when the stabilizer $G^+(P)_v$ of $v$ in $G^+(P)$ is trivial. On the other hand, when $v$ is chosen such that it is fixed by some $s_i$, ${S}_P(v)$ is often obtained by a well-known geometric operation on the polyhedron $P$. In the case that $s_0(v)=v$, ${S}_P(v)$ is often the medial of $P$; if $s_1(v)=v$, then ${S}_P(v)$ is almost always the dual of $P$; and if $s_2(v)=v$, then ${S}_P(v)$ is almost always $P$ itself. We also use the notation ${S}_P^i(v)$ in place of ${S}_P(v)$ if $s_i(v)=v$, for $i=0,1,2$. In a sense, ${S}_P^i(v)$ is a ``degenerate" snub.
\smallskip

We will discuss the snub construction for abstract regular polyhedra and regular maps in Section~\ref{section: topology}.

\subsection{Polyhedral structure of ${S}_P(v)$}
\label{polyhedralstruc}

In this section, we show that the snub ${S}_P(v)$ constructed from a regular or chiral polyhedron $P$ as described in the previous section indeed gives a skeletal polyhedron if the initial vertex $v$ satisfies the IPC of (\ref{initialplacem}). This is done in Theorem~\ref{thm1}. If $v$ satisfies the IPC, then $I_v=\{0,1,2\}$, $G_v^+(P)=\{1\}$, and $G^{+}(P)$ acts simply transitively on the vertex set of the snub. Geometrically nice snubs arise in particular when $v$ is chosen in the (open) fundamental region $D$ of $G^+(P)$.
\smallskip    

First note that if $P$ is a regular or chiral polyhedron, then $s_0, s_1, s_2$ are distinct isometries that satisfy the condition $\langle s_i \rangle \cap \langle s_j \rangle = \{1\}$ whenever $i\neq j$. This fact, together with the IPC for $v$, guarantees that the base edges $e_0, e_1, e_2$ are distinct, and that the base faces $F_2^0, F_2^1, F_2^2$ are also distinct. In particular, we can say that $F_2^1$ and $F_2^2$ only intersect at $v$, that $F_2^0$ and $F_2^1$ only intersect on $e_1$, and that $F_2^0$ and $F_2^2$ only intersect at $v$ (just as in Figure~\ref{fig:base faces}).  

Next, observe that each edge and face in ${S}_P(v)$ can be assigned a unique {\it type\/}, $i=0,1,2$, since the edge or face is an image of the base edge $e_i$ or the base face $F_2^i$ under $G^{+}(P)$, respectively. In other words, $G^{+}(P)$ has three edge orbits and three face orbits represented by the three base edges and three base faces, respectively. We sometimes refer to an edge or a face of type $i$ as a {\it type-$i$ edge} or {\it type-$i$ face}, respectively. Note that the assignment of types is well-defined, by the simple transitivity on the vertex-set. In fact, to explain by way of example, suppose that an edge $e$ of ${S}_P(v)$ has both type 1 and type 2. Then, there exist $s, t\in G^+(P)$ such that $e = s(e_1) = t(e_2)$  and thus $\{s(v), ss_1(v)\} = \{t(v), ts_2(v)\}$. Here, if $s(v) = t(v)$, then $s = t$ and thus $e_1 = e_2$, a contradiction. If $s(v) = ts_2(v)$, then $s = ts_2$ and thus $s_2 = t^{-1}s$;  similarly, we must also have $s_1 = s^{-1}t = s_2^{-1}$, which is impossible. Once the edges are known to have a unique type, we can easily derive the same result for faces. For example, suppose that a face $F$ of $\mathcal{S}_P(v)$ has type 1 and type 2. Then, there are $s, t\in G^+(P)$ such that $F = s(F_2^1) = t(F_2^2)$ and thus, $t^{-1}s(F_2^1) = F_2^2$. But $F_2^1$ contains only type-1 edges, and $F_2^2$ only type-2 edges, so the edge $t^{-1}s(e_1)$ of $F_2^2$ must have both types 1 and 2, a contradiction. \smallskip

Note that this allows us to assign colorings to the edges and the faces of $S_P(v)$. In our figures we use red for type-1 faces/edges, purple for type-2 faces/edges, and green for type-0 faces. 
\smallskip

The stabilizers of the base edges $e_i$ and base faces $F_2^i$ of $S_P(v)$ in $G^+(P)$ are given by 
\begin{equation}
\label{stabei}
G^{+}_{e_0}(P) = \langle s_0 \rangle,\;\, G^{+}_{e_1}(P) = G^{+}_{e_2}(P) = \{1\} 
\end{equation}
and 
\begin{equation}
\label{stabF2i}
G^{+}_{F_2^0}(P) = \{1\}, \;\, G^{+}_{F_2^1}(P) = \langle s_1 \rangle, \;\, G^{+}_{F_2^2}(P) = \langle s_2 \rangle,
\end{equation}
respectively. This is easy to verify for edges. For example, if $g\in G^+(P)$ fixes $e_0$, then clearly $g(v) = v$ or $g(v) = s_0(v)$, and thus $g=1$ or $g=s_0$. The face $F_2^0$ has only one edge of type 0, namely $e_0$, so if $g$ fixes $F_2^0$, then $g$ fixes $e_0$ and thus $g=1$ or $g=s_0$; however, $s_0(F_2^0) \neq F_2^0$, so $g=1$. If $g\in G^+(P)$ fixes $F_2^1$, then $g(v) = s_1^j(v)$ and thus $g=s_1^j$ for some $j$. The argument for ${F_2^2}$ is similar. 
\medskip

We now can establish the following theorem.

\begin{theorem}
\label{thm1}
Let $P$ be a regular or chiral polyhedron of type $\{p,q\}$, let $v\in \mathbb{E}^3$, and suppose $v$ satisfies the IPC of (\ref{initialplacem}) for $G^+(P)$. Then $S_{P}(v)$ is a geometric polyhedron, $G^{+}(P)$ acts simply vertex-transitively on $S_{P}(v)$, and $S_{P}(v)$ is of snub type $p.3.3.q.3$. 
\end{theorem}

\begin{proof} 
First note that the faces of ${S}_P(v)$ really are polygons. Since every face is an image of a base face, it is sufficient to explain this for the base faces $F_2^1$, $F_2^2$, and $F_2^0$. The latter is a triangle, so there is nothing to show. For $F_2^1$ we can argue as follows. Since $s_1^i(e_1) = \{s_1^i(v), s_1^{i+1}(v)\}$, the cycle $v, s_1(v), s_1^2(v), \dots , s_1^{-1}(v), v$ passes through all the vertices/edges of $F_2^{1}$ and determines a finite polygon if $p$ is finite (that is, $s_1$ has finite order). If $p=\infty$, then the two-sided infinite path $\dots,s_1^{-3}(v), s_1^{-2}(v), s_1^{-1}(v),v,s_1(v),s_1^2(v),s_1^3(v),\dots$ will do the trick. The argument for the (finite) base face $F_2^{2}$ is analogous. 

Next, we show that the edge graph of ${S}_P(v)$ is connected. Clearly it suffices to prove that for any vertex $w \in {S}_P(v)$ there exists an edge path between the base vertex $v$ and $w$. By construction, $w=g(v)$ for some $g \in G^+(P)$. If $g=1$, there is nothing to prove. 

If $g\neq 1$, then we may write $g = t_{1}t_{2}\ldots t_{k}$ where each $t_j$ is one of $s_1$, $s_1^{-1}$, $s_2$, or $s_2^{-1}$ (the inverses are only needed for infinite faces), and proceed by induction on $k$. If $k=1$, then $g$ is one of $s_1$, $s_1^{-1}$, $s_2$, or $s_2^{-1}$, and $g(v)$ is connected to $v$ by a single edge, namely $e_{1}:=F_{1}^{1}$, $s_{1}^{-1}(e_{1})$, $e_{2}:=F_{1}^{2}$, or $s_{2}^{-1}(e_{2})$, respectively. (Note that $e_0:=F_{1}^{0}$ was not used.) Thus the claim holds for $k=1$. 

Now suppose $k>1$ and the claim holds for all elements of $G^+(P)$ that can be expressed as a product of the above form with fewer than $k$ factors. Let $g = t_{1}t_{2}\ldots t_{k}$ as before, and set $h :=t_{1}t_{2}\ldots t_{k-1}$. By inductive hypothesis, $v$ can be connected to $h(v)$ by an edge path in ${S}_P(v)$. Now, since $v$ can be connected to $t_{k}(v)$ by a single edge, $h(v)$ can also be connected to $h(t_{k}(v))=g(v)$ by a single edge. Hence, adjoining the edge connecting $h(v)$ and $g(v)$ to the edge path from $v$ and $h(v)$ then gives an edge path from $v$ to $g(v)$. (Note that this process did not employ any edges of ${S}_P(v)$ that are images of $e_0$ under $G^{+}(P)$. Our argument also shows that at most $k$ edges are needed to connect $v$ and $g(v)$, but the number will generally be smaller if type-0 edges are also used.) Thus ${S}_P(v)$ is connected. \smallskip

Next, we prove that the vertex figures of ${S}_P(v)$ are connected. By the vertex-transitivity of $S_P(v)$ it is sufficient to consider the vertex figure  at $v$. First observe that there are exactly five edges containing $v$. In fact, suppose that the construction leads to a sixth edge, call it $e$, containing $v$. Then $e$ is the image of $e_1$, $e_2$, or $e_0$ under some $g \in G^+(P)$. Since $g\neq 1$, we cannot have $g(v)=v$, and thus $g(s_1(v))=v$, $g(s_2(v))=v$, or $g(s_0(v)) = v$, respectively, giving $g=s_1^{-1}$, $g=s_2^{-1}$, or $g=s_0^{-1}$. In either case we arrive at a contradiction. Our next arguments will show that the vertex degree is indeed 5.

Now, to demonstrate the connectedness of the vertex figure we will cycle through all faces containing $v$, as follows:
\begin{itemize}
    \item $F_2^{1} = \{e_1, s_1(e_1), s_1^2(e_1), ... , s_1^{-1}(e_1)\}$ if $p<\infty$, or\\
    $F_2^{1} = \{\ldots,s_1^{-2}(e_1),s_1^{-1}(e_1), e_1, s_1(e_1), s_1^2(e_1),\ldots\}$ if $p=\infty$.
    \item $F_2^{0} = \{e_1, s_1(e_2), e_0\}$
    \item $s_0(F_2^{0}) = \{s_0(e_1), s_2^{-1}(e_2), e_0\}$
    \item $F_2^{2} = \{e_2, s_2(e_2), ... , s_2^{-1}(e_2)\}$
    \item $s_2s_0(F_2^{0}) = s_1^{-1}(F_2^0) = \{s_1^{-1}(e_1), e_2, s_1^{-1}(e_0)\}$
\end{itemize}
Reading this list top-down and cycling back, observe that each pair of successive faces shares an edge containing $v$. This shows that the vertex figure at $v$ is a pentagon.

It remains to explain why every edge of $S_P(v)$ lies in two faces. Clearly it suffices to consider the base edges $e_1, e_2, e_0$. Let us start with $e_0$. Since $e_0$ only appears as an edge in faces of type 0, every face $F$ containing $e_0$ is such that $F=s(F_2^0)$ for some $s\in G^+(P)$. But, since $F_2^0 = \{e_0, e_1, s_1(e_2)\}$, the only edge that can map to $e_0$ under $s$ is a type-0 edge. But then $s(e_0) = e_0$, implying that either $s = 1$ or $s = s_0$. Thus, there are at most two faces containing $e_0$. Next, consider $e_1$. If $e_1$ lies in a face $F$, then $F$ could have type~0 or type~1. If $F$ has type 0, then $F=s(F_2^0)$ for some $s\in G^+(P)$. But then, as above, the edge of type 1 in $F_2^0$ must be mapped to $e_1$ and thus by (\ref{stabei}), $s = 1$ and $F=F_2^0$. If $F$ has type 1, then $F=s(F_2^1)$ and therefore $ss_1^k(e_1) = e_1$ for some $k$. Then again by (\ref{stabei}), $ss_1^k = 1$ and thus $s \in \langle s_1 \rangle$ and $F = F_2^1$. Hence, $e_1$ lies in two faces. An analogous argument applies to  $e_2$. 

Thus ${S}_P(v)$ is a geometric polyhedron. The remaining statements are clear by construction.
\end{proof}
\smallskip

Note that the inductive argument proving that ${S}_P(v)$ is connected gives rise to an algorithm that allows us to find an edge path between any two vertices in ${S}_P(v)$. 
\smallskip

\begin{remark}
\label{vinclD}
Usually, if $v$ does not satisfy the IPC of (\ref{initialplacem}), then ${S}_P(v)$ still is a geometric polyhedron. There are however rare cases, in particular when the index of $G^+(P)$ in $G(P)$ is 1, where an edge can lie in more than two faces. 
\end{remark}\smallskip

In the situation of Theorem~\ref{thm1}, as $G^{+}(P)$ acts simply transitively on the vertices of $S_P(v)$ and the vertex figures are pentagons, $S_P(v)$ has exactly 10 flag orbits under $G^+(P)$, and thus the number of flags itself is $10\,|G^{+}(P)|$ if $P$ is finite.  In certain instances the full symmetry group $G(S_P(v))$ of $S_P(v)$ is larger than its subgroup $G^+(P)$. In this case there are two possibilities. In the first, the stabilizer of the base vertex $v$ in $G(S_p(v))$ has order $2$ and contains a symmetry interchanging the $p$-gonal face $F_2^1$ and the $q$-gonal face $F_2^2$; in particular, $p=q$ and the $p$-gonal faces and $q$-gonal faces of $S_P(v)$ are congruent under symmetries of $S_P(v)$. In the second, $p=q=3$ and the vertex-stabilizers are isomorphic to $D_5$ (see Proposition~\ref{all3}).
\smallskip

We will see in Section~\ref{section: topology} in thw combinatorial context that the isomorphism type of the snub does not depend on the initial vertex $v$ as long as $v$ satisfies the IPC. In other words, if both $v$ and $v'$ satisfy the IPC, then $S_P(v)$ and $S_P(v')$ are isomorphic as abstract polyhedra. (However, in a geometric context, these snubs would not be congruent or similar in general.)

Not every regular or chiral geometric polyhedron has a geometric dual (but each has an abstract dual). The following observation addresses the relationship between the snubs in the case where geometric duals exist.
\smallskip

\begin{theorem}
Let $P$ be a regular or chiral polyhedron of type $\{p,q\}$ in $\mathbb{E}^3$, let $v\in \mathbb{E}^3$, and suppose $v$ satisfies the IPC of (\ref{initialplacem}) for $G^+(P)=\langle s_1,s_2\rangle$. Further, suppose that $P$ has a geometric dual $P^{\delta}$, and that in the chiral case the snub $S_P(v)$ is constructed from the enantiomorphic form of $P^{\delta}$ associated with the generators $s_2^{-1},s_1^{-1}$ of $G^{+}(P^\delta)$. Then the dual pair of geometric polyhedra $P$ and $P^{\delta}$ have congruent snubs. More precisely, ${S}_{P^{\delta}}(v)$ is congruent to ${S}_P(v)$. 
\end{theorem}

\begin{proof}
First note that $G^{+}(P^\delta)=G^{+}(P)$ and thus the IPC for $G^{+}(P^\delta)$ and $G^{+}(P)$ are identical conditions. 

Now suppose $P$ is regular, with $G(P) = \langle t_0, t_1, t_2 \rangle$ and  
$s_1 = t_0t_1, s_2 = t_1t_2, s_0 = t_0t_2$. 
Then $G(P^\delta) = \langle t_2, t_1, t_0 \rangle$ and 
$G^{+}(P^\delta)=\langle s'_1,s'_2,s'_0\rangle$ with 
$$s'_1 = t_2t_1 = s_2^{-1},\ \, s_2' = t_1t_0 = s_1^{-1},\, \ 
s'_0 = t_2t_0 = s_0^{-1} = s_0.$$ 
Let us look at the base edges $e'_0,e'_1,e'_2$ of $S_{P^\delta}(v)$. Here, $e'_0 = e_0$ and 
$$ e'_1 = \{v, s'_1(v)\} = \{v, s_2^{-1}(v)\} = s_2^{-1}(e_2), \ \, e'_2 = \{v, s'_2(v)\} = \{v, s_1^{-1}(v)\} = s_1^{-1}(e_1).$$ 
Now for the base faces, 
$${(F')}_2^1 = \{s(e'_1)| s \in \langle s'_1 \rangle\} = \{s(s_2^{-1}(e_2))|s\in \langle s_2^{-1} \rangle\} =  \{s(e_2)|s\in \langle s_2 \rangle\}.$$ 
It follows that ${(F')}_2^1 = F_2^2$. Similarly, ${(F')}_2^2 = F_2^1$. Finally, 
$${(F')}_2^0 = \{e'_0, e'_1, s'_1(e'_2)\} = \{e_0, s_2^{-1}(e_2), s_2^{-1}(s_1^{-1}(e_1))\} = s_0(\{e_0, s_1(e_2), e_1\}) = s_0(F_2^0).$$ 
Thus, the orbits of the base faces of ${S}_P(v)$ will be identical to the orbits of the base faces of ${S}_{P^\delta}(v)$, but the roles of the $p$-gonal faces and the $q$-gonal faces are interchanged. 

Finally, suppose $P$ is chiral. In this case $\delta$ transforms the generators $s_1,s_2$ of $G(P)$ according to 
$(s_{1}, s_{2}) \mapsto (s_2^{-1},s_1^{-1})$ (see~\cite{Chiral 1}). By our assumption on the enantiomorphic form of $P^\delta$,
the elements $s_2^{-1},s_1^{-1}$ are exactly the generators of $G(P^\delta)$ and the type set $I_v$ remains the same, namely $I_v=\{0,1,2\}$. But then the same argument as above yields the desired result.  
\end{proof}

\subsection{Uniformity}
\label{uniformity}

In general, the snub $S_P(v)$ of Theorem~\ref{thm1} is not a uniform polyhedron. By construction, $S_P(v)$ is vertex-transitive and has congruent regular $p$-gonal faces and congruent regular $q$-gonal faces for any choice of $v$. However, the triangle faces (of type 0), while mutually congruent, are usually not equilateral. In fact, if $v$ satisfies the IPC of (3), then ${S}_P(v)$ is uniform if and only if the base face $F_2^0$ of type 0 (with vertices $v, s_0(v),s_1(v)$) is an equilateral triangle, that is, 
\begin{equation}
\label{uni1}
||s_1(v)-v|| = ||s_0(v)-v|| = ||s_1(v)-s_0(v)||.
\end{equation}
If we set $w_1 := s_1(v)-v$, $w_2 := s_0(v)-v$ and $w_3 := s_1(v)-s_0(v)$, then the equations in (\ref{uni1}) above are equivalent to 
\begin{equation}
\label{uni2}
w_1(1)^2+w_1(2)^2+w_1(3)^2 = w_2(1)^2+w_2(2)^2+w_2(3)^2 = w_3(1)^2+w_3(2)^2+w_3(3)^2,
\end{equation}
where $w_i(j)$ denotes the $j^{\text{th}}$ coordinate of $w_i$. If we write $v=(x,y,z)$, the equations in (\ref{uni2}) are quadratic in $x,y,z$. We call them the \textit{uniformity equations} for ${S}_P(v)$. Thus, if $v$ satisfies the IPC of (3), the snub ${S}_P(v)$ is a uniform polyhedron if and only if $v=(x,y,z)$ is a solution of the uniformity equations. We call a solution $v=(x,y,z)$ to the equations~(\ref{uni2}) \textit{acceptable} if $v$ satisfies the IPC of (3) for $G^+(P)$.
\smallskip

The following lemma shows that there are instances of regular polyhedra $P$ where the uniformity equations (\ref{uni2}) do not have an acceptable solution. 

\begin{lemma}
\label{lemma: uniformity conditions}
Let $P$ be a geometric regular polyhedron in $\mathbb{E}^3$, let $v\in\mathbb{E}^3$, and suppose $v$ satisfies the IPC of (\ref{initialplacem}) for $G^+(P)=\langle s_1,s_2,s_0\rangle$. If the generator $s_0$ of $G^{+}(P)$ is a plane reflection, then there does not exist an acceptable solution $v$ to the uniformity equations for ${S}_P(v)$.  
\end{lemma}

\begin{proof}
For the sake of contradiction suppose that there exists an acceptable solution $v$ to the uniformity equations for $S_P(v)$. Then, ${S}_P(v)$ is a uniform polyhedron and $F_2^0$ is regular. It follows that the plane of reflection of $s_0$ (or line of reflection of $s_0$ in 2 dimensions) is the perpendicular bisector of $e_0$ and passes through $s_1(v)$. But then $s_0$ fixes $s_1(v)$ and thus lies in the  stabilizer of $s_1(v)$. This contradicts the fact that $G^{+}_{s_1(v)}(P) = s_1 G^{+}_v(P) s_1^{-1}$ is trivial.  
\end{proof}

Note that no geometrically chiral polyhedron $P$ has $s_0$ as a plane reflection, so Lemma~\ref{lemma: uniformity conditions} has been restricted to regular polyhedra.  
\smallskip

\begin{remark} 
\label{hemicubesnub}
If the snub ${S}_P(v)$ cannot be made uniform for a suitable choice of $v$, there still can exist a uniform faithful realization of the underlying abstract polyhedron of ${S}_P(v)$. An example is given by the snub of the hemi-cube $\{4,3\}_3$, the Petrie-dual of the tetrahedron $\{3,3\}$. Notice that ${S}_{\{4,3\}_3}(v)$ cannot be made uniform, by Lemma~\ref{lemma: uniformity conditions}. However, ${S}_{\{4,3\}_3}(v)$ is combinatorially isomorphic to the standard snub cube, which is a uniform convex polyhedron in 3-space.   
\end{remark}
\smallskip

\subsection{Topology of ${S}_P(v)$} \label{section: topology}

In this section, we show that the geometric snubs discussed in the previous sections are faithful geometric realizations of combinatorial snubs derived from abstract polyhedra. The details are summarized in Theorem~\ref{thm: isomorphisms between wythoffians} below. We begin by explaining how to perform the snub construction on abstract regular polyhedra. \smallskip 

Every abstract polyhedron $\mathcal{P}$ with finite faces (2-faces) and vertex figures can be represented as a map $M$ on a closed surface (see \cite[Chs.~7B,\,7C]{41}). We usually will not distinguish between $\mathcal{P}$ and $M$. If $\Gamma(M)$ denotes the automorphism group of $M$, then $\Gamma(M)\cong\Gamma(\mathcal{P})$ and $\Gamma(M)$ can be realized as a group of homeomorphisms of $M$.

Now suppose that $M$ is a regular map (corresponding to a regular polyhedron $\mathcal{P}$) on a surface $K$, with automorphism group $\Gamma(M) = \langle \rho_0, \rho_1, \rho_2 \rangle$ and combinatorial rotation subgroup $\Gamma^+(M)= \langle \sigma_1,\sigma_2\rangle$, where $\sigma_{1}:=\rho_0\rho_1, \sigma_{2}:=\rho_1\rho_2$. Consider the triangulation $B(M)$ of $K$ obtained as the ``barycentric subdivision" (order complex) of $M$. The triangles in $B(M)$ correspond to the flags of $M$. Pick a (closed) {\it base triangle} $\Delta_1$ in $B(M)$ (corresponding to the base flag $\Phi$ of $\mathcal{P}$) and set $\Delta:=\text{int}(\Delta_1)$. If $\Gamma^+(M)$ has index $2$ in $\Gamma(M)$, then $\Delta':= \text{int}(\Delta_1 \cup \rho_0(\Delta_1))$ is an (open) fundamental region (fundamental triangle) of $\Gamma^+(M)$ on $K$ and includes $\Delta$. Its closure, $\Delta_1 \cup \rho_0(\Delta_1)$, is referred to as the {\em double base triangle} in $B(M)$. If the index is 1, then instead $\Delta$ itself is an (open) fundamental region (triangle) of $\Gamma^+(M)$ on $K$. Either way, we let $z$ be an initial point in $\Delta$ and consider the orbit of $z$ under $\Gamma^+(M)$. Using  $\sigma_1,\sigma_2$, we can construct the {\it snub\/} $S_M(z)$ of $M$ as we did in Section~\ref{snubcon}. By letting $z$ be in $\Delta$ we can achieve that all vertices, edges, and faces of the snub lie on $K$ and that $S_M(z)$ is 2-cell embedded into $K$ if $M$ is orientable; in the non-orientable case, the faces of $S_M(z)$ overlap on $K$. 

One can also perform the snub construction purely combinatorially on the abstract polyhedron $\mathcal{P}$ itself, by taking the orbit of $\Phi$ under $\Gamma^+(\mathcal{P})$ as the vertex set of ${S}_{\mathcal{P}}(z)$ and using $\sigma_1,\sigma_2$ to impose the combinatorial structure on this vertex set as we did in Section~\ref{snubcon}. As the stabilizer of $\Phi$ in $\Gamma^{+}(\mathcal{P})$ is trivial and $z$ was chosen in the (open) fundamental region, the construction will be analogous to the geometric construction presented before. 

This also shows that in the construction of the snub of a regular map $M$ we could have chosen $z$ outside $\Delta$, as long as the stabilizer of $z$ in $\Gamma^{+}(M)$ is trivial (that is, $z$ satisfies a condition like the IPC, for $\Gamma^{+}(M)$). However, for our purposes it is sufficient to let $z\in\Delta$. This has the added benefit that $S_M(v)$ is cellularly embedded on the surface $K$ if $M$ is orientable. In general, this will no longer be true for other choices of $z$.

The above constructions can also be applied to chiral maps (and abstract  polyhedra, respectively), with minor modifications: if $\Gamma(M) = \langle \sigma_1, \sigma_2 \rangle$, then we choose as fundamental region $\Delta'$ for $\Gamma(M)$ the interior of the union of the base triangle $\Delta_1$ and its 0-adjacent triangle (flag); this union is again referred to as the {\it double base triangle\/}. The initial vertex for the snub would again have to be chosen in $\Delta:=\text{int}(\Delta_1)$ to guarantee that the snub is 2-cell embedded into the surface. In the chiral case, the underlying surface $K$ is orientable and also serves as the surface for the snub. As explained in the geometric context, changing the base flag of a chiral polyhedron to its 0-adjacent flag as the new base flag, will result in new distinguished generators of $\Gamma(M)$ (but not in a change of the fundamental region). If $\sigma_1,\sigma_2$ are the distinguished generators associated with the base flag, then $\sigma_1^{-1},\sigma_1^2\sigma_2$ are the distinguished generators associated with the 0-adjacent flag.
\smallskip

\begin{theorem}
\label{thm: isomorphisms between wythoffians}
Let $P$ be a regular or chiral polyhedron in $\mathbb{E}^3$ with finite faces (and vertex figures), and let $G^+(P)=\langle s_1,s_2\rangle$. Let $M$ denote the underlying abstract polyhedron realized as a map on a closed surface $K$, let $B(M)$ denote the barycentric subdivision of $M$ on $K$, and let $\Delta$ denote the interior of the base triangle of $B(M)$. Then the geometric and combinatorial snubs are related as follows. 
\begin{itemize}
    \item[(a)] If $v,v'\in\mathbb{E}^3$ and both $v$ and $v'$ satisfy the IPC of (\ref{initialplacem}) for $G^+(P)$, then the geometric snubs ${S}_P(v)$ and ${S}_P(v')$ are isomorphic. 
        
    \item[(b)] If $z,z'\in\Delta$, then the combinatorial snubs ${S}_M(z)$ and ${S}_M(z')$ are isomorphic. 
        
    \item[(c)] If $v\in\mathbb{E}^3$ and $v$ satisfies the IPC of (\ref{initialplacem}) for $G^+(P)$, and if $z\in\Delta$, then the geometric snub ${S}_P(v)$ is isomorphic to the combinatorial snub ${S}_M(z)$. 
    
    \end{itemize}
\end{theorem}

\begin{proof}
We only prove the third part and construct an isomorphism from ${S}_M(z)$ to ${S}_P(v)$. Direct arguments for the first and second  parts are similar (but note that each also follows from the third part). Our proof exploits the fact that the action of $\Gamma^+(M)$ on the vertex set of ${S}_M(z)$ is equivalent to that of $G^+(P)$ on the vertex set of $P$, and that the incidence and the base faces in ${S}_P(v)$ are strictly speaking induced by the incidence and the base faces in ${S}_M(z)$. \smallskip

We will assume that $P$ is regular, as the proof of the chiral case proceeds in almost exactly the same manner. Thus $G(P) = \langle r_0, r_1, r_2\rangle$ and $G^+(P)=\langle s_1, s_2,s_0\rangle$ where $s_1:=r_0r_1$, $s_2:= r_1r_2$, and $s_{0}:=s_1s_2=r_0r_2$. As $M$ is also regular, $\Gamma(M) = \langle \rho_0, \rho_1, \rho_2\rangle$ and $\Gamma^+(M) = \langle\sigma_1,\sigma_2\rangle$ where $\sigma_1 := \rho_0\rho_1$, $\sigma_2 := \rho_1\rho_2$, and $\sigma_0 := \sigma_1\sigma_2=\rho_0\rho_2$. In particular, 
\[\begin{array}{rrcl}
\kappa:& \Gamma(M)& \mapsto& G(P)\\
&\rho_i &\mapsto& r_i
\end{array}\]
is a group isomorphism whose restriction to $\Gamma^+(M)$ is an isomorphism onto $G^+(P)$ mapping $\sigma_i$ to $s_i$ for each $i$. (If $P$ was chiral and $G(P)=\langle s_1,s_2\rangle$, the finiteness of the faces of $P$ would imply that $M$ is abstractly chiral \cite{47}, and so the setup would be $\kappa: \Gamma^+(M)\rightarrow G(P)$ defined by $\sigma_1 \mapsto s_1$ and $\sigma_2 \mapsto s_2$.) \smallskip

We will build an isomorphism $\beta:{S}_M(z) \mapsto {S}_P(v)$ inductively from its restrictions, also denoted $\beta$, to the vertex set and the edge set. The vertex, edge, and face sets of ${S}_M(z)$ are denoted $V_M$, $E_M$, and $F_M$, respectively, and similarly for ${S}_P(v)$.

To begin with, let again $v$ be the base vertex of ${S}_P(v)$ and recall that it satisfies the IPC. We denote the base edges of ${S}_P(v)$ by 
$$e_1^P = \{v,s_1(v)\}, \ e_2^P = \{v,s_2(v)\}, \ e_0^P = \{v,s_0(v)\},$$ 
and the base faces of ${S}_P(v)$ by 
$$F_2^{P,i} = \{g(e_i^P)|g\in \langle s_i \rangle\} \;\, (i=1,2), \;\;\, F_2^{P,0} = \{e_0^P,e_1^P,s_1(e_2^P)\}.$$ 
Similarly, let $z$ be the base vertex of ${S}_M(z)$ lying in $\Delta$ and denote the base edges of ${S}_M(z)$ by  
$$e_1^M = \{z,\sigma_1(z)\}, \ e_2^M = \{z,\sigma_2(z)\},\ e_0^M = \{z,\sigma_0(z)\}$$ and the base faces of ${S}_M(z)$ by 
$$F_2^{M,i} = \{\gamma(e_i^M)|\gamma\in \langle \sigma_i\rangle\} \;\, (i=1,2), \;\;\, F_2^{M,0} = \{e_0^M,e_1^M,\sigma_1(e_2^M)\}.$$ 

On the vertex set we define $\beta$ by
\begin{equation}
\label{betavertex}
\begin{array}{rlcl}
\beta:& V_M& \mapsto& V_P\\
&\tau(z)& \mapsto& \kappa(\tau)(v) \;\;\; (\tau\in \Gamma^+(M)).
\end{array}
\end{equation}
Then $\beta(z) = v$, and as the base vertices have trivial stabilizers, it is straightfoward to check that $\beta$ is well-defined and bijective. 
\smallskip

Note that the actions of $\Gamma(M)$ on $V_M$ and $G(P)$ on $V_P$ are equivalent, so their restrictions to $\Gamma^+(M)$ and $G^+(P)$ are also equivalent. More explicitly, $\beta(\gamma(\tau(z))) = \kappa(\gamma)(\beta(\tau(z)))$ whenever $\gamma,\tau\in\Gamma^{+}(M)$. This equivalence is reflected in the definition of $\beta$. \smallskip 

From here we proceed by extending (\ref{betavertex}) to 
\begin{equation}
\label{betaedge}
\begin{array}{rlcl}
\beta:& E_M& \mapsto& E_P\\
&\tau(e_i^M)& \mapsto& \kappa(\tau)(e_i^P) \;\;\; (\tau\in \Gamma^+(M),\, i=0,1,2),
\end{array}
\end{equation}
and show that this mapping is well-defined and bijective. In fact, since edges have a unique type and the stabilizers of the base edges are as stated as in Section~\ref{polyhedralstruc}, we find that   
\begin{equation}
\begin{aligned}\nonumber
    \tau(e_i^M) = \tau'(e_j^M)  
    &\iff i=j, \  \begin{cases}
        \tau'=\tau, \text{ if } i=1,2 \\   
        \tau' = \tau \text{ or } \tau' = \tau \sigma_0, \text{ if } i=0
    \end{cases} \\[.03in]
    &\iff i=j, \  \begin{cases}
        \{\kappa(\tau')(v), \kappa(\tau')s_i(v)\} = \{\kappa(\tau)v,\kappa(\tau)s_i(v)\} , \text{ if } i=1,2 \\
        \{\kappa(\tau')(v), \kappa(\tau')s_0(v)\} = \{\kappa(\tau)v,\kappa(\tau)s_0(v)\}, \text{ if } i=0
    \end{cases} \\[.03in] 
    &\iff\; \kappa(\tau)(e_i^P) = \kappa(\tau')(e_j^P).
\end{aligned}
\end{equation}
This shows that $\beta$ of (\ref{betaedge}) is well-defined and injective. 
The surjectivity is implied by the surjectivity of $\kappa$. \smallskip

The combined mapping $\beta:V_M \cup E_M \rightarrow V_P \cup E_P$ is incidence preserving in both directions. Suppose $\tau,\varphi\in \Gamma^+(M)$ and vertex $\tau(z)$ of $M$ is contained in edge $\varphi(e_i^M)$. Then 
$$\tau(z) \in \varphi(e_i^M) = \{\varphi(z), \varphi \gamma(z)\}$$ 
with $\gamma = \sigma_1, \sigma_2, \sigma_0$ as $i=1,2,0$. Hence $\tau = \varphi$ or $\tau = \varphi\gamma$, and therefore $\kappa(\tau) = \kappa(\varphi)$ or $\kappa(\tau) = \kappa(\varphi)\kappa(\gamma)$ (with $\kappa(\gamma) = s_1, s_2, s_0$). It follows that 
$$\kappa(\tau)(v)\in \{\kappa(\varphi)(v), \kappa(\varphi)\kappa(\gamma)(v)\} = \kappa(\varphi)(e_i^M),$$ as required. As all steps can be reversed, the mapping is incidence preserving in both directions. \smallskip

Next, we define $\beta$ on the face set by
$$\begin{array}{rlcl}
\beta:& F_M& \mapsto& F_P\\
&\tau(F_2^{M,i})& \mapsto& \kappa(\tau)(F_2^{P,i}) \;\;\; (\tau\in \Gamma^+(M),\, i=0,1,2),
\end{array}$$
and prove that $\beta$ is well-defined and bijective. The surjectivity is clear from the surjectivity of $\kappa$. From the structure of the face stabilizers in (\ref{stabF2i}), and the fact that faces have a unique type, we obtain  
\begin{equation}
    \begin{aligned}\nonumber
        \tau(F_2^{M,i}) = \tau'(F_2^{M,j}) \!&\!\iff\!\! i=j,
        \begin{cases}
            \tau'=\tau\sigma_i^k \text { for some } k, \text{ if } i=1,2 \\
            \tau' = \tau, \text{ if } i=0
        \end{cases} \\[.03in]
        \!&\!\iff\!\! i=j, \begin{cases}
            \kappa(\tau')(F_2^{P,i}) = \kappa(\tau)s_i^k(F_2^{P,i}) = \kappa(\tau)(F_2^{P,i}) 
                                    \text{ for some } k, \text{ if } i=1,2 \\
            \kappa(\tau')(F_2^{P,0}) = \kappa(\tau)(F_2^{P,0}), \text{ if } i=0
        \end{cases} \\[.03in]     
        \!&\!\iff\!\!\kappa(\tau)(F_2^{P,i}) = \kappa(\tau')(F_2^{P,j}) .
    \end{aligned}
\end{equation}
\smallskip

It remains to show that the combined mapping $\beta: V_M \cup E_M \cup F_M \rightarrow V_P \cup E_P \cup F_P$ is incidence preserving. We already checked the incidence preserving property for the vertex-edge pairs. Now suppose we have an edge $\tau(e_i^M)$ incident with a face $\varphi(F_2^{M,j})$,  where $\tau,\varphi\in \Gamma^+(M)$. If $j=1,2$, then all edges of $\varphi(F_2^{M,j})$ are also of type $j$, so we have $i=j$. If $j=0$, then all three types occur among the edges of $\varphi(F_2^{M,j})$. First, we treat the case $j=1,2$ (so now $i=j$) and appeal to the structure of the edge stabilizers: 
\begin{equation}
    \begin{aligned}\nonumber
    \tau(e_i^M) \in \varphi(F_2^{M,i}) = \{\varphi\gamma(e_i^M)|\gamma \in \langle \sigma_i \rangle\} 
        &\!\iff\! \tau = \varphi \gamma \text{ for some } 
                            \gamma \in \langle \sigma_i \rangle\\
        &\!\iff\! \kappa(\tau) = \kappa(\varphi) g \text{ for some } 
                            g \in \langle s_i \rangle\\
        &\!\iff\! \kappa(\tau)(e_i^P) \in \{\kappa(\varphi)g(e_i^P)|g \in \langle s_i \rangle\}=\kappa(\varphi)(F_2^{P,i}).
    \end{aligned}
\end{equation}

\noindent Similarly, if $j=0$ (and $i=0,1,2$), then 
\begin{equation}
    \begin{aligned}\nonumber
    \tau(e_i^M) \in \varphi(F_2^{M,0}) \!=\! \{\varphi(e_0^M),\varphi(e_1^M), \varphi \sigma_1(e_2^M)\}
    &\iff \tau \in \varphi\langle \sigma_0 \rangle \text{ or } \tau = \varphi,\varphi\sigma_1 \\
    &\iff \kappa(\tau) \in \kappa(\varphi)\langle s_0 \rangle \text{ or } \kappa(\tau) =    
      \kappa(\varphi),\kappa(\varphi)s_1\\
    &\iff \kappa(\tau)(e_i^P) \in \{\kappa(\varphi)e_0^P, \kappa(\varphi)e_1^P, \kappa(\varphi)s_1e_2^P\} \\
    &\iff \kappa(\tau)(e_i^P) \in \kappa(\varphi)(F_2^{P,0}).
    \end{aligned}
\end{equation}

\noindent Hence, $\beta$ is incidence preserving on the edge-face pairs. Then $\beta$ is also incidence preserving on the vertex-face pairs, since the incidence between vertices and faces is given by the transitive closure of the incidence between vertices and edges, and between edges and faces. Thus, the combined mapping $\beta$ is an isomorphism, $\beta: {S}_M(z) \rightarrow {S}_P(v)$.
\smallskip

Note that the isomorphism $\beta:{S}_M(z) \rightarrow {S}_P(v)$ is built  inductively from its restrictions to $V_M$ and $E_M$, respectively, so we can think of $\beta$ as $\beta = \{\beta_k\}_{k=0,1,2}$, where $\beta_0: V_M \rightarrow V_P$, $\beta_1:E_M\rightarrow E_P$, $\beta_2: F_M\rightarrow F_P$. The exact relationship between the $\beta_k$ is as follows: 

\noindent
\begin{itemize}
\item[(i)] The map $\beta_0$ is given by 
    $$\begin{array}{rlcl}
    \beta_0:& V_M& \mapsto& V_P\\
            &\tau(z)& \mapsto& \kappa(\tau)(v) \;\;\; (\tau\in \Gamma^+(M)).
    \end{array}$$
    
    \item[(ii)] Then $\beta_1$ is given by 
    $$\begin{array}{rlcl}
    \beta_1: & E_M         &\mapsto& E_P\\
             & \{z_1,z_2\} &\mapsto& \{\beta_{0}(z_1),\beta_0(z_2)\}.
    \end{array}$$
    To verify this, write  
    $$\{z_1,z_2\} = \{\tau(z),\tau\gamma(z)\} = \tau(\{z,\gamma(z)\}) = \tau(e_i^M),$$
    for some $\tau\in\Gamma^{+}(M)$ and with $\gamma = \sigma_1, \sigma_2, \sigma_0$, and observe that under $\beta$ this gets mapped to 
    $$\kappa(\tau)(e_i^P) = \{\kappa(\tau)(v),\kappa(\tau)\kappa(\gamma)(v)\} = \{\beta_0(z_1),\beta_0(z_2)\}.$$ 
    
    \item[(iii)] Finally, $\beta_2$ is given by 
    $$\begin{array}{rlcl}
    \beta_2:& F_M         & \mapsto& F_P\\
          &F=\{f_1,\ldots,f_m\}& \mapsto& \{\beta_1(f_1),\ldots,\beta_1(f_m)\},
    \end{array}$$
    where $F$ is a generic face in $F_M$ with edges $f_1,\dots,f_m$. In order to verify this, look at the faces of types 0, 1, 2 separately. For faces of type $0$, if  
    $$F=\{f_1,f_2,f_3\} = \{\tau(e_0^M), \tau(e_1^M), \tau\sigma_1(e_2^M)\} = \tau(F_2^{M,0})$$ 
    with $\tau\in \Gamma^+(M)$, then $\beta$ maps $F$ to $$\kappa(\tau)(F_2^{P,0}) = \{\kappa(\tau)(e_0^P), \kappa(\tau)(e_1^P), \kappa(\tau)s_1(e_2^P)\} = \{\beta_1(f_1), \beta_1(f_2), \beta_1(f_3)\}.$$ 

    For faces of types $i=1,2$ (here $m=p$ or $m=q$, respectively), if 
    $$F=\{f_1, \dots, f_m\} = \{\tau\gamma(e_i^M)|\gamma\in \langle \sigma_i\rangle \} = \tau(F_2^{M,i}),$$ then under $\beta$ this maps to 
    $$\kappa(\tau)(F_2^{P,i}) = \{\kappa(\tau)g(e_i^P)|g\in \langle s_i\rangle\}  = \{\kappa(\tau\gamma)(e_1^P)|\gamma\in \langle \sigma_1 \rangle\} = \{\beta_1(f_1),\dots, \beta_1(f_m)\}.$$    
\end{itemize}
\smallskip

The recursive setup of $\beta=\{\beta_k\}_{k=0,1,2}$ is similar to Theorem 5A1 in \cite{41}. It uses the important property that the edges of $P$ can be identified with 2-element subsets of the vertex set, and the faces of $P$ with $m$-element subsets of the edge set of $P$. In geometric polyhedra this holds automatically. Therefore it also holds for $M$. 
\end{proof}
\smallskip

We already noted earlier that the snub of an orientable regular polyhedron or a chiral polyhedron (with finite faces) is again orientable. The following theorem says that, more generally, the snub of any geometric regular polyhedron with finite faces is orientable, even if the original polyhedron is non-orientable. 
\smallskip

\begin{theorem}
\label{doublecover}
Let $P$ be a regular or chiral polyhedron in $\mathbb{E}^3$ with finite faces, let $v\in\mathbb{E}^3$, and suppose that $v$ satisfies the IPC of (\ref{initialplacem}) for $G^+(P)$.
Then, ${S}_P(v)$ is orientable.
\end{theorem}
  
\begin{proof}
It suffices to consider regular non-orientable polyhedra, as the claim holds in all other cases. Suppose that $P$ is a regular, non-orientable polyhedron. Then, $P$ is combinatorially isomorphic to a non-orientable map $M$ on a closed surface $K_M$. By Wilson~\cite{Wilson} (for the compact case) and Jones~\cite{Jones} (for the general case), there exists a unique regular map $N$ on an orientable compact surface $K_N$ which is a smooth two-fold covering of $M$ and hence of $P$. (Uniqueness of the double cover is not discussed in~\cite{Jones}, but it is not difficult to prove.) Here smoothness simply means that the Schläfli symbols of $N$ and $M$ are the same, $\{p,q\}$. More explicitly, there exists a 2-1 covering projection $f$ from $K_N$ to $K_M$ which maps $N$ to $M$ and is one-to-one around the face centers and vertices of $N$.\smallskip 

We will set up an isomorphism between a combinatorial snub ${S}_N(z)$ of $N$ on $K_N$ and the geometric snub ${S}_P(v)$ in $\mathbb{E}^3$. Alternatively, we could work with a snub of $M$ on $K_M$, but we prefer to work directly with the snub of $P$. Since the snubs of orientable regular maps are orientable, this will show that $\mathcal{S}_P(v)$ is orientable.\smallskip 

To this end, let $\Gamma(N)=\langle \rho_0, \rho_1, \rho_2 \rangle$ and $G(P) = \langle r_0, r_1,r_2\rangle$ $(\cong \Gamma(M)$). Again, $\Gamma(N)$ is viewed as a group of homeomorphisms of $K_N$. Then $\rho_i\mapsto r_i$ $(i=0,1,2)$ defines a surjective homomorphism $\kappa':\Gamma(N) \rightarrow G(P)$ and since $f$ is a twofold covering, the kernel of $\kappa'$ has order 2. In particular, it follows from the proof in \cite{Wilson} (see also \cite{Conder Wilson}) that the restriction of $\kappa'$ to $\Gamma^+(N)$ is an isomorphism. Thus 
\[\begin{array}{rrcl}
\kappa:& \Gamma^+(N)& \mapsto& G^+(P)\\
       & \sigma_i &\mapsto& s_i
\end{array}\]
defines an isomorphism. Note that $G^+(P) = G(P)$, since $P$ is non-orientable; but that $\Gamma^+(N)$ has index 2 in $\Gamma(N)$, since $N$ is orientable.\smallskip 

Next, we turn to the snubs ${S}_N(z)$ and ${S}_P(v)$. Here, $z$ is chosen in $\Delta$ and thus $z$ is moved under $\rho_0$. We construct an isomorphism $\beta$ between ${S}_N(z)$ and ${S}_P(v)$ recursively as $\beta = (\beta_k)_{k=0,1,2}$, as in Theorem \ref{thm: isomorphisms between wythoffians}. Thus, $$\beta_0: V_N \rightarrow V_P, \ \beta_1:E_N \rightarrow E_P, \ \beta_2:F_N \rightarrow F_P,$$ where $\beta_0$ is given by $$\tau(z)\mapsto \kappa(\tau)(v) \ \ (\tau\in \Gamma^+(N)),$$ 
$\beta_1$ by 
$$\{z_1,z_2\} \mapsto \{\beta_0(z_1),\beta_0(z_2)\},$$ 
and $\beta_2$ by
$$\{f_1,\dots,f_n\} \mapsto \{\beta_1(f_1),\dots,\beta_1(f_n)\}.$$ The proof showing that $\beta$ is an isomorphism is identical to the proof given in Theorem~\ref{thm: isomorphisms between wythoffians}. 
\end{proof}
\smallskip

Thus, the snubs of orientable regular polyhedra or chiral polyhedra lie cellularly on the same surface as their underlying polyhedron, and for any snub of a finite non-orientable regular polyhedron $P$, we can find a unique orientable regular map whose snub is combinatorially isomorphic to the snub of $P$. Sometimes this orientable regular map is realizable as a polyhedron in $\mathbb{E}^3$. An example of this occurrence was mentioned in Remark~\ref{hemicubesnub}:\ the snubs of the hemi-cube and cube are isomorphic, $S_{\{4,3\}_3} \cong S_{\{4,3\}}$.
\smallskip

\section{Most uniform polyhedra of snub type are snubs} 
\label{completenessarguments}

Recall that a uniform polyhedron $P$ is said to be of \textit{snub type} $p.3.3.q.3$ if each vertex is surrounded, in cyclic order and up to reversal of orientation, by a regular $p$-gon, two regular triangles, a regular $q$-gon, and another regular triangle, allowing $p$ or $q$ to be 3 or $\infty$ or to coincide. In this section we will explore which uniform polyhedra of snub type (with $q<\infty$) arise as snubs of some regular or chiral polyhedron of type $\{p,q\}$ in $\mathbb{E}^3$. \smallskip

We first dispose of the vertex symbol $3.3.3.3.3$. There are just two uniform polyhedra of this type, and both of them are snubs.

\begin{proposition}
\label{all3}
The icosahedron $\{3,5\}$ and the great icosahedron $\{3,\frac{5}{2}\}$ are the only uniform polyhedra of snub type $3.3.3.3.3$ in $\mathbb{E}^3$. Both are snubs of the regular tetrahedron. 
\end{proposition}

\begin{proof}
A uniform polyhedron $P$ in $\mathbb{E}^3$ of snub type has 5-valent vertices. If its vertex symbol is $3.3.3.3.3$, then it has a 
Schl\"afli symbol, $\{3,5\}$, and is combinatorially an icosahedron or a hemi-icosahedron and thus is finite. As $P$ is finite and vertex-transitive, its vertices lie on a sphere $S$ about its center. Also, as the edges of a uniform polyhedron are congruent, the vertices adjacent to a given vertex $x$ all lie on a sphere $S'$ centered at $x$ and thus form a planar vertex figure with vertices on the circle $S\cap S'$. As the faces of $P$ are regular triangles, the vertex figure at $x$ is an equilateral pentagon with vertices inscribed in a circle. There are just two possibilities for the vertex figure:\ a regular pentagon $\{5\}$ or a regular pentagram $\{\frac{5}{2}\}$. As $P$ is vertex-transitive, the vertex figures at different vertices are congruent, leaving only two possibilities for $P$, the icosahedron and great icosahedron. Chiral polyhedra are infinite and cannot have finite snubs. As all entries in the vertex symbol are $3$'s, only the tetrahedron can have an icosahedron or great icosahedron as a snub. Both the (regular) icosahedron and the (regular) great icosahedron are indeed snubs of the tetrahedron.
\end{proof}
\smallskip

\begin{figure}[ht]
    \centering
    \includegraphics[height=6cm]{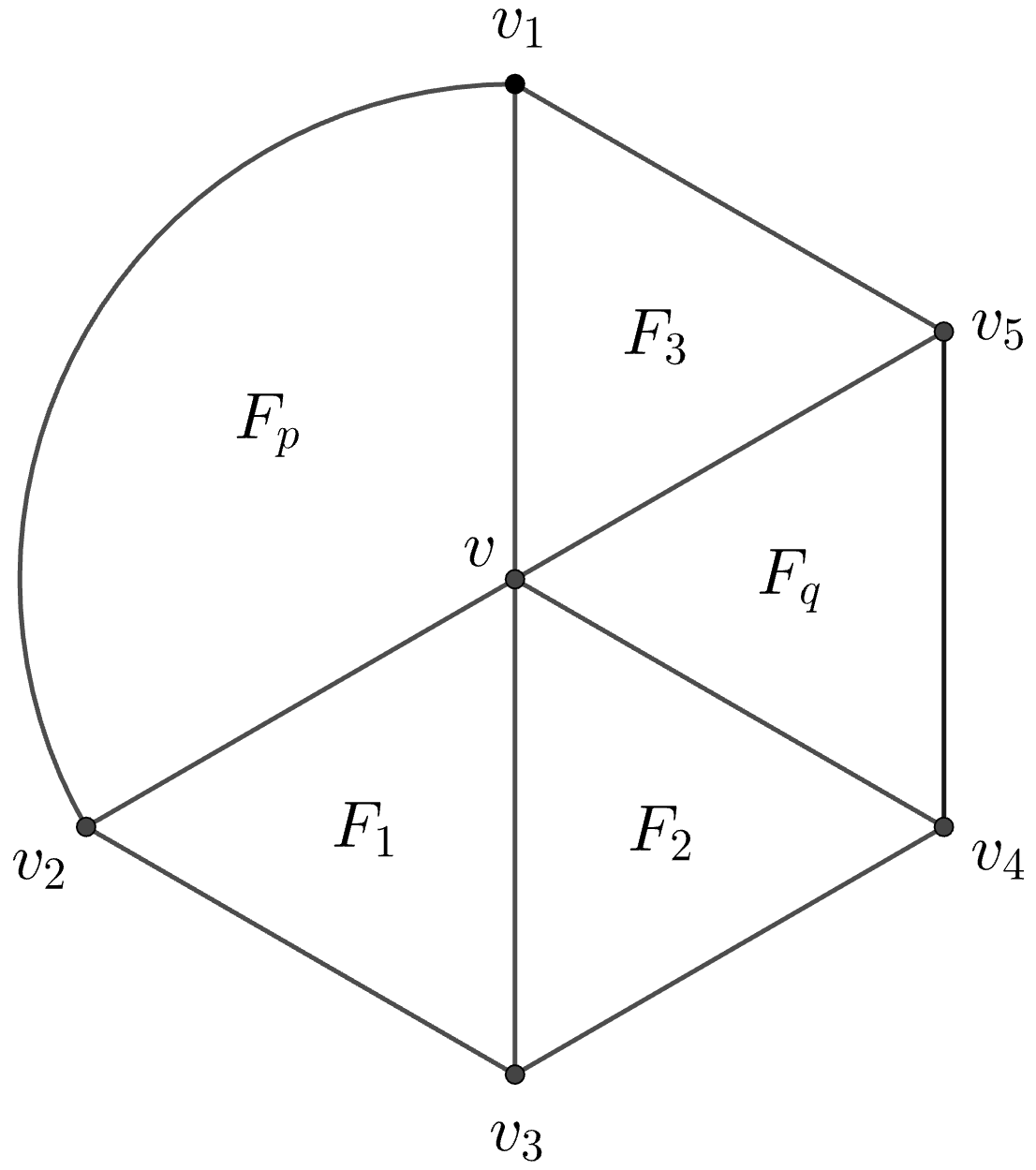}
    \caption{Adjacent vertices and faces at vertex $v$}
    \label{fig: p is q is 3 uniform}
\end{figure}

Next we turn to vertex symbols with four $3$'s. If $q=3$ and $p\neq 3$ (say), so that the type is $p.3.3.3.3$, we can designate some of the triangular faces of $P$ as ``special triangles" and denote them $3^*$, so that each vertex lies in exactly one special triangle, in a manner consistent with the vertex symbol $p.3.3.3^*.3$. To see this, consider the neighborhood of a vertex $v$ as in Figure~\ref{fig: p is q is 3 uniform}. We will analyze what faces (not shown in the figure) can be incident to the edges $\{v_2,v_3\}, \{v_3,v_4\},\{v_4,v_5\}$ and $\{v_5,v_1\}$. First, notice that $\{v_5,v_1\}$ and $\{v_2,v_3\}$ must both lie in two triangles. If both $\{v_3,v_4\}$ and $\{v_4,v_5\}$ also lie in two triangles, then we break the vertex symbol at the new vertices of these triangles. Thus, one of $F_q$ or $F_2$ must neighbor a $p$-face, and it follows that the other must neighbor a triangle. Without loss of generality, assume that the face surrounded by triangles is $F_q$. Thus, $F_q$ is the only triangle containing $v$ which does not share an edge with a $p$-gon. We will call such a triangle a \textit{special triangle}. This feature occurs at every vertex of $P$, that is, every vertex has a vertex symbol $p.3.3.3^*.3$ where $3^*$ denotes a special triangle. Note that the elements of the symmetry group $G(P)$ map special triangles to special triangles. \smallskip

Note that the vertex stabilizers of a uniform polyhedron $P$ of snub type $p.3.3.q.3$ in the symmetry group $G(P)$ have orders $1$ or $2$, unless $p=q=3$ (and $P$ is $\{3,5\}$ or $\{3,\frac{5}{2}\}$, by Proposition~\ref{all3}). This is forced by the vertex symbol, as well as considerations as above if $p=3$ or $q=3$. Moreover, if the vertex-stabilizers are non-trivial, then $p=q$ and the $p$-face and $q$-face at a vertex are congruent (under an element of the vertex-stabilizer); however, these  conditions may not be sufficient to guarantee a non-trivial vertex stabilizer.\smallskip

As we saw in previous sections, the snubs of regular or chiral polyhedra often give rise to uniform polyhedra of snub type. We now discuss the converse: which uniform polyhedra $P$ of snub type are snubs? We find that the answer is ``most", although this is hard to quantify exactly.

\begin{figure}[ht]
    \centering
    \includegraphics[height=6cm]{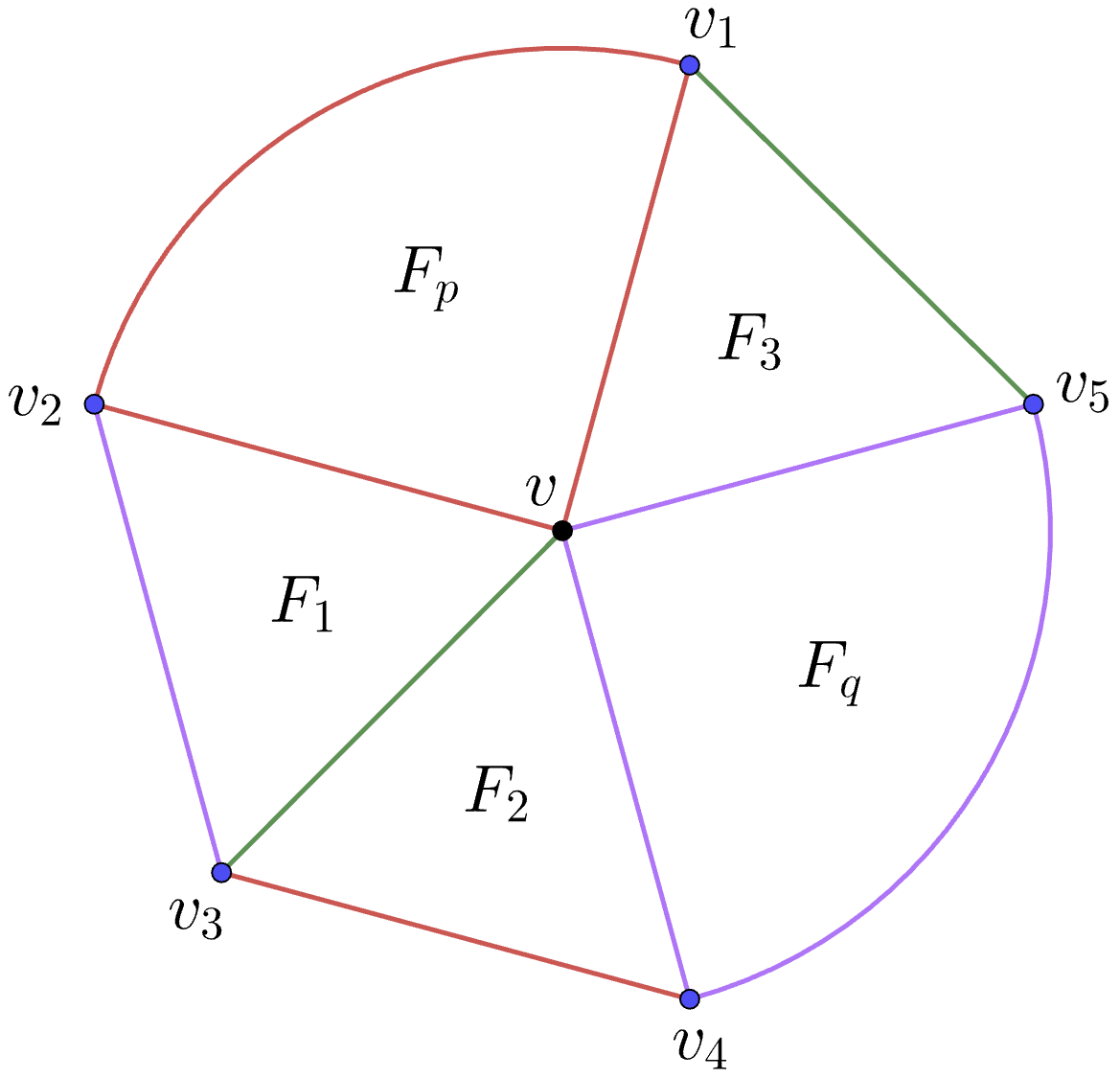}
    \caption{Faces and vertices around $v$}
    \label{fig: uniform reference}
\end{figure}

Proceeding from a uniform polyhedron $P$ of snub type $p.3.3.q.3$ we first identify suitable generators of $G(P)$, or of a suitable subgroup of $G(P)$,  which later will serve as distinguished generators when this group is considered as the combinatorial rotation subgroup of the symmetry group of a regular or chiral polyhedron. \smallskip

\begin{theorem}
\label{existsigmas}
Let $P$ be a uniform polyhedron in $\mathbb{E}^3$ of snub type $p.3.3.q.3$. Suppose that the $p$-gonal faces and the $q$-gonal faces are not congruent (this holds trivially if $p\neq q$). Let $v$ be a vertex of $P$, and let the faces around $v$ be labeled as in Figure~\ref{fig: uniform reference}. Then $G(P)$ contains elements $\sigma_1,\sigma_2$ (which are unique) with the following properties:
\noindent\begin{itemize}
\item[(a)] If $p<\infty$, then $\sigma_1$ cycles through the vertices of $F_p$ and has order $p$. If $p=\infty$, then $\sigma_1$ shifts the vertices of $F_p$ one step along and has infinite order.
\item[(b)] If $q<\infty$, then $\sigma_2$ cycles through the vertices of $F_q$ and has order $q$. If $q=\infty$, then $\sigma_2$ shifts the vertices of $F_q$ one step along and has infinite order.
\item[(c)] $\sigma_0:=\sigma_1\sigma_2$ is an involution.
\item[(d)] $G(P)=\langle\sigma_1,\sigma_2\rangle$.
\end{itemize}
\end{theorem}

\begin{proof} 
We present the proof under the assumption that $p\neq q$ and $p,q>3$. Very similar arguments will also work in the cases when $p=q$ and the $p$-gons and $q$-gons are incongruent, or when $p\neq 3$ and $q=3$ (then the special triangles are the $q$-gons). Note that the case $p=q=3$ is excluded by our assumption on non-congruence (see Proposition \ref{all3}). \smallskip

Consider the neighborhood of vertex $v$ as described in Figure \ref{fig: uniform reference}. Since $P$ is uniform, $G(P)$ acts transitively on the vertices of $P$ and by the assumptions on $p,q$, the vertex-stabilizers are trivial. Thus, there exists (a unique) $\sigma_1 \in G(P)$ such that $\sigma_1(v) = v_2$. Note that since each vertex of $P$ lies in exactly one $p$-gon, and $F_p$ is the $p$-gon containing $v$ and $v_2$, we must have $\sigma_1(F_p)=F_p$. Then, viewed as a combinatorial automorphism of $F_p$, the element $\sigma_1$ can only cycle or shift the vertices of $F_p$ one step along or be an involution interchanging $v$ and $v_2$. The latter is impossible since then $\sigma_1$ would fix a vertex, $v_3$.\smallskip

An analogous argument shows that there exists (a unique) $\sigma_2 \in G(P)$ mapping $v$ to $v_5$ such that $\sigma_2$ either cycles through or shifts (by one step) the vertices of $F_q$. \smallskip

Now let $\sigma_0 := \sigma_1 \sigma_2$. Then, using the notation of Figure~\ref{fig: uniform reference} and bearing in mind that $\sigma_1$ and $\sigma_2$ preserve adjacency of faces, we can conclude that  $\sigma_1\sigma_2(v) = \sigma_1(v_5) = v_3$ and $\sigma_1\sigma_2(v_3) = \sigma_1(v_1) = v$. Thus, $\sigma_0^2(v) = v$ on purely combinatorial grounds, and thus $\sigma_0^2 = 1$. Thus $\sigma_0$ has order 2.  \smallskip

The last part of the theorem will follow easily once we have established  that $A:=\langle\sigma_1,\sigma_2\rangle$ acts transitively on the  vertices. In fact, as a vertex-transitive subgroup of a simply vertex-transitive group, $G(P)$, the group $A$ must coincide with $G(P)$. The proof of the vertex-transitivity of $A$ relies on the observation  that for any two $p$-gonal or $q$-gonal faces $H$ and $F$ there exists a finite sequence of $p$-gonal or $q$-gonal faces and of vertices, 
$$H=H_0,w_1,H_1,w_2,H_2,\ldots,w_{k-1},H_{k-1},w_k,H_k = F,$$
in which successive faces $H_{i}$ and $H_{i+1}$ share a vertex, $w_{i+1}$, and $p$-gons and $q$-gons alternate. We call such a sequence a $(p,q)$-{\it path\/} of length $k$ joining $H$ to $F$.
\smallskip

First we show that for any $p$-gonal or $q$-gonal face $F$ there exists an element $\sigma=\sigma_F\in A$ which fixes $F$ and cycles through (or shifts) the vertices of $F$ in order. The point here is that $\sigma\in A$. In fact, under the assumptions on $p,q$ the vertex-transitivity of $G(P)$ also implies the transitivity on the $p$-gonal faces and on the $q$-gonal faces, and so the existence of such an element $\sigma$ in $ G(P)$ is clear from the existence of $\sigma_1$ and $\sigma_2$. Now to prove that $\sigma$ can be chosen in $A$, we consider $(p,q)$-paths from $F_q$ to faces $F$ and prove the statement by induction on the length of the shortest such $(p,q)$-path. If $F=F_q$ we may take $\sigma:=\sigma_2$. Now, proceeding inductively, if $$F_q=H_0,w_1,H_1,w_2,H_2,\ldots,w_{k-1},H_{k-1},w_k,H_k = F$$ is a shortest $(p,q)$-path from $F_q$ to $F$, then the sub-path joining $F_q$ to $H_{k-1}$ clearly is a shortest $(p,q)$-path from $F_q$ to $H_{k-1}$ and thus the inductive hypothesis provides an element $\sigma_{H_{k-1}}\in A$ which fixes $H_{k-1}$ and cycles through (or shifts) the vertices of $H_{k-1}$ in order. In particular, since $H_{k-1}$ and $F$ share a vertex, $F=\sigma_{H_{k-1}}^j(H_{k-2})$ for some $j$. Then $\sigma_{F} := \sigma_{H_{k-1}}^{j} \sigma_{H_{k-2}}\sigma_{H_{k-1}}^{-j}$ gives us the desired element in $A$ for $F$.\smallskip

Finally, to establish the vertex-transitivity of $A$, let $w$ be a vertex and $F$ be a $p$-gon or a $q$-gon at $w$. Connect $v$ and $w$ via a $(p,q)$-path from $F_q$ to $F$ as above; bear in mind that $v$ is a vertex of $F_q$. Also, set $w_0:=v$ and $w_{k+1}:=w$. Since $w_i$ and $w_{i+1}$ both lie in $H_{i}$ there exists $\sigma_{H_{i}} \in A$ such that $\sigma_{H_{i}}^j(w_i) = w_{i+1}$ for some $j$ (depending on $i$). Thus, moving along the $(p,q)$-path from $v$ to $w$, we can map a vertex to the next vertex by an element of $A$, so that the composition of all these mappings takes $v$ to $w$. Thus $A$ is vertex-transitive. \smallskip
\end{proof}
\smallskip

The geometry of the generators $\sigma_1,\sigma_2$ of $G(P)$ in Theorem~\ref{existsigmas} is implied by the geometry of the uniform polyhedron $P$. In particular, $\sigma_1$ is a rotation or rotatory reflection if $p<\infty$, or a screw translation, glide reflection, or translation if $p=\infty$; a  similar statement holds for $\sigma_2$ and $q$. The involution $\sigma_0$ is a half-turn or a point reflection, and  interchanges $v$ and $v_3$. Note that $\sigma_0$ cannot be a plane reflection as otherwise it would fix a vertex, $v_2$. \smallskip

\begin{remark}
\label{remarkone}
Let $P$ be a uniform polyhedron of snub type $p.3.3.q.3$. If $p=q>3$ and the $p$-gonal faces and $q$-gonal faces of $P$ are congruent, then the vertex-stabilizers in $G(P)$ have orders 1 or 2. We do not know if uniform polyhedra with congruent $p$-gons and $q$-gons and trivial vertex stabilizers actually exist. However, if the vertex-stabilizers in $G(P)$ are non-trivial then the $p$-gons and $q$-gons at a vertex are congruent under an element of the stabilizer of that vertex. In this case examples do exist and similar arguments as in the proof of Theorem~\ref{existsigmas} show that $\sigma_1$ and $\sigma_2$ exist here as well (satisfying (a), (b), and (c)) and are unique (!), and that $A:=\langle\sigma_1,\sigma_2\rangle$ is a vertex-transitive subgroup of $G(P)$ of index $1$ or $2$. The snub of the square tessellation $\{4,4\}$ is an example where the index is~2 (here the vertex stabilizers in $A$ are trivial). 
\end{remark}
\smallskip 

The existence of the generators $\sigma_1,\sigma_2$ will enable us to establish the remarkable fact that most uniform polyhedra of snub type are snubs of regular or chiral polyhedra. Suppose $P$ is a uniform polyhedron of snub type $p.3.3.q.3$ with $q<\infty$, let $\sigma_1,\sigma_2$ be as in Theorem~\ref{existsigmas} or as in the latter case of Remark~\ref{remarkone}, and let $A:=\langle\sigma_1,\sigma_2\rangle$. In most cases we actually have $A=G(P)$, although our arguments will not rely on this. Again, let $v$ be a vertex of $P$ (with a neighborhood as in Figure~\ref{fig: uniform reference}), and let 
$u$ be the center of the (finite) face $F_q$. We now generate a regular or chiral geometric polyhedron $Q$ from $u$ and $A$ as an orbit structure, and then show that $Q$ has $P$ as a snub. To begin with, we let
\begin{equation}
\label{conversebase}
K_0:=u,\; K_1:=\{u,\sigma_1\sigma_2(u)\} = \{u,\sigma_1(u)\},\; 
K_2:=\{\sigma(u)|\sigma\in \langle \sigma_1 \rangle\},
\end{equation}
respectively, denote the base vertex, base edge, and base face of $Q$. Then, the vertex set, edge set, and face set of $Q$ are given by the orbits of the base vertex, base edge, and base face under $A$. In the following theorem we prove that, under relatively mild conditions, $Q$ indeed is a geometric polyhedron with the desired features. We let $A_{K_0}$, $A_{K_1}$, and $A_{K_2}$ denote the stabilizers of $K_0$, $K_1$, and $K_2$ in $A$, respectively, and observe that
\begin{equation}
\label{stabcond}
\langle \sigma_2 \rangle \subseteq A_{K_0},\; 
\langle \sigma_1\sigma_2 \rangle\subseteq A_{K_1},\;
\langle \sigma_1 \rangle\subseteq A_{K_2}.
\end{equation}
The occurrence of equality in these inclusions turns out to be nearly sufficient to conclude that $Q$ has the desired properties.
\smallskip

\begin{theorem}
\label{thm: completeness condition}
Let $P$ be a uniform polyhedron in $\mathbb{E}^3$ of snub type $p.3.3.q.3$ (with $q<\infty$). Suppose that the $p$-gonal faces and the $q$-gonal faces are not congruent. Assume further that a $p$-gonal face and a $q$-gonal face of $P$ cannot share more than one vertex of $P$. Let $v$ be a vertex of $P$, let $\sigma_1,\sigma_2$ be as in Theorem~\ref{existsigmas}, let $A:=\langle\sigma_1,\sigma_2\rangle$, and let 
$K_0$, $K_1$, $K_2$, and $Q$ be as above.  If $A_{K_0} = \langle \sigma_2 \rangle$, $A_{K_1} = \langle \sigma_1\sigma_2 \rangle$, and $A_{K_2} = \langle \sigma_1 \rangle$ (that is, the inclusions in (\ref{stabcond}) are  equalities), then $Q$ is a directly regular or chiral geometric polyhedron and $P=\mathcal{S}_Q(v)$. Thus $P$ is a snub of $Q$.
\end{theorem}
\smallskip

\begin{proof}
Since the vertex-stabilizers for $P$ in $G(P)$ are trivial, Theorem~\ref{existsigmas}(d) shows that $A=G(P)$. 

We first look at $Q$ combinatorially and show that $Q$ is an abstract polyhedron. The generators $\sigma_1,\sigma_2$ of $A$ by definition satisfy the relations $$\sigma_1^p = \sigma_2^q = (\sigma_1\sigma_2)^2 = 1,$$ 
and since $F_p$ and $F_q$ only intersect at $v$, it follows that $A$ has the intersection property,
$$\langle \sigma_1 \rangle \cap \langle \sigma_2 \rangle 
=\langle \sigma_1 \rangle \cap \langle \sigma_1\sigma_2 \rangle 
=\langle \sigma_2 \rangle \cap \langle \sigma_1\sigma_2 \rangle = \{1\}.$$ 
Then, by \cite{Chiral abstract polytopes}, we can generate from $A$ a directly regular or chiral abstract polyhedron $\mathcal{Q}$ whose vertices, edges, and faces  are the left cosets in $A$ of the subgroups $A^0:=\langle \sigma_2 \rangle$, $A^1:=\langle \sigma_1\sigma_2 \rangle$, and $A^2:=\langle \sigma_1 \rangle$, respectively, with incidence given by non-empty intersection of cosets (that is, $\varphi A^i \leq \tau A^j$ in $\mathcal{Q}$ if and only if $i\leq j$ and $\varphi A^i \cap \tau A^j \neq \emptyset$). Further, $A$ is a subgroup of the automorphism group of $\mathcal{Q}$ and acts on $\mathcal{Q}$ with two flag-orbits such that adjacent flags are in distinct orbits. \smallskip

Next we establish that the mapping given by 
$$\begin{array}{rlcl}
f: & \mathcal{Q} &\mapsto &Q \\
   & \varphi A^i &\mapsto &\varphi(K_i)\;\; (i=0,1,2;\,\varphi\in A)\\
\end{array}$$
is a polytope isomorphism between $\mathcal{Q}$ and the face set of $Q$. First note that our assumptions on the stabilizers $A_{K_i}$ allow us to prove that $f$ is a bijection. In fact, since $A_{K_i}=A^i$ for each~$i$, we have $\varphi(K_i)=\psi(K_i)$ if and only if $\varphi A_{K_i}=\psi A_{K_i}$; that is, if and only if $\varphi A^i =\psi A^i$. Thus $f$ is well-defined and injective. Clearly, $f$ is also surjective, so $f$ is bijective. Note that the group $A$ acts on both $\mathcal{Q}$ and $Q$ in an incidence preserving manner. It remains to show that both mappings $f$ and $f^{-1}$ preserve incidence. The following argument shows that $f$ itself preserves incidence. If $\varphi A^i \leq \tau A^j$ in $\mathcal{Q}$, then $i\leq j$ and $\varphi A^i \cap \tau A^j \neq \emptyset$, so there exists $\alpha \in A^i$ and $\beta \in A^j$ such that $\varphi \alpha = \tau \beta$; then 
$$f(\varphi A^i) = \varphi (K_i) = \varphi \alpha (K_i) = \tau \beta (K_i) \leq \tau \beta (K_j) = \tau (K_j) \leq f(\tau A^j),$$ 
as required. For the proof that $f^{-1}$ also preserves incidence we only need to consider faces $\varphi(K_i)$ of $Q$ with $\varphi(K_i)\leq K_j$ where $(i,j)=(0,1),(0,2),(1,2)$. The vertices of $K_1$ are $(\sigma_1\sigma_2)^{k}(K_0)$ for $k=0,1$, and the vertices and edges of $K_2$ are given by $\sigma_1^l(K_0)$ and $\sigma_1^l(K_1)$ for $l\geq 0$, respectively. Now bear in mind our assumptions on the stabilizers of the base faces $K_{l}$ of $Q$, and recall that $A_{K_l}=A^l$ for $l=0,1,2$. If $(i,j)=(0,1)$, it then follows that  
$\varphi\in (\sigma_1\sigma_2)^{k}A^{0}$ for some $k=0,1$ and thus
$$\varphi A^0\cap A^1 = (\sigma_1\sigma_2)^{k}\langle\sigma_2\rangle
\cap\langle\sigma_{1}\sigma_2\rangle\neq\emptyset.$$ If $(i,j)=(0,2)$, then 
$\varphi\in \sigma_1^{l}A^{0}$ for some $l\geq 0$ and thus 
$\varphi A^{0}\cap A^2 
= \sigma_1^{l}\langle\sigma_2\rangle \cap \langle\sigma_1\rangle\neq\emptyset$. Similarly, if $(i,j)=(1,2)$, then 
$\varphi A^{1}\cap A^2 
= \sigma_1^{l}\langle\sigma_1\sigma_2\rangle \cap \langle\sigma_1\rangle\neq\emptyset$. It follows that $\varphi A^i \leq A^j$ in $\mathcal{Q}$ in each case. This proves that $f^{-1}$ is also incidence preserving. \smallskip

Thus $Q$ is also an abstract polytope, isomorphic to $\mathcal{Q}$ under $f$. In particular, $Q$ is a geometric polytope (a faithful realization of $\mathcal{Q}$ in $\mathbb{E}^3$ in the sense of \cite{41}) and $A$ is a subgroup of the symmetry group of $Q$ acting on $Q$ with two flag-orbits such that adjacent flags are in distinct orbits. The generators $\sigma_1,\sigma_2$ are the distinguished generators of $A$ corresponding to this action, and $A=G^+(Q)$. Further, by our assumptions on $P$, the stabilizer of $v$ in $A$ is trivial and thus $v$ satisfies the IPC for $G^+(Q)$.
Thus $v$ can serve as the initial vertex for a snub ${S}_Q(v)$ of $Q$. 
\smallskip

Finally, we claim that $P={S}_Q(v)$. In fact, using the notation of Section~\ref{snubcon} with $s_1=\sigma_1$, $s_2=\sigma_2$, and $s_0=\sigma_1\sigma_2$, the faces of ${S}_Q(v)$ can be seen to match up with the faces of $P$. In fact, from the initial vertex $v$ of the snub ${S}_Q(v)$ the base edges of ${S}_Q(v)$ are obtained as $\{v,s_i(v)\}$ for $i=0,1,2$ and thus are edges of $P$ emanating from vertex $v$ of $P$. Similarly, the base faces of ${S}_Q(v)$ are faces of $P$ containing the vertex $v$. The rest is accomplished via the action of $A$. This completes the proof.
\end{proof}
\smallskip

\begin{remark}
    Let $P$ be a uniform polyhedron of snub type $p.3.3.q.3$ (with $q<\infty$). If the $p$-gonal faces and $q$-gonal faces are congruent, the vertex stabilizers in $G(P)$ are non-trivial, and $A$ is an index $2$ subgroup of $G(P)$ (and thus the vertex stabilizers in $A$ are trivial), then the statement of Theorem \ref{thm: completeness condition} also holds and the proof carries over. We do not know of an example of a uniform polyhedron of snub type satisfying the conditions that the $p$-gonal faces are congruent to the $q$-gonal faces, the vertex stabilizers in $G(P)$ are non-trivial, and $A=G(P)$. 
\end{remark}

\begin{remark}
\label{stab-u}
Suppose again that $u\,(=K_0)$ is the center of the face $F_q$ at vertex $v$ of the uniform polyhedron $P$. The assumption of Theorem~\ref{thm: completeness condition} on $A_u\,(=A_{K_0})$ is rather weak,  and we know of no example where it is not satisfied. Clearly, if $\varphi\in A_u$, then $\varphi(F_q)$ is also a $q$-gonal face of $P$ centered at $u$ and may or may not coincide with $F_q$. Geometrically, the condition $A_{u} = \langle \sigma_2 \rangle$ simply means that $F_q$ is the only $q$-gonal face of $P$ centered at~$u$. In fact, if the condition on $A_u$ does not hold, then there is more than one $q$-gon centered at~$u$ and the $q$-gons of $P$ centered at $u$ overlap in a certain sense. In a geometric figure, this geometric property would be easy to spot. Note that the polyhedron $P$ being of snub type prevents $F_q$ from having a dihedral stabilizer in $G(P)$. 
\end{remark}
\smallskip

\section{Generators and fundamental regions for $G^+(P)$}
\label{secgplus}

The remaining three sections focus on the snubs of the finite regular polyhedra in $\mathbb{E}^3$. The snubs of the remaining regular polyhedra will be described in detail in~\cite{Sk2}.

For the construction of snubs from regular polyhedra $P$, we require explicit formulas for the generators $s_1,s_2$ (and $s_0=s_1s_2$) of the combinatorial rotation subgroup $G^+(P)$. We present the information in Tables~\ref{gplusfinite1}, \ref{gplusfinite2} in the form

\begin{center}
\begin{tabular}{ |c|c|c|c| } 
 \hline
 $P$ & $s_1$ & $s_2$ & $s_0$ \\
 \hline
\end{tabular}
\end{center}
\noindent
Our tables list only one polyhedron of a dual pair, as duals have the same symmetry groups.

\begin{table}
\begin{center}
\begin{tabular}{ |c|c|c|c| } 
 \hline
 \rule{0pt}{2.25ex} $P$ & $s_1$ & $s_2$ & $s_0$  \\ [0.25ex]
 \hline
 \rule{0pt}{2.25ex} $\{3,3\}$ & $(\frac{\beta}{2}+\frac{z}{\sqrt{2}},\frac{\beta}{2}-\frac{z}{\sqrt{2}},\frac{\alpha}{\sqrt{2}})$ & $(\frac{\alpha}{2}-\frac{z}{\sqrt{2}},-\frac{\alpha}{2}-\frac{z}{\sqrt{2}},-\frac{\beta}{\sqrt{2}})$ & $(y,x,-z)$  \\ [0.25ex]
 \hline
 \rule{0pt}{2.25ex} $\{4,3\}_3$ & $(-y,x,-z)$ & $(\frac{\alpha}{2}-\frac{z}{\sqrt{2}},-\frac{\alpha}{2}-\frac{z}{\sqrt{2}},-\frac{\beta}{\sqrt{2}})$ & $(\frac{\alpha}{2}+\frac{z}{\sqrt{2}},\frac{\alpha}{2}-\frac{z}{\sqrt{2}},\frac{\beta}{\sqrt{2}})$  \\ [0.25ex]
 \hline
 \rule{0pt}{2.25ex} $\{4,3\}$ & $(y,-x,z)$ & $(y,-z,-x)$ & $(-z,-y,-x)$  \\ [0.25ex]
 \hline
 \rule{0pt}{2.25ex} $\{6,3\}_4$ & $(-z,-x,-y)$ & $(y,-z,-x)$ & $(x,-y,z)$  \\ [0.25ex]
 \hline
 \rule{0pt}{2.25ex} $\{6,4\}_3$ & $(-z,-x,-y)$ & $(-y,x,z)$ & $(-z,y,-x)$  \\ [0.25ex]
 \hline
\end{tabular}
\end{center}
\caption{Generators $s_1,s_2, s_0$ of $G^+(P)$ for the non-icosahedral finite regular polyhedra $P$, with $\alpha:=x+y$, $\beta:=x-y$.}
\label{gplusfinite1}
\end{table}

\begin{table}
\begin{center}
\begin{tabular}{ |c|c|c| } 
 \hline
 \rule{0pt}{2.25ex} $P$ & $s_1$ & $s_0$  \\ [0.25ex]
 \rule{0pt}{2.25ex}  & $s_2$ &   \\ [0.25ex]
 \hline
 \rule{0pt}{2.25ex} $\{3,5\}$ & $(\alpha x +\frac{y}{2} +\beta z,\frac{x}{2}-\beta y-\alpha z, \alpha y - \beta x-\frac{z}{2})$ & $(x,-y,-z)$  \\ [0.25ex]
 \rule{0pt}{2.25ex}  & $(\alpha x-\frac{y}{2}+\beta z,\frac{x}{2}+\beta y-\alpha z,\beta x+\alpha y + \frac{z}{2})$ &   \\ [0.25ex]
 \hline
 \rule{0pt}{2.25ex} $\{10,5\}_3$ & $(\alpha x+\frac{y}{2}+\beta z,\beta y-\frac{x}{2}+\alpha z,\alpha y-\beta x -\frac{z}{2})$ & $(x,y,-z)$  \\ [0.25ex]
 \rule{0pt}{2.25ex}  & $(\alpha x-\frac{y}{2}+\beta z,\frac{x}{2}+\beta y-\alpha z, \beta x +\alpha y+\frac{z}{2})$ &   \\ [0.25ex]
 \hline
 \rule{0pt}{2.25ex} $\{10,3\}_5$ & $(\alpha x+\frac{y}{2}+\beta z,\beta y -\frac{x}{2}+\alpha z, \alpha y-\beta x -\frac{z}{2})$ & $(x,-y,z)$  \\ [0.25ex]
 \rule{0pt}{2.25ex}  & $(\alpha x+\frac{y}{2}-\beta z,\frac{x}{2}-\beta y+\alpha z,\beta x-\alpha y -\frac{z}{2})$ &   \\ [0.25ex]
 \hline
 \rule{0pt}{2.25ex} $\{5,\frac{5}{2}\}$ & $(\frac{x}{2} + \beta y + \alpha z, \beta x + \alpha y - \frac{z}{2}, \frac{y}{2} - \alpha x + \beta z)$ & $(x,-y,-z)$  \\ [0.25ex]
 \rule{0pt}{2.25ex}  & $(\frac{x}{2} - \beta y + \alpha z, \beta x - \alpha y - \frac{z}{2}, \alpha x + \frac{y}{2} - \beta z)$ &   \\ [0.25ex]
 \hline
 \rule{0pt}{2.25ex} $\{6,\frac{5}{2}\}$ & $(\frac{x}{2} + \beta y + \alpha z, \frac{z}{2} - \alpha y - \beta x, \frac{y}{2} - \alpha x + \beta z)$ & $(x,y,-z)$  \\ [0.25ex]
 \rule{0pt}{2.25ex}  & $(\frac{x}{2} - \beta y + \alpha z, \beta x - \alpha y - \frac{z}{2}, \alpha x + \frac{y}{2} - \beta z)$ &   \\ [0.25ex]
 \hline
 \rule{0pt}{2.25ex} $\{6,5\}$ & $(\frac{x}{2} + \beta y + \alpha z, \frac{z}{2} - \alpha y - \beta x, \frac{y}{2} - \alpha x + \beta z)$ & $(x,-y,z)$  \\ [0.25ex]
 \rule{0pt}{2.25ex}  & $(\frac{x}{2} + \beta y - \alpha z, \beta x + \alpha y + \frac{z}{2}, \alpha x - \frac{y}{2} + \beta z)$ &   \\ [0.25ex]
 \hline
 \rule{0pt}{2.25ex} $\{3,\frac{5}{2}\}$ & $(\alpha y - \beta x + \frac{z}{2}, \beta z - \frac{y}{2} - \alpha x, \frac{x}{2} - \beta y + \alpha z)$ & $(x,-y,-z)$  \\ [0.25ex]
 \rule{0pt}{2.25ex}  & $(\alpha y - \beta x - \frac{z}{2}, \alpha x + \frac{y}{2} + \beta z, \frac{x}{2} - \beta y - \alpha z)$ &   \\ [0.25ex]
 \hline
 \rule{0pt}{2.25ex} $\{\frac{10}{3},\frac{5}{2}\}$ & $(\alpha y - \beta x + \frac{z}{2}, \beta z - \frac{y}{2} - \alpha x, \beta y - \frac{x}{2} - \alpha z)$ & $(x,-y,z)$  \\ [0.25ex]
 \rule{0pt}{2.25ex}  & $(\alpha y - \beta x - \frac{z}{2}, \alpha x + \frac{y}{2} + \beta z, \frac{x}{2} - \beta y - \alpha z)$ &   \\ [0.25ex]
 \hline
 \rule{0pt}{2.25ex} $\{\frac{10}{3},3\}$ & $(\alpha y - \beta x + \frac{z}{2}, \beta z - \frac{y}{2} - \alpha x, \beta y - \frac{x}{2} - \alpha z)$ & $(x,y,-z)$  \\ [0.25ex]
 \rule{0pt}{2.25ex}  & $(\frac{z}{2} - \alpha y - \beta x, \alpha x - \frac{y}{2} - \beta z, \frac{x}{2} + \beta y + \alpha z)$ &   \\ [0.25ex]
 \hline
\end{tabular}
\end{center}
\caption{Generators $s_1,s_2, s_0$ of $G^+(P)$ for the icosahedral finite regular polyhedra $P$, with $\alpha:= \frac{1+\sqrt{5}}{4}$, $\beta:= \frac{\sqrt{5}-1}{4}$.}
\label{gplusfinite2}
\end{table}

Fundamental regions for the groups $G^{+}(P)$ can be found by the method described at the end of Section \ref{regpo}:\ choose a point $w$ not fixed by any nontrivial element of $G^+(P)$, and then take as the fundamental region the interior of the Dirichlet-Voronoi region (cell) of $w$ in the Dirichlet-Voronoi tessellation of $\mathbb{E}^3$ determined by the point orbit of $w$ under $G^+(P)$. In Table~\ref{gplusfunregfinite}, we present the point $w$ used to generate the Dirichlet-Voronoi tessellation and thus the fundamental region $D$ for $G^+(P)$. The fundamental region $D$ itself is the interior of the convex cone with apex (0,0,0) spanned by the non-zero vectors in the set $V$ of Table~\ref{gplusfunregfinite}. 

\begin{table}
$$\begin{array}{cc}
{\begin{tabular}{|c||c|c|} 
 \hline
 \rule{0pt}{2.5ex} $P$ & $w$ & $V$ \\ [0.5ex] 
 \hline\hline
 \rule{0pt}{2.5ex} $\{3,3\}$ & $(\frac{4}{9},0,-\frac{2}{9\sqrt{2}})$ & $V_1$ \\ [0.5ex]
 \hline
 \rule{0pt}{2.5ex} $\{4,3\}_3$ & $(\frac{11}{24}, \frac{1}{8}, -\frac{\sqrt{2}}{12})$ & $V_2$ \\ [0.5ex]
 \hline
 \rule{0pt}{2.5ex} $\{4,3\}$ & $(\frac{1}{3},\frac{1}{3},-\frac{2}{3})$ & $V_3$ \\ [0.5ex]
 \hline
 \rule{0pt}{2.5ex} $\{6,3\}_4$ & $(\frac{1}{3},\frac{1}{3},-\frac{2}{3})$ & $V_3$ \\ [0.5ex]
 \hline
 \rule{0pt}{2.5ex} $\{6,4\}_3$ & $(\frac{1}{2},\frac{1}{4},-\frac{3}{4})$ & $V_4$ \\ [0.5ex]
 \hline
\rule{0pt}{2.5ex} $\{10,5\}_3$ & $(\frac{5+4\sqrt{5}}{12},\frac{1+\sqrt{5}}{24},\frac{1}{4})$ & $V_6$ \\ [0.5ex]
 \hline
 \rule{0pt}{2.5ex} $\{10,3\}_5$ & $(\frac{5+4\sqrt{5}}{12},\frac{1+\sqrt{5}}{24},\frac{1}{4})$ & $V_6$ \\ [0.5ex]
 \hline
 \rule{0pt}{2.5ex} $\{6,\frac{5}{2}\}$ & $(\frac{5+4\sqrt{5}}{12},\frac{1+\sqrt{5}}{24},\frac{1}{4})$ & $V_6$ \\ [0.5ex]
 \hline
 \rule{0pt}{2.5ex} $\{6,5\}$ & $(\frac{5+4\sqrt{5}}{12},\frac{1+\sqrt{5}}{24},\frac{1}{4})$ & $V_6$ \\ [0.5ex]
 \hline
 \rule{0pt}{2.5ex} $\{\frac{10}{3},3\}$ & $(\frac{5+4\sqrt{5}}{12},\frac{1+\sqrt{5}}{24},\frac{1}{4})$ & $V_6$ \\ [0.5ex]
 \hline
 \rule{0pt}{2.5ex} $\{\frac{10}{3},\frac{5}{2}\}$ & $(\frac{5+4\sqrt{5}}{12},\frac{1+\sqrt{5}}{24},\frac{1}{4})$ & $V_6$ \\ [0.5ex]
 \hline
 \rule{0pt}{2.5ex} $\{3,5\}$ & $(\frac{7+5\sqrt{5}}{18},\frac{1+\sqrt{5}}{18},\frac{1}{3})$ & $V_5$ \\ [0.5ex]
 \hline
 \rule{0pt}{2.5ex} $\{5,\frac{5}{2}\}$ & $(\frac{7+5\sqrt{5}}{18},\frac{1+\sqrt{5}}{18},\frac{1}{3})$ & $V_5$ \\ [0.5ex]
 \hline
 \rule{0pt}{2.5ex} $\{3,\frac{5}{2}\}$ & $(\frac{7+5\sqrt{5}}{18},\frac{1+\sqrt{5}}{18},\frac{1}{3})$ & $V_5$ \\ [0.5ex]
 \hline
\end{tabular}} &
{\begin{tabular}{|c||l|}
 \hline
 \rule{0pt}{2.5ex} $V_1$ & $(1,0,-\frac{1}{\sqrt{2}}),(\frac{1}{2},\frac{1}{2},0),(\frac{1}{2},-\frac{1}{2},0),(\frac{1}{3},0,\frac{1}{3\sqrt{2}})$ \\ [0.5ex]
 \hline
 \rule{0pt}{2.5ex} $V_2$ & $(1,0,-\frac{1}{\sqrt{2}}),(\frac{1}{2},\frac{1}{2},0),(\frac{1}{3},0,\frac{1}{3\sqrt{2}})$ \\ [0.5ex]
 \hline
 \rule{0pt}{2.5ex} $V_3$ & $(0,0,-1),(1,0,-1),(1,1,-1),(0,1,-1)$ \\ [0.5ex]
 \hline
 \rule{0pt}{2.5ex} $V_4$ & $(1,1,-1), (1,0,-1), (0,0,-1)$ \\ [0.5ex]
 \hline
 \rule{0pt}{2.5ex} $V_5$ & $(\frac{1+\sqrt{5}}{2},0,1),(\frac{1+\sqrt{5}}{2},0,0),(\frac{2+\sqrt{5}}{3},\frac{1+\sqrt{5}}{6},0)$,\\[.5ex] 
 &$(\frac{3+\sqrt{5}}{4},\frac{1+\sqrt{5}}{4},\frac{1}{2})$ \\ [0.5ex]
 \hline
 \rule{0pt}{2.5ex} $V_6$ & $(\frac{1+\sqrt{5}}{2},0,1), (\frac{1+\sqrt{5}}{2},0,0), (\frac{2+\sqrt{5}}{3},\frac{1+\sqrt{5}}{6},0)$ \\ [0.5ex]
 \hline
\end{tabular}}
\end{array}$$
\caption{Fundamental regions for $G^{+}(P)$ for the finite regular polyhedra $P$.}
\label{gplusfunregfinite}
\end{table}

\section{Snubs of the finite regular polyhedra} 
\label{finite uniform polyhedra}

In this section, as an illustration of our methods, we apply the snub construction to the eighteen finite regular polyhedra $P$,  with an initial vertex $v$ chosen either in the (open) fundamental region $D$ of $G^+(P)$ or as a point on the boundary of $D$ fixed by one of the generators $s_0$, $s_1$, or $s_2$. For a discussion of the snubs for the remaining regular polyhedra we refer to \cite{Sk2}. Our computations are based on the generators $s_1,s_2,s_0$ of $G^{+}(P)$ in Tables~\ref{gplusfinite1},~\ref{gplusfinite2} and the fundamental regions $D$ of $G^{+}(P)$ in Table~\ref{gplusfunregfinite}. Recall our notation ${S}_P^i(v)$ for a snub $S_P(v)$ derived from an initial point $v$ fixed by $s_i$ ($i=0,1,2$). 

If $P$ is either a Platonic solid or a Kepler-Poinsot polyhedron and $v$ satisfies the IPC of (3) for $G^+(P)$, then $S_P(v)$ corresponds to an already well-known snub. In particular, by choosing $v$ as an appropriate acceptable solution to the uniformity equations for $S_P(v)$ we obtain the following uniform polyhedra: a regular icosahedron or a regular great icosahedron (from the tetrahedron), a snub cube (from the cube and the octahedron), a snub dodecahedron (from the dodecahedron and the icosahedron), a snub dodecadodecahedron, or an inverted snub dodecadodecahedron  (from $\{\frac{5}{2},5\}$ and $\{5,\frac{5}{2}\}$), and a great snub icosidodecahedron, a great inverted snub icosidodecahedron, or a great retrosnub icosidodecahedron (from $\{3,\frac{5}{2}\}$ and $\{\frac{5}{2},3\}$). See \cite{11,Wiki} for the names of these polyhedra. Furthermore, the degenerate snubs $S_P^0$, $S_P^1$, and $S_P^2$ of $P$ are the medial of $P$, the dual of $P$, and $P$ itself, respectively. \smallskip

The remaining nine finite regular polyhedra are the Petrie-duals of the Platonic solids and the Petrie-duals of the Kepler-Poinsot polyhedra. Amongst these, the only polyhedron which is orientable is the Petrie-dual of the cube. Furthermore, the generator $s_0$ is a reflection for each of these nine polyhedra, meaning that their snubs cannot be made uniform by Lemma~\ref{lemma: uniformity conditions}. If $P$ is one of these nine polyhedra, then $S_P^1$ collapses to a point (the origin) and $S_P^2$ is similar to $P$ itself. The degenerate snub $S_P^0$ is of most interest, as here we obtain new classes of finite uniform polyhedra; we present these in Section~\ref{new finite uniform polyhedra}.

We now present the snubs (with $v\in D$) generated from these remaining regular polyhedra, beginning with the snub of $\{4,3\}_3$, the Petrie-dual of $\{3,3\}$. 
\bigskip

\vspace{0.5em}
\begin{centering}

\begin{tblr}{cells={valign=m,halign=c},row{1}={rowsep=4pt},row{2}={rowsep=4pt},row{3}={rowsep=4pt},row{4}={rowsep=4pt},hlines,vlines,column{1}={3.25cm}}
\SetCell[r=4]{} \includegraphics[height=2.8cm,valign=c]{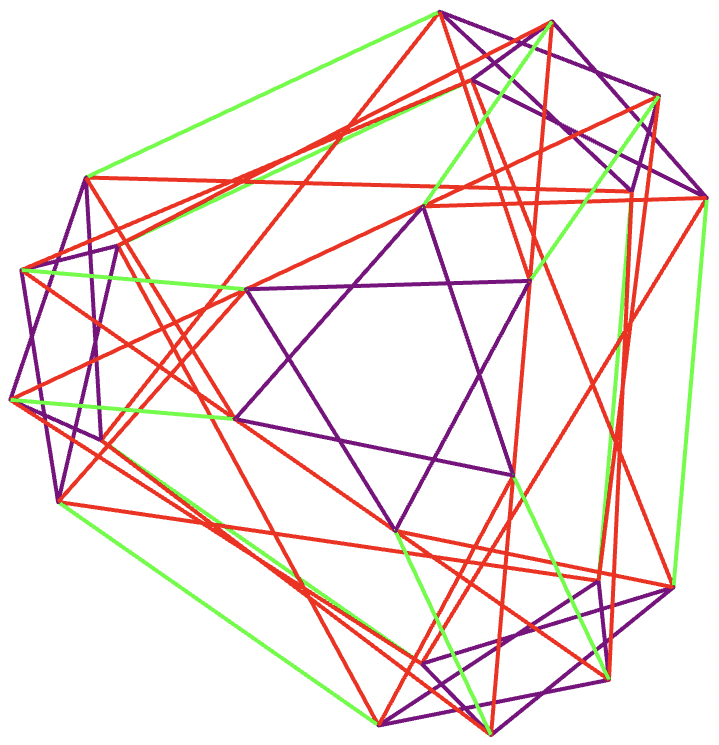}
    & $S_{\{4,3\}_3}(v)$ with $v\in D$ \\
    & $4_s.3.3.3.3$ or $4_c.3.3.3.3$ \\
    & $(f_0,f_1^0, f_1^1, f_1^2, f_2^0, f_2^1, f_2^2) = (24,12,24,24,24,6,8)$ \\
    & $\chi(S_P) = 2$ \\ 
\end{tblr}

\end{centering}
\vspace{0.5em}

The second row in our tables lists all achievable vertex symbols. Recall our convention to record a regular convex $p$-gon, a regular star $p$-gon (with density $d$), or a regular skew $p$-gon as $p_c$, $\frac{p}{d}$, or $p_s$, respectively. The third row in the tables presents the detailed $f$-vector of the snub, where $f_i^j$ is the number of $i$-faces of type $j$. The fourth row gives the Euler characteristic of the snub $S_P(v)$ (with $v \in D$). 

For $P=\{4,3\}_3$, since $s_1$ is a rotatory reflection, depending on the choice of $v$ (still in $D$) we can either obtain skew or convex faces of type 1. The hemi-cube $\{4,3\}_3$ is a non-orientable regular map on the projective plane whose two-fold cover is the cube $\{4,3\}$ on the sphere. Thus, by Theorem~\ref{thm: isomorphisms between wythoffians}, $S_{\{4,3\}_3}(v) \cong S_{\{4,3\}}(w)$ where $w$ lies in the fundamental region of $G^+(\{4,3\})$. By Remark~\ref{hemicubesnub}, $S_{\{4,3\}_3}(v)$ cannot be made uniform geometrically, but $S_{\{4,3\}}(w)$ can be made uniform via the realization as the standard snub cube.
\smallskip

We now present the remaining snubs in the same format as above. We often refer to the notation of~\cite{Conder} to name the underlying surface map.
\bigskip

\vspace{0.5em}
\begin{centering}

\begin{tblr}{cells={valign=m,halign=c},row{1}={rowsep=4pt},row{2}={rowsep=4pt},row{3}={rowsep=4pt},row{4}={rowsep=4pt},hlines,vlines,column{1}={3.25cm}}
\SetCell[r=4]{} \includegraphics[height=2.8cm,valign=c]{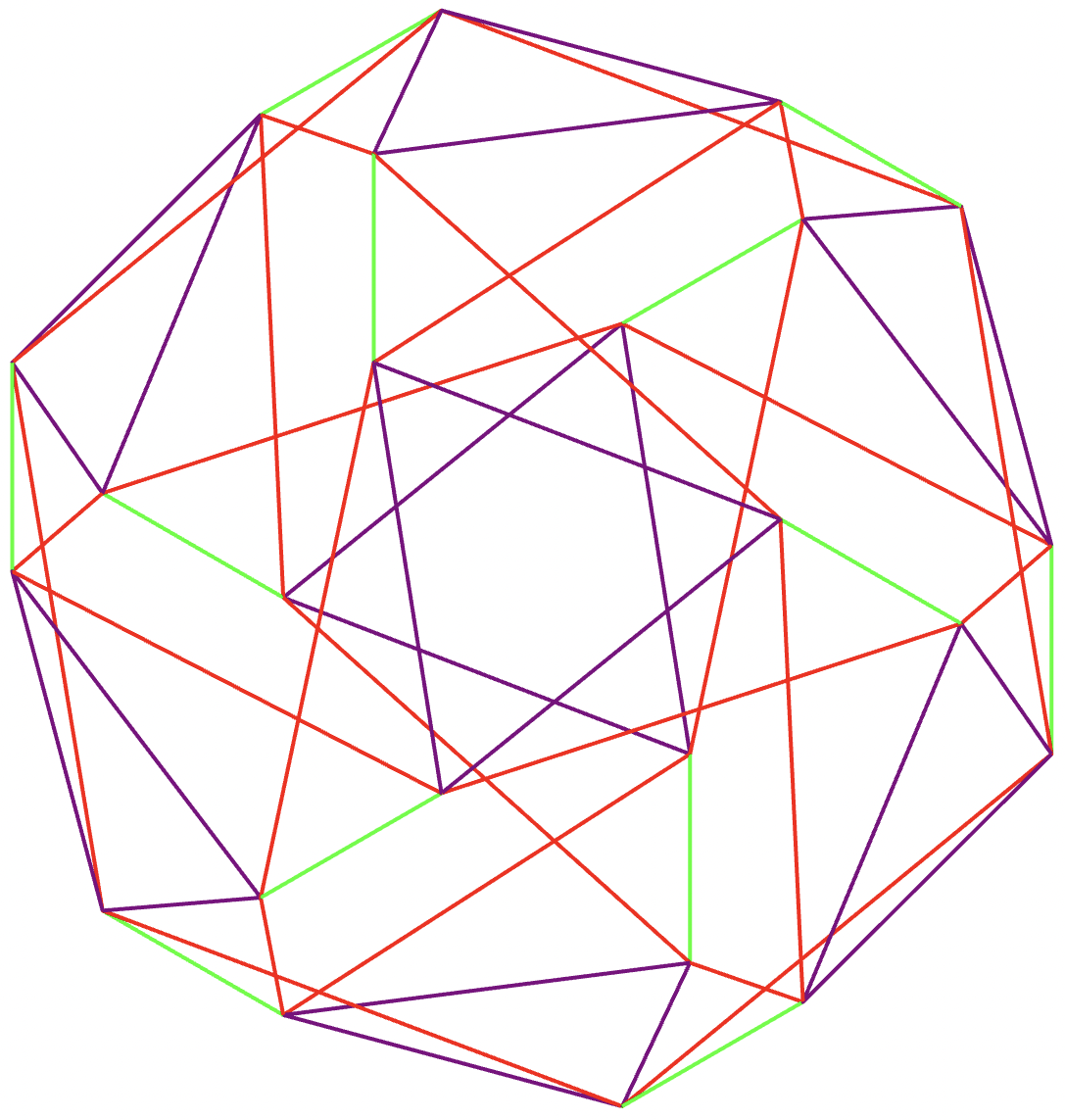}
    & $S_{\{6,3\}_4}(v)$ with $v\in D$ \\
    & $6_s.3.3.3.3$ or $6_c.3.3.3.3$ \\
    & $(f_0,f_1^0, f_1^1, f_1^2, f_2^0, f_2^1, f_2^2) = (24,12,24,24,24,4,8)$ \\
    & $\chi(S_P) = 0$ \\ 
\end{tblr}

\end{centering}
\vspace{0.5em}

Since $\{6,3\}_4$ is orientable, $S_{\{6,3\}_4}(v)$ can be cellularly embedded on the same surface as $\{6,3\}_4$, namely the torus. As our later considerations will show, this is the only finite Petrie-dual of a classical regular polyhedron that does not have faces sharing centers in pairs.
\bigskip

\vspace{0.5em}
\begin{centering}

\begin{tblr}{cells={valign=m,halign=c},row{1}={rowsep=4pt},row{2}={rowsep=4pt},row{3}={rowsep=4pt},row{4}={rowsep=4pt},hlines,vlines,column{1}={3.25cm}}
\SetCell[r=4]{} \includegraphics[height=2.8cm,valign=c]{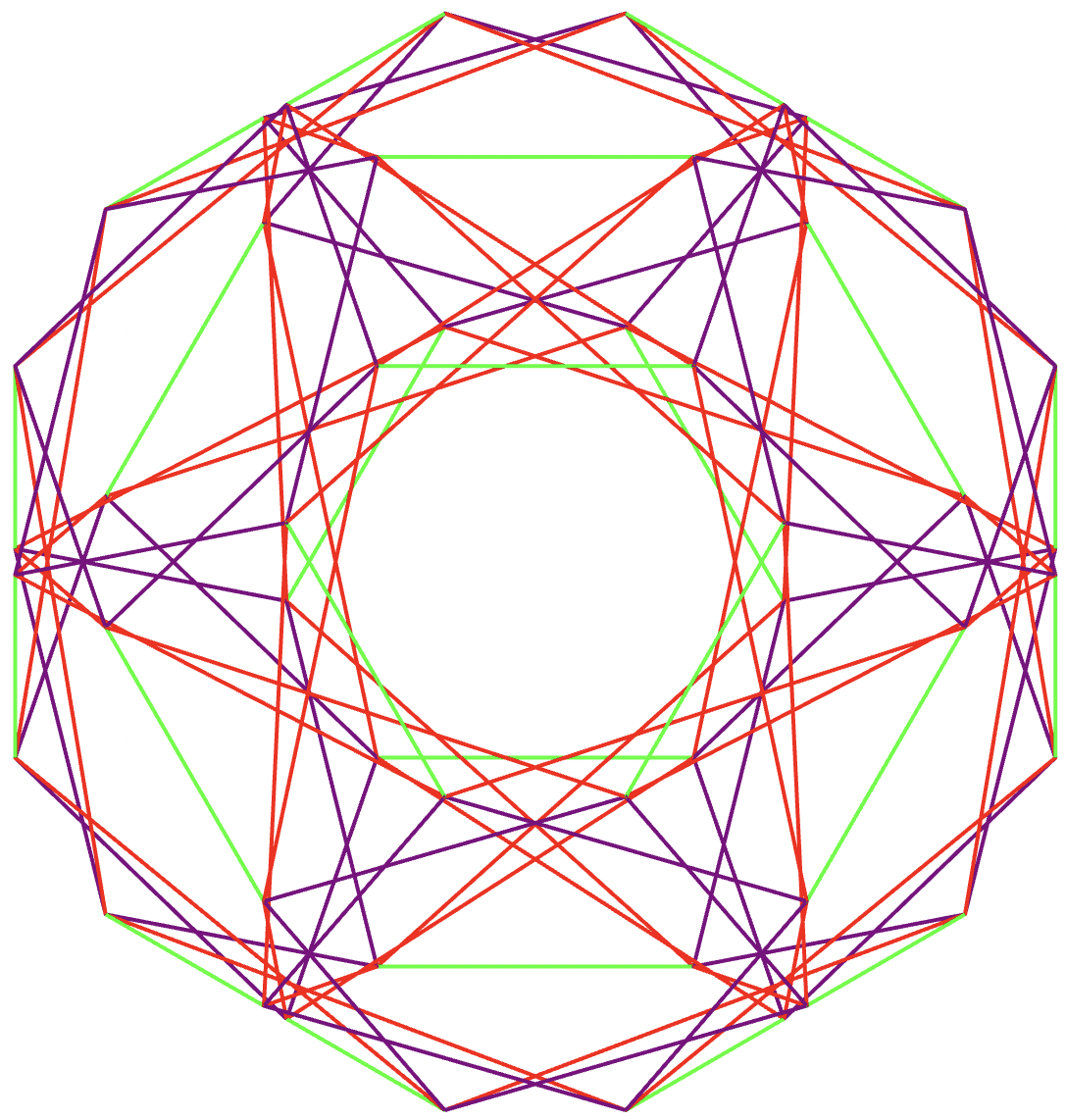}
    & $S_{\{6,4\}_3}(v)$ with $v\in D$ \\
    & $6_s.3.3.4_c.3$ or $6_c.3.3.4_c.3$ \\
    & $(f_0,f_1^0, f_1^1, f_1^2, f_2^0, f_2^1, f_2^2) = (48,24,48,48,48,8,12)$ \\
    & $\chi(S_P) = -4$ \\  
\end{tblr}

\end{centering}
\vspace{0.5em}

Note that $\{6,4\}_3$ is a regular map on a non-orientable surface of genus~4 whose two-fold cover is a regular map $M$ of type $\{6,4\}$ on an orientable surface of genus~3 with Petrie polygons of length 6, denoted R3.4' in \cite{Conder}. Thus, by the proof of Theorem~\ref{doublecover}, $S_{\{6,4\}_3}(v) \cong S_M(w)$ where $w$ lies $\Delta$. Recall that $\Delta$ is the interior of the base triangle of the barycentric subdivision $B(M)$ of the map $M$ (see Section \ref{section: topology}).
\bigskip

\vspace{0.5em}
\begin{centering}

\begin{tblr}{cells={valign=m,halign=c},row{1}={rowsep=4pt},row{2}={rowsep=4pt},row{3}={rowsep=4pt},row{4}={rowsep=4pt},hlines,vlines,column{1}={3.25cm}}
\SetCell[r=4]{} \includegraphics[height=2.8cm,valign=c]{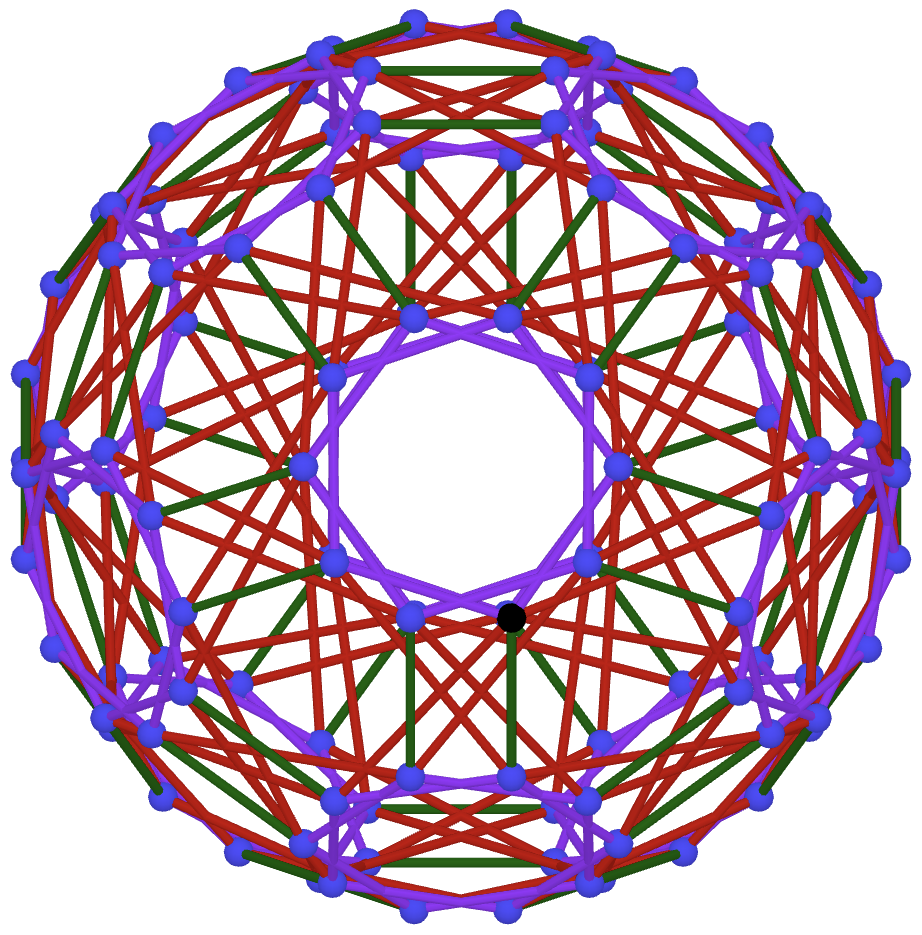}
    & $S_{\{10,5\}_3}(v)$ with $v\in D$ \\
    & $10_s.3.3.5_c.3$ or $10_c.3.3.5_c.3$ \\
    & $(f_0,f_1^0, f_1^1, f_1^2, f_2^0, f_2^1, f_2^2) = (120,60,120,120,120,12,24)$ \\
    & $\chi(S_P) = -24$ \\  
\end{tblr}
    
\end{centering}
\vspace{0.5em}

Note that $\{10,5\}_3$ (N14.3') is a regular map on a non-orientable surface of genus 14 whose two-fold cover is a regular map $M$ of type $\{10,5\}$ on an orientable surface of genus 13 with Petrie polygons of length 6 (R13.8'). Thus, $S_{\{10,5\}_3}(v) \cong S_{M}(w)$ where $w$ lies in $\Delta$.
\bigskip

\vspace{0.5em}
\begin{centering}

\begin{tblr}{cells={valign=m,halign=c},row{1}={rowsep=4pt},row{2}={rowsep=4pt},row{3}={rowsep=4pt},row{4}={rowsep=4pt},hlines,vlines,column{1}={3.25cm}}
\SetCell[r=4]{} \includegraphics[height=2.8cm,valign=c]{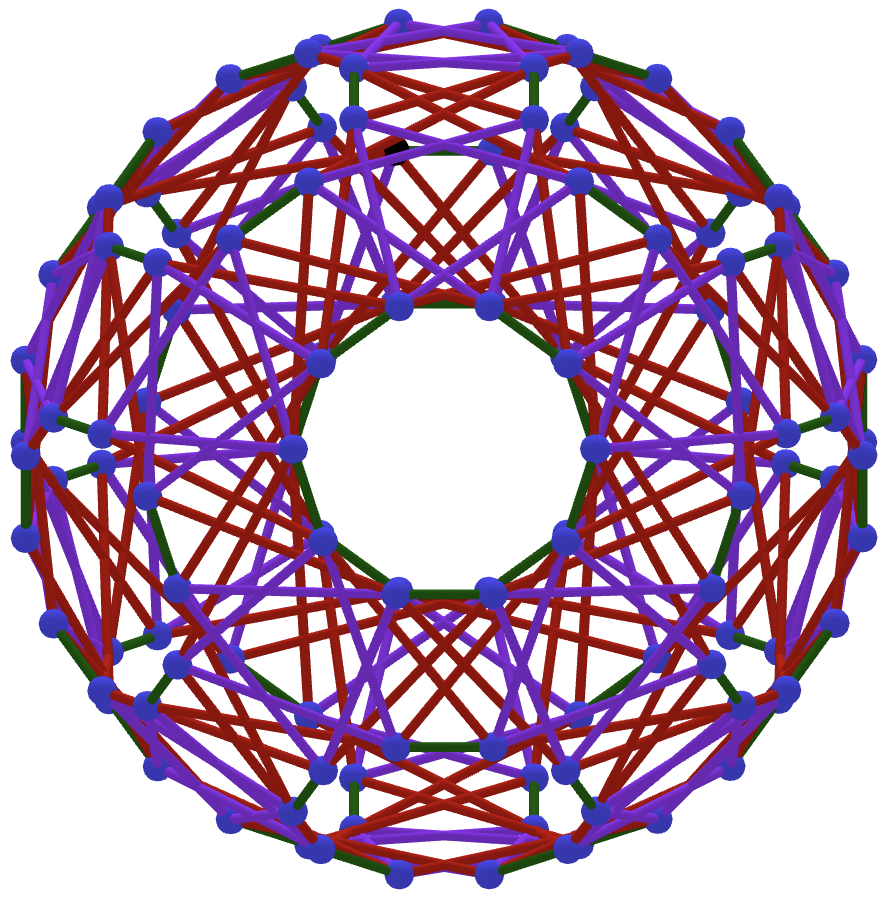}
    & $S_{\{10,3\}_5}(v)$ with $v\in D$ \\
    & $10_s.3.3.3.3$ or $10_c.3.3.3.3$ \\
    & $(f_0,f_1^0, f_1^1, f_1^2, f_2^0, f_2^1, f_2^2) = (120,60,120,120,120,12,40)$ \\
    & $\chi(S_P) = -8$ \\ 
\end{tblr}
    
\end{centering}
\vspace{0.5em}

Note that $\{10,3\}_5$ (N6.2') is a regular map on a non-orientable surface of genus 6 whose two-fold cover is a regular map $M$ of type $\{10,3\}$ on an orientable surface of genus 5 with Petrie polygons of length 10 (R5.2'). Thus, $S_{\{10,3\}_5}(v) \cong S_{M}(w)$ where $w$ lies in $\Delta$.
\bigskip

\vspace{0.5em}
\begin{centering}

\begin{tblr}{cells={valign=m,halign=c},row{1}={rowsep=4pt},row{2}={rowsep=4pt},row{3}={rowsep=4pt},row{4}={rowsep=4pt},hlines,vlines,column{1}={3.25cm}}
\SetCell[r=4]{} \includegraphics[height=2.8cm,valign=c]{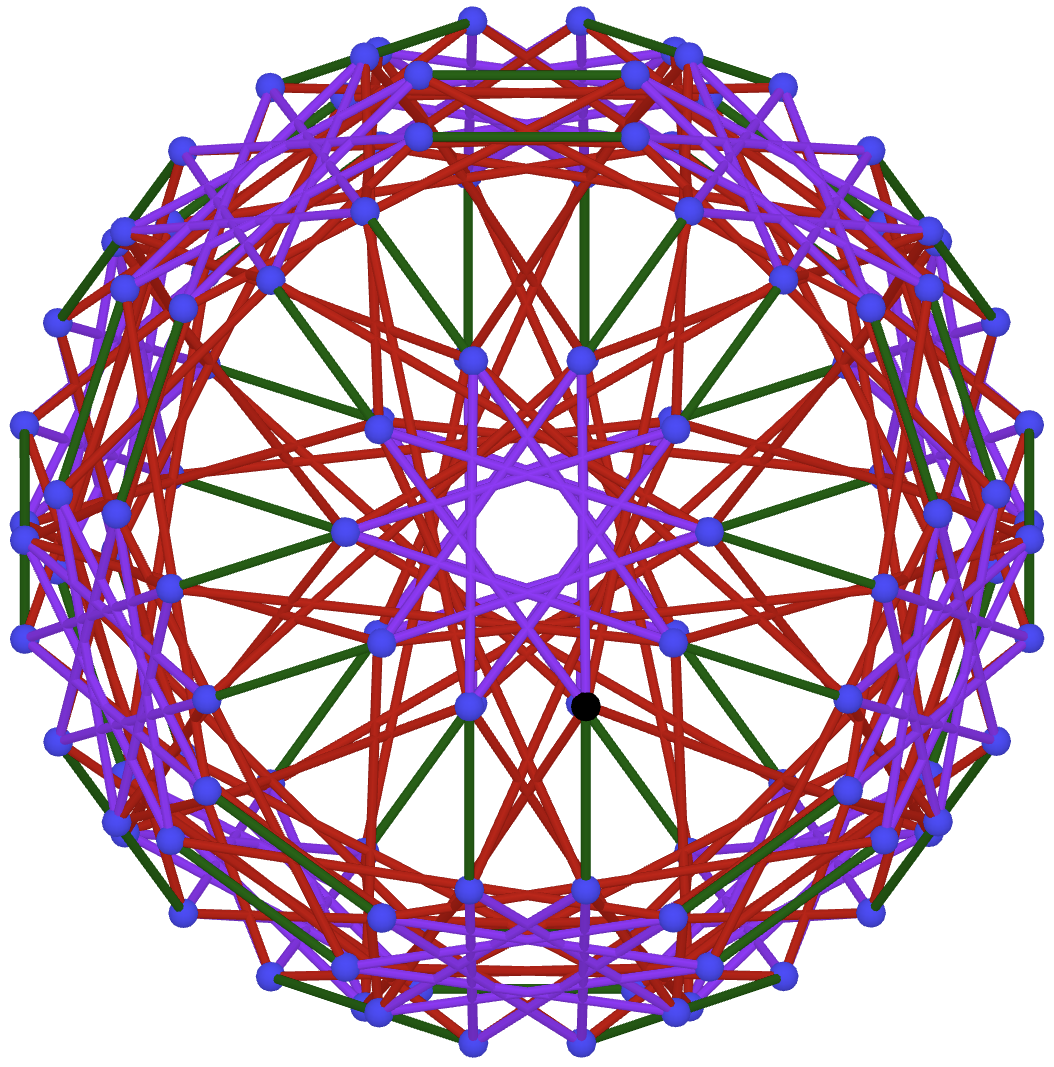}
    & $S_{\{6,\frac{5}{2}\}}(v)$ with $v\in D$ \\
    & $6_s.3.3.\frac{5}{2}.3$ or $6_c.3.3.\frac{5}{2}.3$ \\
    & $(f_0,f_1^0, f_1^1, f_1^2, f_2^0, f_2^1, f_2^2) = (120,60,120,120,120,20,24)$ \\
    & $\chi(S_P) = -16$ \\ 
\end{tblr}
    
\end{centering}
\vspace{0.5em}

Note that $\{6,\frac{5}{2}\}$ gives a regular map on a non-orientable surface of genus~10 (N10.5') whose two-fold cover is a regular map $M$ of type $\{6,5\}$ of genus 9 with Petrie polygons of length 10 (R9.15'). Thus, $S_{\{6,\frac{5}{2}\}}(v) \cong S_{M}(w)$ where $w$ lies in $\Delta$.
\bigskip

\vspace{0.5em}
\begin{centering}

\begin{tblr}{cells={valign=m,halign=c},row{1}={rowsep=4pt},row{2}={rowsep=4pt},row{3}={rowsep=4pt},row{4}={rowsep=4pt},hlines,vlines,column{1}={3.25cm}}
\SetCell[r=4]{} \includegraphics[height=2.8cm,valign=c]{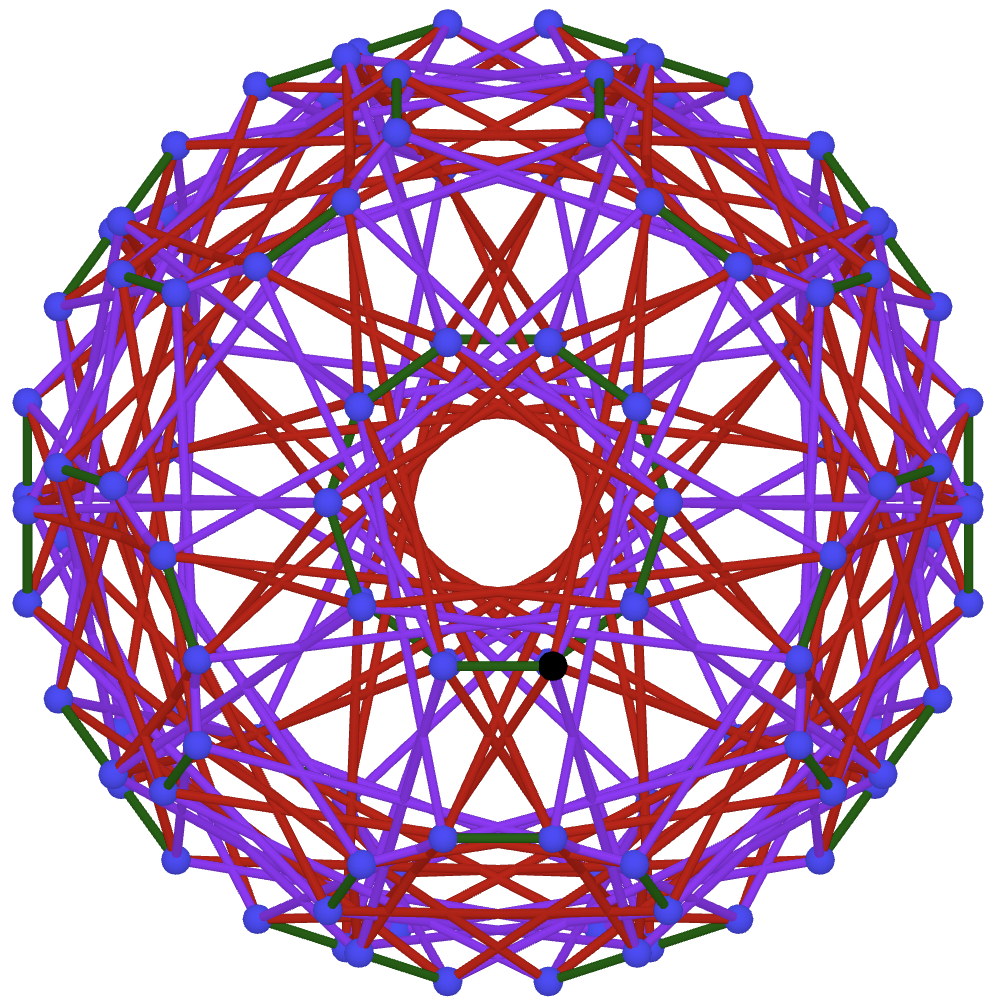}
    & $S_{\{6,5\}}(v)$ with $v\in D$ \\
    & $6_s.3.3.5_c.3$ or $6_c.3.3.5_c.3$ \\
    & $(f_0,f_1^0, f_1^1, f_1^2, f_2^0, f_2^1, f_2^2) = (120,60,120,120,120,20,24)$ \\
    & $\chi(S_P) = -16$ \\  
\end{tblr}
    
\end{centering}
\vspace{0.5em}

Combinatorially, the polyhedron $\{6,5\}$ is $\{6,5\}_{5,3}$ (N10.5'), a regular map on a non-orientable surface of genus 10 whose two-fold cover is a regular map $M$ of type $\{6,5\}$ on an orientable surface of genus 9 with Petrie polygons of length 10 (R9.15'). Again, $S_P(v) \cong S_{M}(w)$ where $w$ lies in $\Delta$. It is worth noting that $S_{\{6,5\}} \cong S_{\{6,\frac{5}{2}\}}$.
\bigskip

\vspace{0.5em}
\begin{centering}

\begin{tblr}{cells={valign=m,halign=c},row{1}={rowsep=4pt},row{2}={rowsep=4pt},row{3}={rowsep=4pt},row{4}={rowsep=4pt},hlines,vlines,column{1}={3.25cm}}
\SetCell[r=4]{} \includegraphics[height=2.8cm,valign=c]{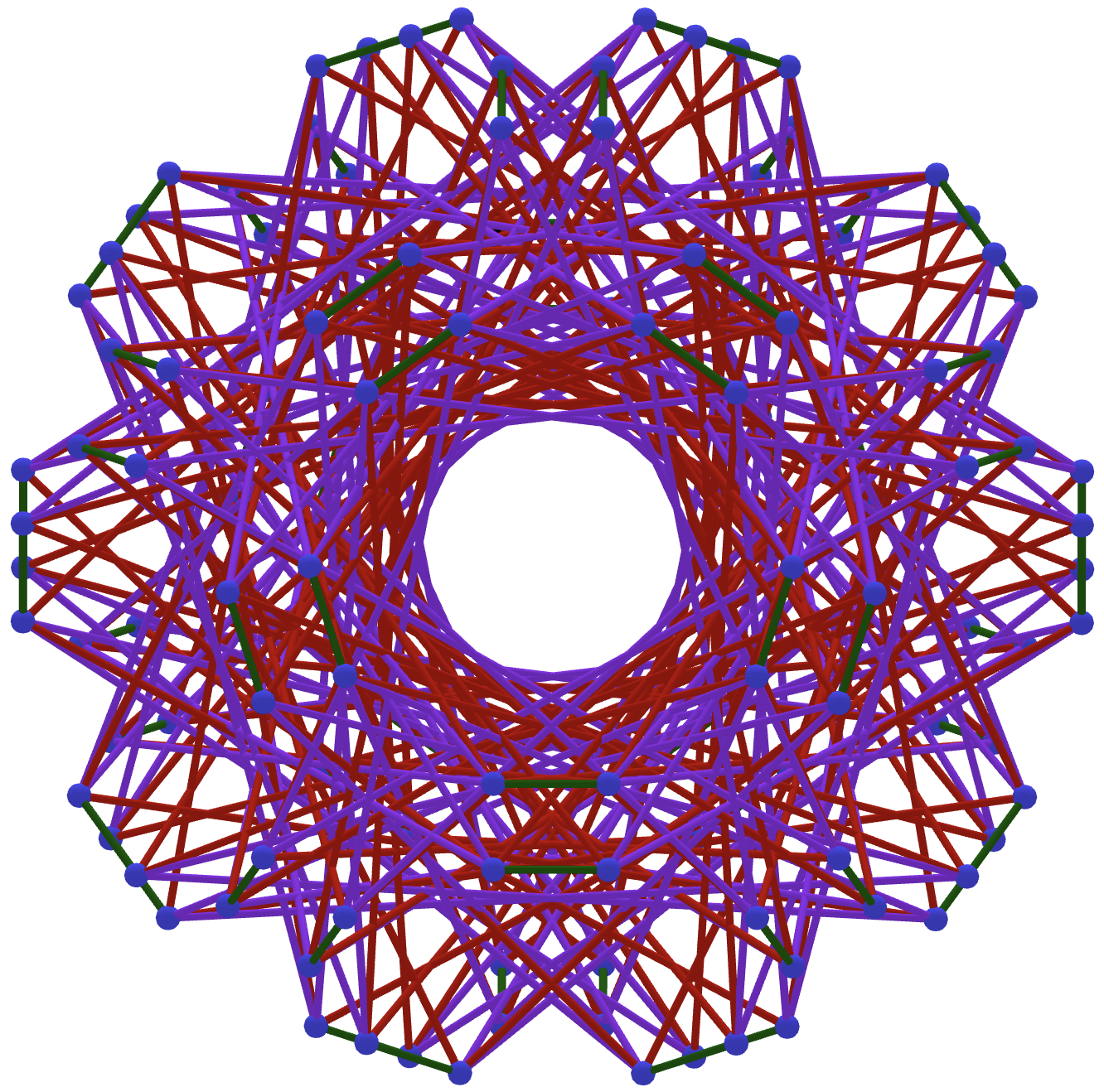}
    & $S_{\{\frac{10}{3},\frac{5}{2}\}}(v)$ with $v\in D$ \\
    & $(\frac{10}{3})_s.3.3.\frac{5}{2}.3$ or $\frac{10}{3}.3.3.\frac{5}{2}.3$ \\
    & $(f_0,f_1^0, f_1^1, f_1^2, f_2^0, f_2^1, f_2^2) = (120,60,120,120,120,12,24)$ \\
    & $\chi(S_P) = -24$ \\  
\end{tblr}
    
\end{centering}
\vspace{0.5em}

Combinatorially, $\{\frac{10}{3},\frac{5}{2}\}$ is $\{10,5\}_3$ (N14.3'), a non-orientable regular map of genus 14 whose two-fold cover is an orientable regular map $M$ of type $\{10,5\}$ of genus 13 with Petrie polygons of length 6 (R13.8'). We have $S_P(v) \cong S_{M}(w)$ where $w$ lies in $\Delta$.
\bigskip

\vspace{0.5em}
\begin{centering}

\begin{tblr}{cells={valign=m,halign=c},row{1}={rowsep=4pt},row{2}={rowsep=4pt},row{3}={rowsep=4pt},row{4}={rowsep=4pt},hlines,vlines,column{1}={3.25cm}}
\SetCell[r=4]{} \includegraphics[height=2.8cm,valign=c]{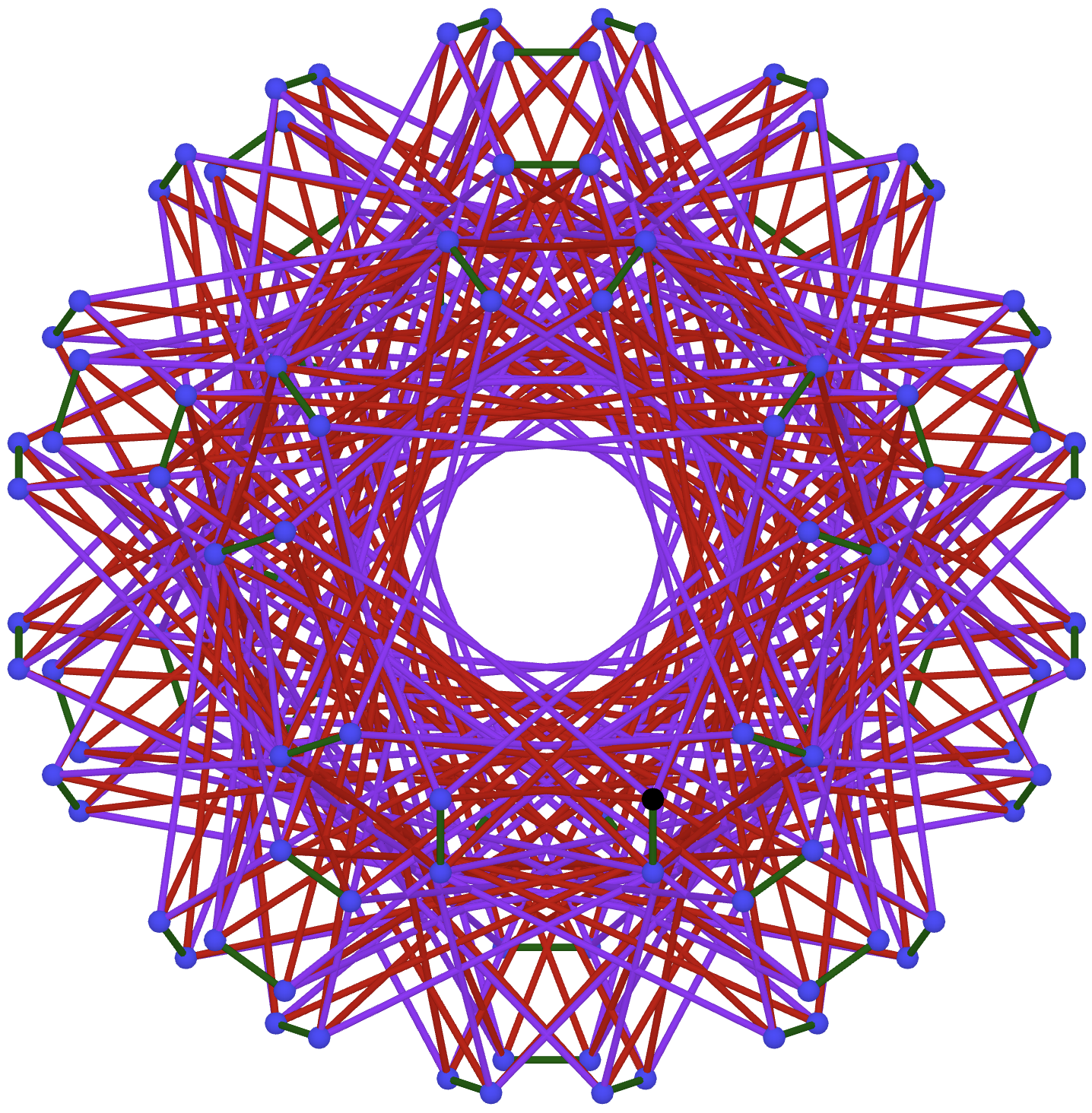}
    & $S_{\{\frac{10}{3},3\}}(v)$ with $v\in D$ \\
    & $(\frac{10}{3})_s.3.3.3.3$ or $\frac{10}{3}.3.3.3.3$ \\
    & $(f_0,f_1^0, f_1^1, f_1^2, f_2^0, f_2^1, f_2^2) = (120,60,120,120,120,12,40)$ \\
    & $\chi(S_P) = -8$ \\  
\end{tblr}
    
\end{centering}
\vspace{0.5em}

Finally, $\{\frac{10}{3},3\}$ is isomorphic to $\{10,3\}_5$ (N6.2'), a non-orientable regular map of genus 6 whose two-fold cover is an orientable regular map $M$ of type $\{10,3\}$ of genus 10 with Petrie polygons of length 5 (R5.2'). We have $S_P(v) \cong S_{M}(w)$ where $w$ lies in $\Delta$.

\section{New finite uniform polyhedra}
\label{new finite uniform polyhedra}

In this section we present new uniform polyhedra obtained by applying the snub construction to the regular finite polyhedra. 
As we saw in Section~\ref{finite uniform polyhedra}, there are no new uniform polyhedra arising as genuine snubs (that is with $v$ satisfying the IPC of (3)) from these regular polyhedra. However, there are many new uniform degenerate snubs of the form $S_P^0(v)$. For computations we again refer to Tables~\ref{gplusfinite1},~\ref{gplusfinite2},~\ref{gplusfunregfinite}.

We consider the classes of finite uniform polyhedra of the form $S_P^0(v)$ obtained when $v$ is chosen such that $s_0(v)=v$. Note that we say classes here because the fixed point set of each $s_0$ is a plane. Thus, by varying $v$ on this plane we can obtain infinitely many polyhedra which are combinatorially equivalent, but not similar (geometrically equivalent). In each of the tables below we have chosen a representative for each class of new uniform polyhedra; the vertex figures are crossed quadrilaterals, either bowties or butterflies as shown in Figures~\ref{fig:bow tie} and~\ref{fig:butterfly}. 

\begin{table}[H]
\centering

\begin{tblr}{cells={valign=m,halign=c},row{1}={rowsep=4pt},row{2}={rowsep=4pt},row{3}={rowsep=4pt},row{4}={rowsep=4pt},hlines,vlines,column{1}={3.25cm}}
\SetCell[r=4]{} \includegraphics[height=2.8cm,valign=c]{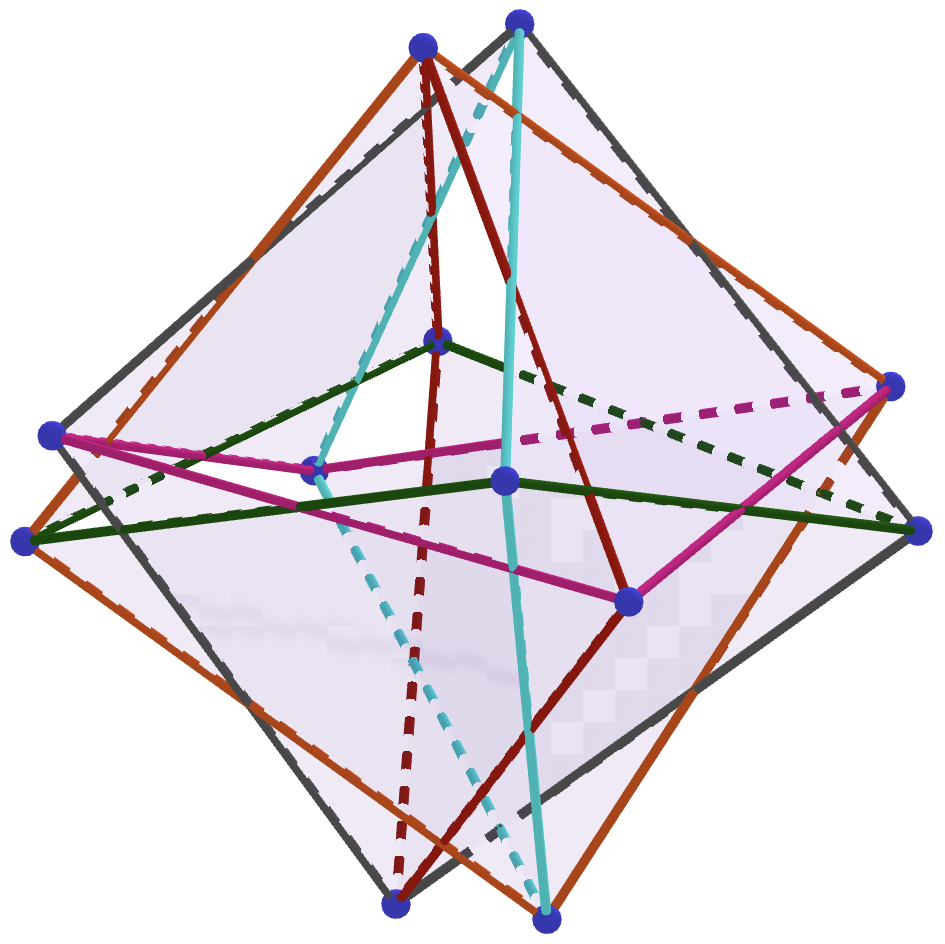}
    & $\mathcal{S}_{\{4,3\}_3}(\frac{1}{2},\frac{3}{10},\frac{\sqrt{2}}{10})$ \\
    & $4_s.3.4_s.3$ \\
    & Crossed quadrilateral (butterfly) \\
    & $(f_0,f_1, f_2^1, f_2^2) = (12,24,8,6)$ \\ 
\end{tblr}

\vspace{0.5em}

\begin{tblr}{cells={valign=m,halign=c},row{1}={rowsep=4pt},row{2}={rowsep=4pt},row{3}={rowsep=4pt},row{4}={rowsep=4pt},hlines,vlines,column{1}={3.25cm}}
\SetCell[r=4]{} \includegraphics[height=2.8cm,valign=c]{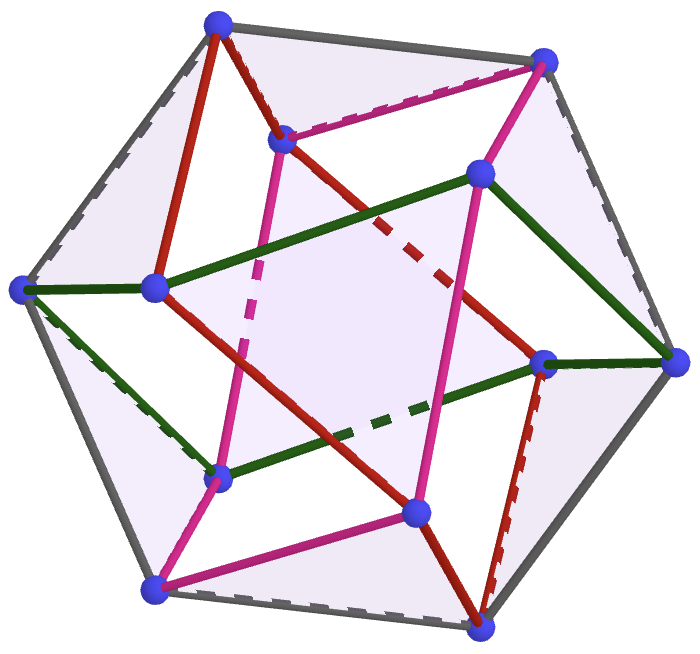}
    & $\mathcal{S}_{\{6,3\}_4}(\frac{1}{2},0,\frac{\sqrt{2}}{10})$ \\
    & $6_s.3.6_s.3$ \\
    & Crossed quadrilateral (butterfly) \\
    & $(f_0,f_1, f_2^1, f_2^2) = (12,24,4,8)$ \\ 
\end{tblr}

\vspace{0.5em}

\begin{tblr}{cells={valign=m,halign=c},row{1}={rowsep=4pt},row{2}={rowsep=4pt},row{3}={rowsep=4pt},row{4}={rowsep=4pt},hlines,vlines,column{1}={3.25cm}}
\SetCell[r=4]{} \includegraphics[height=2.8cm,valign=c]{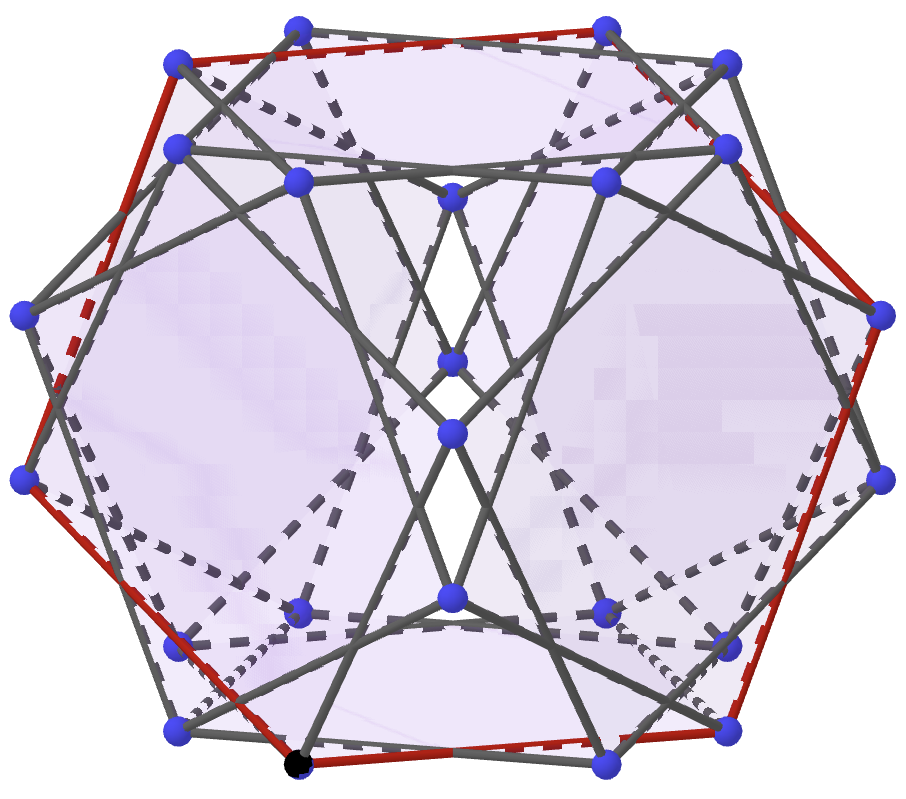}
    & $\mathcal{S}_{\{6,4\}_3}(\frac{1}{2},\frac{1}{10},-\frac{1}{2})$ \\
    & $6_s.4_c.6_s.4_c$ \\
    & Crossed quadrilateral (butterfly) \\
    & $(f_0,f_1, f_2^1, f_2^2) = (24,48,8,12)$ \\ 
\end{tblr}

\vspace{0.5em}

\begin{tblr}{cells={valign=m,halign=c},row{1}={rowsep=4pt},row{2}={rowsep=4pt},row{3}={rowsep=4pt},row{4}={rowsep=4pt},hlines,vlines,column{1}={3.25cm}}
\SetCell[r=4]{} \includegraphics[height=2.8cm,valign=c]{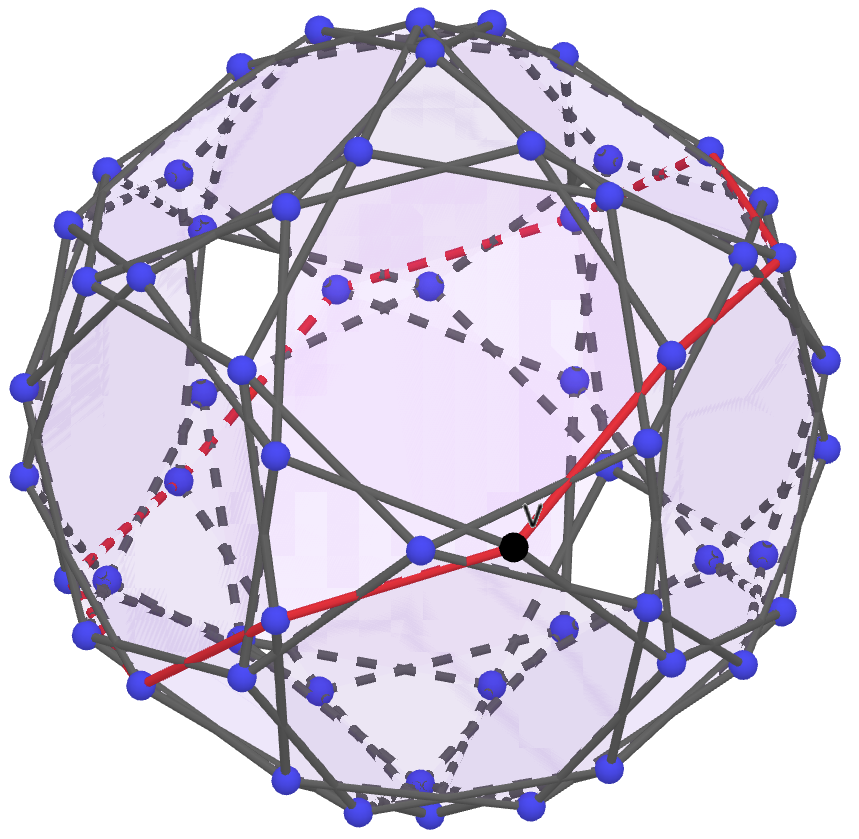}
    & $\mathcal{S}_{\{10,5\}}(\frac{1+\sqrt{5}}{2},\frac{3}{10},0)$ \\
    & $10_s.5_c.10_s.5_c$ \\
    & Crossed quadrilateral (butterfly) \\
    & $(f_0,f_1, f_2^1, f_2^2) = (60,120,12,24)$ \\ 
\end{tblr}

\end{table}

\begin{table}[H]
\centering

\begin{tblr}{cells={valign=m,halign=c},row{1}={rowsep=4pt},row{2}={rowsep=4pt},row{3}={rowsep=4pt},row{4}={rowsep=4pt},hlines,vlines,column{1}={3.25cm}}
\SetCell[r=4]{} \includegraphics[height=2.8cm,valign=c]{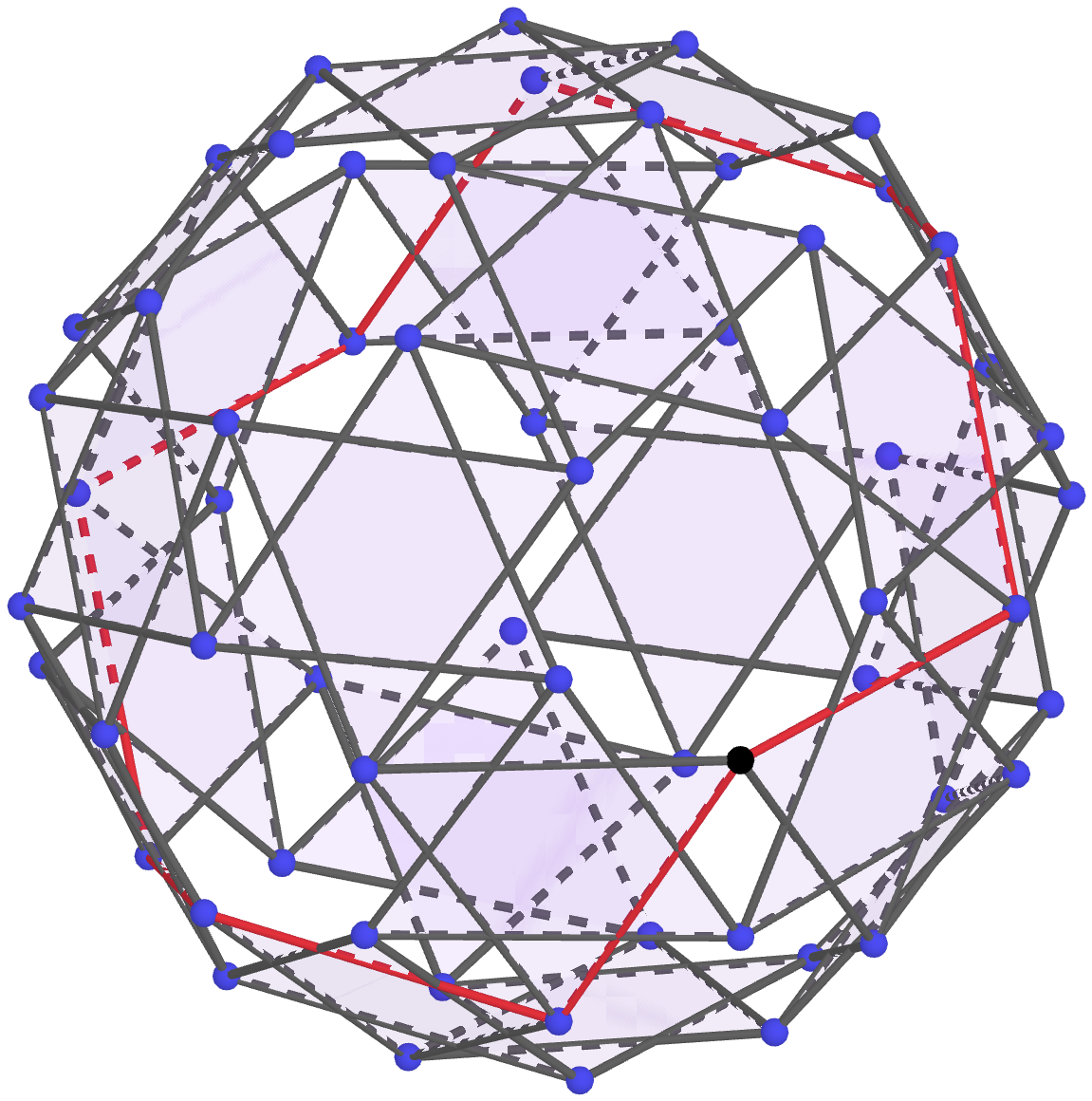}
    & $\mathcal{S}_{\{10,3\}}(1,0,\frac{1}{10})$ \\
    & $10_s.3.10_s.3$ \\
    & Crossed quadrilateral (butterfly) \\
    & $(f_0,f_1, f_2^1, f_2^2) = (60,120,12,40)$ \\ 
\end{tblr}

\vspace{0.5em}

\begin{tblr}{cells={valign=m,halign=c},row{1}={rowsep=4pt},row{2}={rowsep=4pt},row{3}={rowsep=4pt},row{4}={rowsep=4pt},hlines,vlines,column{1}={3.25cm}}
\SetCell[r=4]{} \includegraphics[height=2.8cm,valign=c]{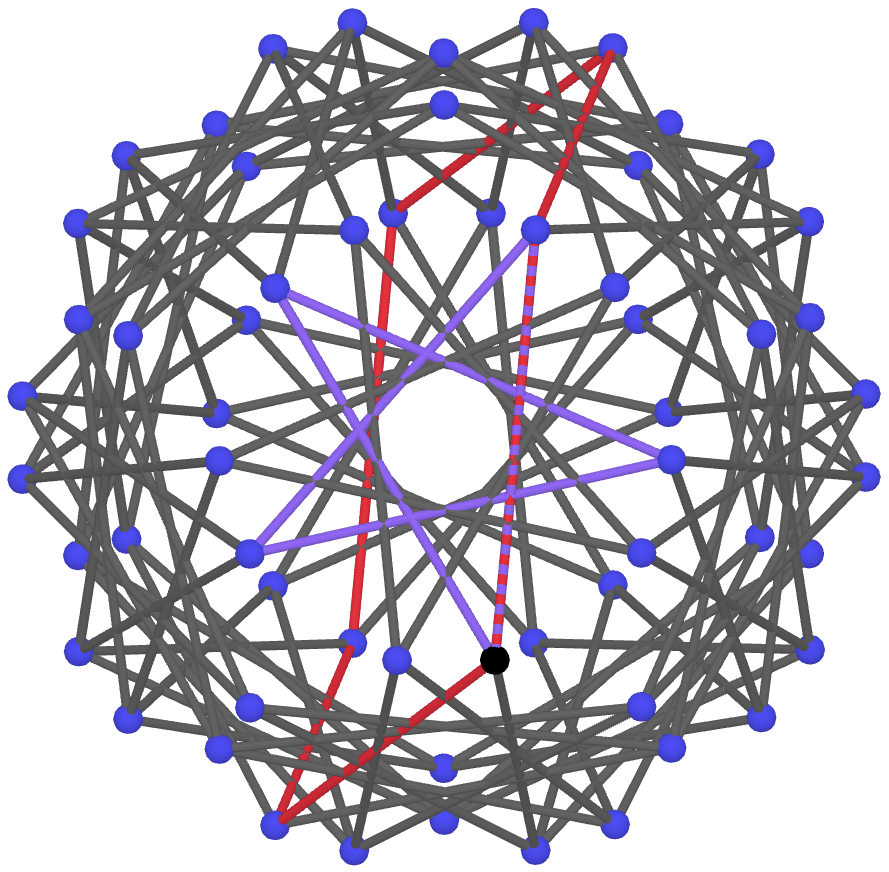}
    & $\mathcal{S}_{\{6,\frac{5}{2}\}}(1,\frac{1}{10},0)$ \\
    & $6_s.\frac{5}{2}.6_s.\frac{5}{2}$ \\
    & Crossed quadrilateral (butterfly) \\
    & $(f_0,f_1, f_2^1, f_2^2) = (60,120,20,24)$ \\ 
\end{tblr}

\vspace{0.5em}

\begin{tblr}{cells={valign=m,halign=c},row{1}={rowsep=4pt},row{2}={rowsep=4pt},row{3}={rowsep=4pt},row{4}={rowsep=4pt},hlines,vlines,column{1}={3.25cm}}
\SetCell[r=4]{} \includegraphics[height=2.8cm,valign=c]{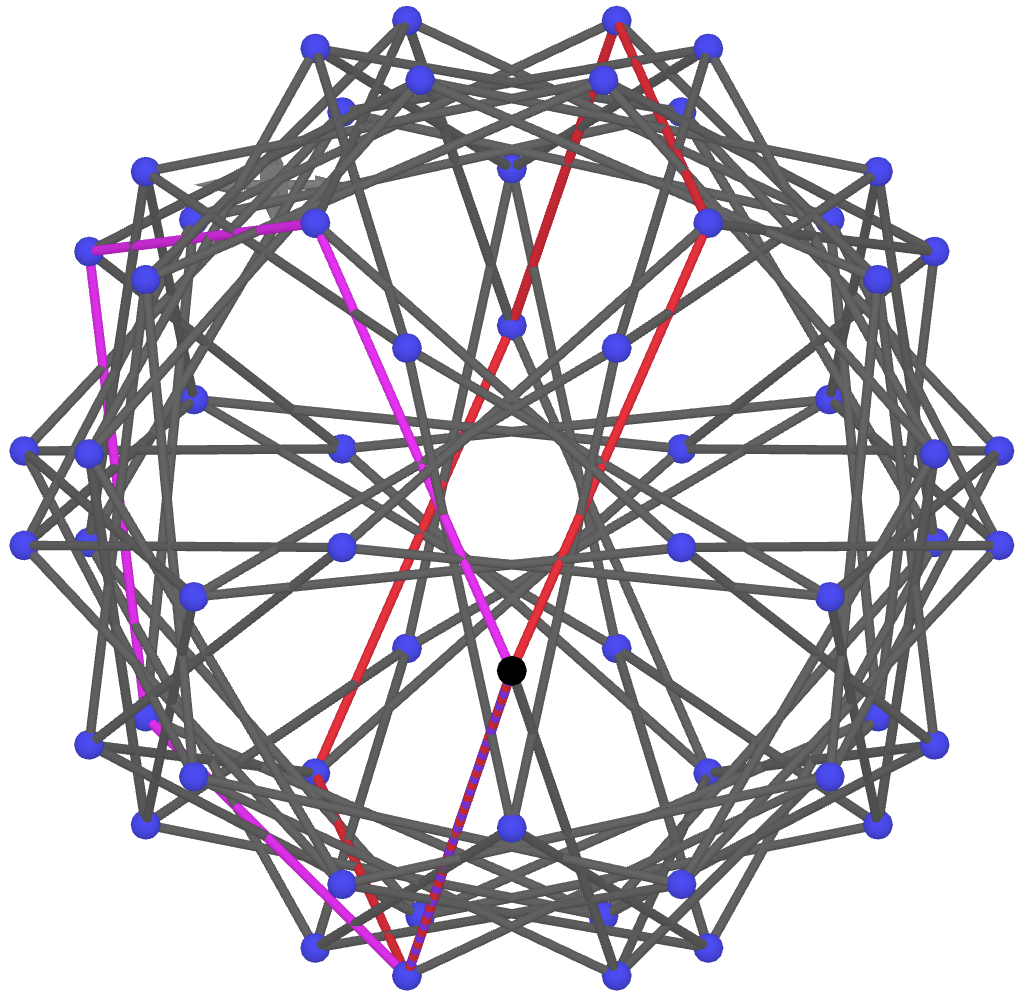}
    & $\mathcal{S}_{\{6,5\}}(1,0,\frac{1}{10})$ \\
    & $6_s.5_c.6_s.5_c$ \\
    & Crossed quadrilateral (butterfly) \\
    & $(f_0,f_1, f_2^1, f_2^2) = (60,120,20,24)$ \\ 
\end{tblr}

\vspace{0.5em}

\begin{tblr}{cells={valign=m,halign=c},row{1}={rowsep=4pt},row{2}={rowsep=4pt},row{3}={rowsep=4pt},row{4}={rowsep=4pt},hlines,vlines,column{1}={3.25cm}}
\SetCell[r=4]{} \includegraphics[height=2.8cm,valign=c]{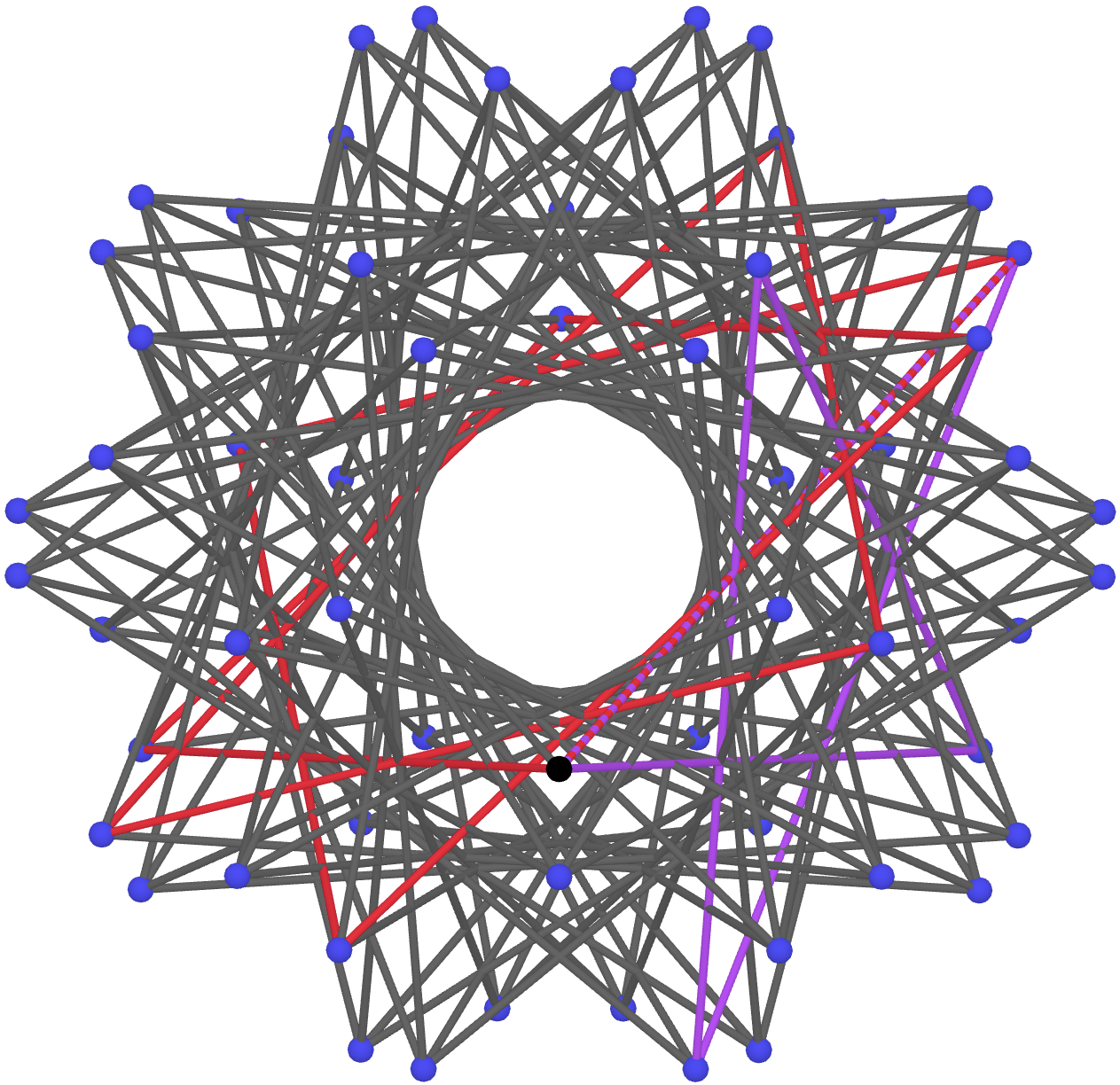}
    & $\mathcal{S}_{\{\frac{10}{3},\frac{5}{2}\}}(1,0,\frac{1}{10})$ \\
    & $(\frac{10}{3})_s.\frac{5}{2}.(\frac{10}{3})_s.\frac{5}{2}$ \\
    & Crossed quadrilateral (butterfly) \\
    & $(f_0,f_1, f_2^1, f_2^2) = (60,120,12,24)$ \\ 
\end{tblr}

\vspace{0.5em}

\begin{tblr}{cells={valign=m,halign=c},row{1}={rowsep=4pt},row{2}={rowsep=4pt},row{3}={rowsep=4pt},row{4}={rowsep=4pt},hlines,vlines,column{1}={3.25cm}}
\SetCell[r=4]{} \includegraphics[height=2.8cm,valign=c]{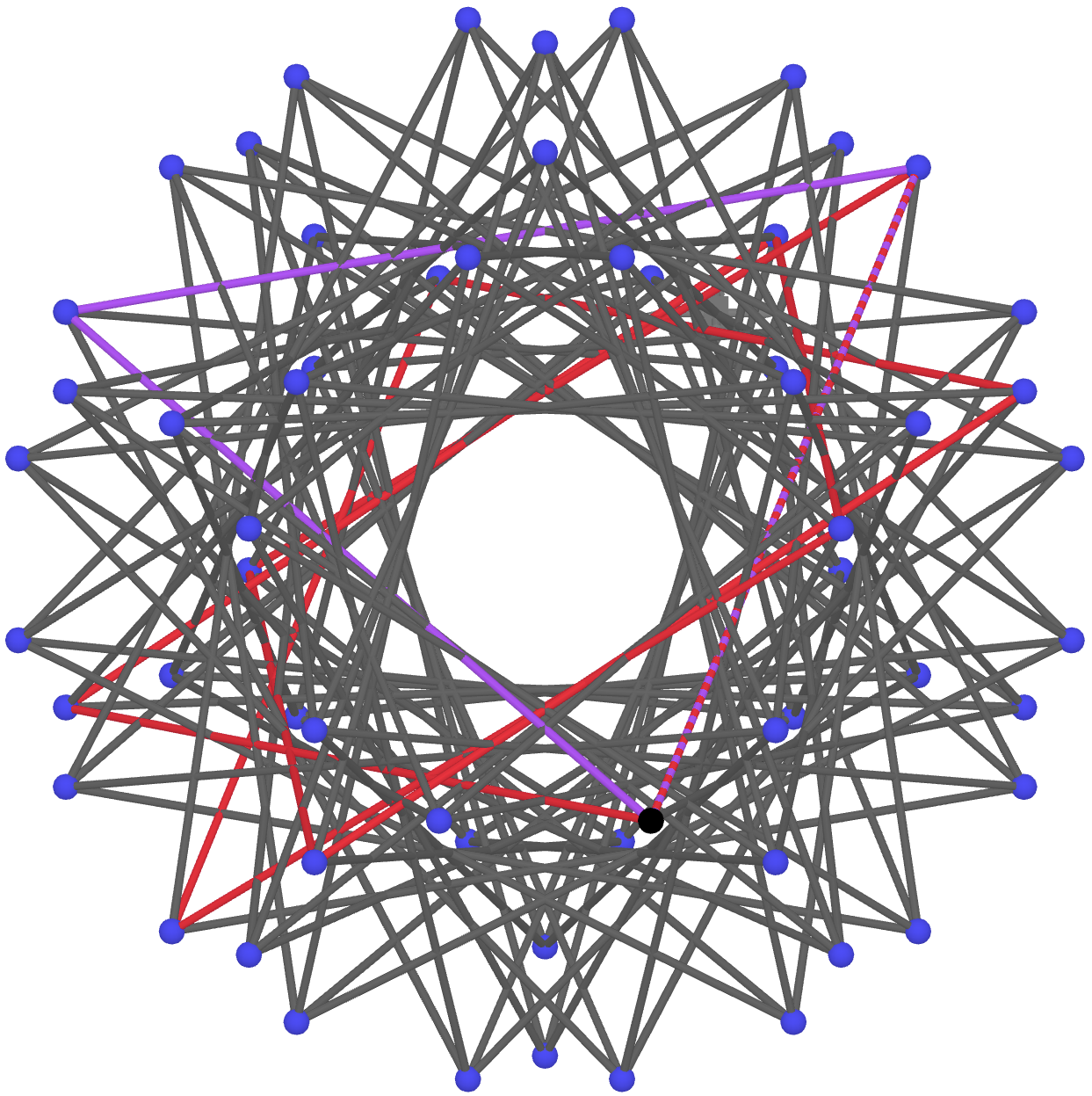}
    & $\mathcal{S}_{\{\frac{10}{3},3\}}(1,\frac{1}{10},0)$ \\
    & $(\frac{10}{3})_s.3.(\frac{10}{3})_s.3$ \\
    & Crossed quadrilateral (butterfly) \\
    & $(f_0,f_1, f_2^1, f_2^2) = (60,120,12,40)$ \\ 
\end{tblr}

\end{table}

\begin{figure}[H]
\centering
\begin{minipage}{.5\textwidth}
  \centering
  \includegraphics[height=3cm]{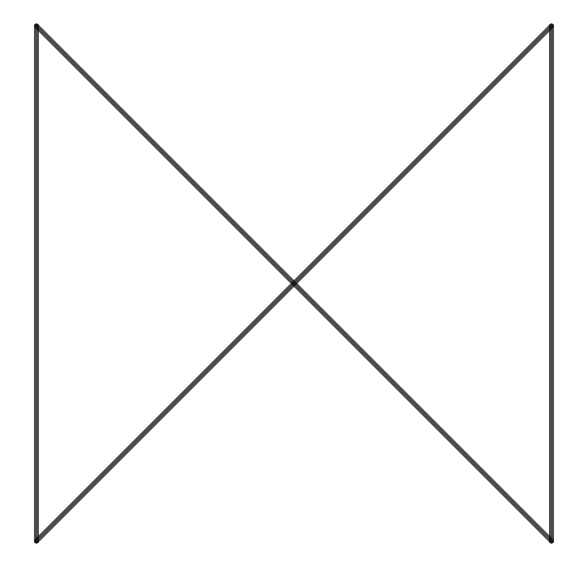}
  \captionof{figure}{Crossed quadrilateral (bow tie)}
  \label{fig:bow tie}
\end{minipage}%
\begin{minipage}{.5\textwidth}
  \centering
  \includegraphics[height=3cm]{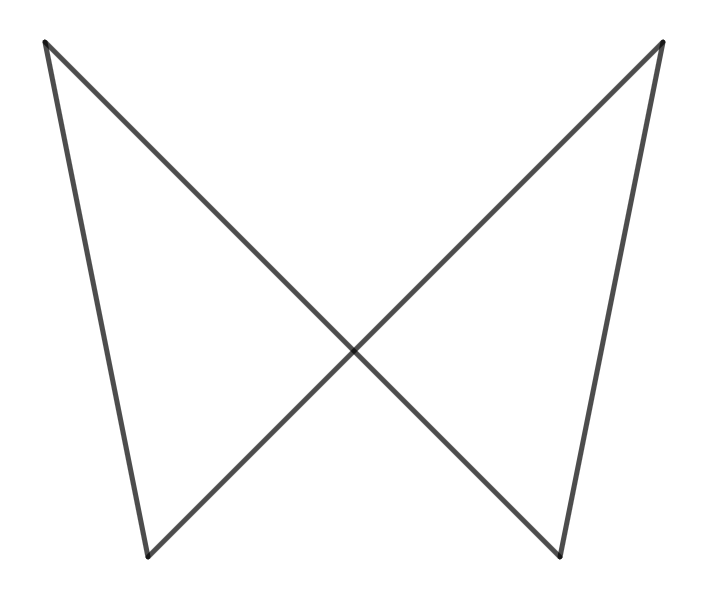}
  \captionof{figure}{Crossed quadrilateral (butterfly)}
  \label{fig:butterfly}
\end{minipage}
\end{figure}

\subsection*{Acknowledgment}
We are grateful to Gareth Jones and Marston Conder for helpful discussions with regards to the existence of two-fold smooth orientable regular coverings of non-orientable regular maps, and to the former for alerting us to his article~\cite{Jones}.

\subsection*{Data Availability Statement}
No new data were created or analyzed in this study. Data sharing is
not applicable to this article.

\end{document}